\documentclass[11pt, english]{article}

\usepackage[english]{babel}    % include this always
\usepackage{amsmath}           % various ams stuff
\usepackage{amssymb}           % dito
\usepackage[utf8]{inputenc}    % smart input of funny chars
\usepackage[T1]{fontenc}       % also for the font encoding
\usepackage{eucal}             % nicer calligraphic fonts
\usepackage{bbm}               % nicer bbm fonts
\usepackage{longtable}         % tables longer than one page
\usepackage{exscale}           % large summation signs in 11pt
\usepackage{mathtools}         % some nice tricks for stackrel
\usepackage[amsmath,thmmarks,hyperref]{ntheorem}           % nicer theorems
\usepackage{graphicx}          % to include eps pictures
\usepackage[sort]{cite}        % nicer citations
\usepackage{mathrsfs}          % nice calligraphic font
\usepackage{array}             % nice tables
\usepackage{stmaryrd}          % more brackets
\usepackage{paralist}          % better enumerate lists
\usepackage{wasysym}           % smiley symbols
\usepackage[a4paper]{geometry} % geometry of page layout
\usepackage{color}             % colors in the pictures
\usepackage{gitinfo}           % include git info
\usepackage{xspace}            % better spacing after macros
\usepackage[expansion=false]{microtype}          % protrusion
\usepackage[backref=page,final=true]{hyperref}   % ref 'n' click, also for the fix me stuff

\renewcommand{\mathbb}[1]{\mathbbm{#1}}


\newcommand{\refitem}[1] {~\textit{\ref{#1}.)}}

\numberwithin{equation}{section}

\allowdisplaybreaks

\renewcommand{\arraystretch}{1.2}

\let\originalleft\left
\let\originalright\right
\renewcommand{\left}{\mathopen{}\mathclose\bgroup\originalleft}
\renewcommand{\right}{\aftergroup\egroup\originalright}

\theoremheaderfont{\normalfont\bfseries}
\theorembodyfont{\itshape}
\newtheorem{lemma}{Lemma}[section]
\newtheorem{proposition}[lemma]{Proposition}
\newtheorem{theorem}[lemma]{Theorem}
\newtheorem{corollary}[lemma]{Corollary}
\newtheorem{definition}[lemma]{Definition}

\theorembodyfont{\rmfamily}

\newtheorem{example}[lemma]{Example}
\newtheorem{remark}[lemma]{Remark}

\makeatletter
\def\theorem@checkbold{}
\makeatother

\theoremheaderfont{\scshape}
\theorembodyfont{\normalfont}
\theoremstyle{nonumberplain}
\theoremseparator{:}
\theoremsymbol{\hbox{$\boxempty$}}
\newtheorem{proof}{Proof}
\theoremsymbol{\hbox{$\triangledown$}}

\newenvironment{lemmalist}{\begin{compactenum}[\itshape i.)]}{\end{compactenum}}
\newenvironment{theoremlist}{\begin{compactenum}[\itshape i.)]}{\end{compactenum}}
\newenvironment{propositionlist}{\begin{compactenum}[\itshape i.)]}{\end{compactenum}}

\newenvironment{corollarylist}{\begin{compactenum}[\itshape i.)]}{\end{compactenum}}

\newcommand{\I}              {\mathrm{i}}
\newcommand{\E}              {\mathrm{e}}
\newcommand{\D}              {\mathop{}\!\mathrm{d}}
\newcommand{\cc}[1]          {\overline{{#1}}}
\newcommand{\sign}           {\operatorname{\mathrm{sign}}}
\newcommand{\RE}             {\mathsf{Re}}
\newcommand{\IM}             {\mathsf{Im}}
\newcommand{\Unit}           {\mathbb{1}}

\newcommand{\at}[1]          {\big|_{#1}}
\newcommand{\At}[1]          {\Big|_{#1}}

\newcommand{\argument}       {\,\cdot\,}
\newcommand{\id}             {\operatorname{\mathsf{id}}}
\newcommand{\image}          {\operatorname{{\mathrm{im}}}}

\newcommand{\Trans}          {{\mathrm{\scriptscriptstyle{T}}}}
\newcommand{\opp}            {\mathrm{opp}}

\newcommand{\tensor}[1][{}]  {\mathbin{\otimes_{\scriptscriptstyle{#1}}}}
\newcommand{\Tensor}         {\operatorname{\mathrm{T}}}
\newcommand{\Anti}           {\Lambda}
\newcommand{\Sym}            {\mathrm{S}}
\newcommand{\Symmetrizer}    {\operatorname{\mathcal{S}}}
\DeclarePairedDelimiter{\abs}{\lvert}{\rvert}

\newcommand{\norm}[1]        {\left\|{#1}\right\|}
\newcommand{\Fun}[1][k]      {\mathscr{C}^{#1}}
\newcommand{\Cinfty}         {\Fun[\infty]}

\newcommand{\halbnorm}[1]    {\operatorname{\mathrm{#1}}}
\newcommand{\Schwartz}       {\mathscr{S}}

\newcommand{\Pol}            {\operatorname{\mathrm{Pol}}}
\newcommand{\group}[1]       {\mathrm{#1}}
\newcommand{\algebra}[1]     {\mathcal{#1}}
\newcommand{\Diffop}         {\operatorname{\mathrm{DiffOp}}}

\DeclareMathSymbol\dAlembert {\mathord}{AMSa}{"03}

\newcommand{\Sec}[1][k]      {\Gamma^{#1}}
\newcommand{\Secinfty}       {\Sec[\infty]}

\newcommand{\End}            {\operatorname{\mathsf{End}}}
\newcommand{\Aut}            {\operatorname{\mathsf{Aut}}}

\newcommand{\cl}             {\mathrm{cl}}
\newcommand{\supp}           {\operatorname{\mathrm{supp}}}

\newcommand{\racts}          {\mathbin{\triangleleft}}
\newcommand{\cTensor}        {\hat{\mathrm{T}}}
\newcommand{\pitensor}       {\mathbin{\otimes_\pi}}
\newcommand{\cpitensor}      {\mathbin{\hat{\otimes}_\pi}}
\newcommand{\piSym}          {\mathrm{S}_\pi}
\newcommand{\cpiSym}         {\hat{\mathrm{S}}_\pi}
\newcommand{\cSym}          {\hat{\mathrm{S}}}

\newcommand{\formalstar}   {\mathbin{\star_{\nu\Lambda}}}

\newcommand{\Bracket}[1]    {\left\{#1\right\}_{\Lambda}}
\newcommand{\starzLambda}   {\mathbin{\star_{z\Lambda}}}
\newcommand{\starhLambda}   {\mathbin{\star_{\frac{\I\hbar}{2}\Lambda}}}

\newcommand{\WeylR}         {\mathcal{W}_{\mathit R}}
\newcommand{\cWeylR}        {\hat{\mathcal{W}}_{\mathit R}}
\newcommand{\WeylRMinus}    {\mathcal{W}_{{\mathit R}^-}}
\newcommand{\cWeylRMinus}   {\hat{\mathcal{W}}_{{\mathit R}^-}}
\newcommand{\cWeylInfty}    {\hat{\mathcal{W}}_\infty}
\newcommand{\Exp}           {\operatorname{\mathrm{Exp}}}
\newcommand{\starwick}      {\mathbin{\star_{\scriptscriptstyle \mathrm{Wick}}}}

\newcommand{\BracketSigma}[1]{\left\{#1\right\}_{\Lambda_\Sigma}}
\newcommand{\starSigma}      {\mathbin{\star_{\Sigma}}}
\newcommand{\Bracketcov}[1]  {\left\{#1\right\}_{\mathrm{cov}}}
\newcommand{\starcov}        {\mathbin{\star_\cov}}

\newcommand{\Secinftysc}    {\Secinfty_{\mathrm{sc}}}

\newcommand{\cov}           {\mathrm{cov}}

\title{A Nuclear Weyl Algebra}

\author{\textbf{Stefan Waldmann}\\[0.5cm]
  Institut für Mathematik \\
  Lehrstuhl für Mathematik X \\
  Universität Würzburg \\
  Campus Hubland Nord \\
  Emil-Fischer-Straße 31 \\
  97074 Würzburg \\
  \\[0.3cm]
  {\small Contact: \texttt{Stefan.Waldmann@mathematik.uni-wuerzburg.de}}
}

\date{
  Dedicated to Hartmann Römer for his 70th birthday\\[0.5cm]
  \small
  Updated Version: June 2013
}

\begin{document}

%
% title page
%

\maketitle

%
% abstract
%

\begin{abstract}
    A bilinear form on a possibly graded vector space $V$ defines a
    graded Poisson structure on its graded symmetric algebra together
    with a star product quantizing it. This gives a model for the Weyl
    algebra in an algebraic framework, only requiring a field of
    characteristic zero. When passing to $\mathbb{R}$ or $\mathbb{C}$
    one wants to add more: the convergence of the star product should
    be controlled for a large completion of the symmetric
    algebra. Assuming that the underlying vector space carries a
    locally convex topology and the bilinear form is continuous, we
    establish a locally convex topology on the Weyl algebra such that
    the star product becomes continuous. We show that the completion
    contains many interesting functions like exponentials. The star
    product is shown to converge absolutely and provides an entire
    deformation. We show that the completion has an absolute Schauder
    basis whenever $V$ has an absolute Schauder basis. Moreover, the
    Weyl algebra is nuclear iff $V$ is nuclear. We discuss
    functoriality, translational symmetries, and equivalences of the
    construction. As an example, we show how the Peierls bracket in
    classical field theory on a globally hyperbolic spacetime can be
    used to obtain a local net of Weyl algebras.
\end{abstract}

\newpage

%
% table of contents
%

\tableofcontents
\newpage

%
% Introduction
%

\section{Introduction}
\label{sec:Introduction}

The Weyl algebra as the mathematical habitat of the canonical
commutation relations has many incarnations and variants: in a purely
algebraic definition it is the universal unital associative algebra
generated by a vector space $V$ subject to the commutation relations
$vw - wv = \Lambda(v, w) \Unit$ where $\Lambda$ is a symplectic form
on $V$ and $v, w \in V$. For $V = \mathbb{k}^2$ with basis $q$, $p$
and the standard symplectic form, the canonical commutation relations
take the familiar form
\begin{equation}
    \label{eq:CCR}
    qp - pq = \Unit.
\end{equation}
Typically, some scalar prefactor in front of $\Unit$ is incorporated.
When working over the complex numbers, a $C^*$-algebraic version of
the canonical commutation relations is defined by formally
exponentiating the generators from $V$ and using the resulting
commutation relations from \eqref{eq:CCR} with $\I\hbar$ in front of
$\Unit$ for the exponentials. It results in a universal $C^*$-algebra
generated by the exponentials subject to the commutation
relations. This version of the Weyl algebra is most common in the
axiomatic approaches to quantum field theory and quantum mechanics.
More recently, an alternative construction of a Weyl algebra in a
$C^*$-algebraic framework has been proposed and studied in
\cite{buchholz.grundling:2008a} based on resolvents instead of
exponentials. In \cite{binz.honegger.rieckers:2004b} the
$C^*$-algebraic Weyl algebra was shown to be a strict deformation
quantization of a certain Poisson algebra which consists of certain
``bounded'' elements in contrast to the ``unbounded'' generators from
$V$ itself, ultimately leading to a continuous field of
$C^*$-algebras.

The two principle versions of the Weyl algebra differ very much in
their behaviour. While the $C^*$-algebraic formulation has a strong
analytic structure, the algebraic version based on the canonical
commutation relations does not allow for an obvious topology: in fact,
it can easily be proven that any submultiplicative seminorm on the
Weyl algebra generated by $q$ and $p$ with commutation relations
\eqref{eq:CCR} necessarily vanishes. In particular, there will be no
structure of a normed algebra possible.

It is now our main objective of this paper to fill this gap and
provide a reasonable topology for the algebraic Weyl algebra making
the product continuous. Starting point will be a general locally
convex topology on $V$, where we also allow for a graded vector space
and a graded version of the canonical commutation relations, i.e. we
treat the Weyl algebra and the Clifford algebra on the same footing.
The algebraic version of the Weyl algebra will be realized by means of
a deformation quantization \cite{bayen.et.al:1978a} of the symmetric
algebra $\Sym^\bullet(V)$ encoded in a star product. The idea is to
treat the Poisson bracket arising from the bilinear form $\Lambda$ as
a \emph{constant} Poisson bracket and use the Weyl-Moyal star product
quantizing it. From a deformation quantization point of view this is a
very trivial situation, though we of course allow for an
infinite-dimensional vector space $V$, see \cite{waldmann:2007a} for a
gentle introduction to deformation quantization. The bilinear form
$\Lambda$ will not be required to be antisymmetric or
non-degenerate. However, we need some analytic properties. For
convenience, we require $\Lambda$ to be \emph{continuous}, a quite
strong assumption in infinite dimensions. Nevertheless, many
interesting examples fulfill this requirement. In particular, for a
finite-dimensional space $V$ this is always the case.

On the tensor algebra $\Tensor^\bullet(V)$ and hence on the symmetric
algebra $\Sym^\bullet(V)$ there is of course an abundance of locally
convex topologies which all induce the projective topology on each
$V^{\tensor n}$. The two extreme cases are the direct sum topology and
the Cartesian product topology. The direct sum topology for the Weyl
algebra with finitely many generators was used in \cite{cuntz:2005a}
to study bivariant $K$-theory.  In \cite{borchers.yngvason:1976a} a
slightly coarser topology on $\Sym^\bullet(\Schwartz(\mathbb{R}^d))$
than the direct sum topology was studied in the context of quantum
field theories, where $\Schwartz(\mathbb{R}^d)$ is the usual Schwartz
space. It turns out that this topology makes
$\Sym^\bullet(\Schwartz(\mathbb{R}^d))$ a topological algebra,
too. However, for our purposes, this topology is still too fine.
Interesting new phenomena are found in \cite{dito:2005a} for formal
star products in the case the underlying locally convex space $V$ is a
Hilbert space. For the class of functions considered in this paper,
the classification program shows much richer behaviour than in the
well-known finite-dimensional case. However, the required
Hilbert-Schmidt property will differ from the requirements we state in
the sequel.

The first main result is that we can define a new locally convex
topology on the tensor algebra $\Tensor^\bullet(V)$ and hence also on
the symmetric algebra $\Sym^\bullet(V)$, quite explicitly by means of
seminorms controlling the growth of the coefficients $a_n \in
\Sym^n(V)$, in such a way that the star product is continuous. The
completion of $\Sym^\bullet(V)$ with respect to this locally convex
topology will contain many interesting entire functions like
exponentials of elements in $V$.  It turns out that even more is true:
the star product converges absolutely and provides an entire
deformation in the sense of \cite{pflaum.schottenloher:1998a}.  In
fact, we have two versions of this construction depending on a real
parameter $R \ge \frac{1}{2}$ leading to a Weyl algebra $\WeylR(V)$
and a projective limit $\WeylRMinus(V)$ for $R > \frac{1}{2}$. Both
share many properties but differ in others.

If the underlying vector space $V$ has an absolute Schauder basis we
prove that the corresponding Weyl algebra also has an absolute
Schauder basis. The second main result is that the Weyl algebra
$\WeylR(V)$ is nuclear whenever we started with a nuclear $V$. This is
of course a very desirable property and shows that the Weyl algebra
enjoys some good properties. If $V$ is even strongly nuclear then the
second version $\WeylRMinus(V)$ gives a strongly nuclear Weyl
algebra. In the case where $V$ is finite-dimensional, we have both for
trivial reasons: an absolute Schauder basis of $V$ and strong
nuclearity. Thus in this case the corresponding Weyl algebra turns out
to be a (strongly) nuclear algebra with an absolute Schauder basis. In
fact, we can show even more: the underlying locally convex space is a
particular Köthe space which can explicitly be described.

Our construction depends functorially on the data $V$ and $\Lambda$ as
well as on a parameter $R$ which controls the coarseness of the
topology. The particular value $R = \frac{1}{2}$ seems to be
distinguished as it is the limit case for which the product is
continuous. For the second variant of our construction, the case $R =
1$ is distinguished as this is the limit case where the exponentials
are part of the completion.

We show that the topological dual $V'$ acts on the Weyl algebra by
translations. These automorphisms are even \emph{inner} if the element
in $V'$ is in the image of the canonical map $V \longrightarrow V'$
induced by the antisymmetric part of $\Lambda$: here we show that the
exponential series of elements in $V$ are contained in the Weyl
algebra $\WeylR(V)$, provided $R < 1$, and in $\WeylRMinus(V)$ for $R
\le 1$. Since the Weyl algebra does not allow for a general
holomorphic calculus, this is a nontrivial statement and puts
heuristic formulas for the star-exponential on a solid ground. In
particular, these exponentials are also the generators of the
$C^*$-algebraic version of the Weyl algebra, showing that there is
still a close relation. However, it does not seem to be easy to make
the transition to the $C^*$-algebraic world more explicitly.

For a finite-dimensional even vector space $V$, we relate our general
construction to the following two earlier version of the Weyl algebra:
first we show that for a suitable choice of the parameters and the
Poisson structure the Weyl algebra discussed in
\cite{omori.maeda.miyazaki.yoshioka:2000a} coincides with $\WeylR(V)$.
Second, we show that the results from \cite{beiser.waldmann:2011a:pre,
  beiser.roemer.waldmann:2007a} yield the second version
$\WeylRMinus(V)$ for the particular value $R = 1$. This way, we have
now a clear picture on the relation between the two approaches.

Finally, we apply our general construction to an example from
(quantum) field theory. We consider a linear field equation on a
globally hyperbolic spacetime manifold. The Green operators of the
normally hyperbolic differential operator encoding the field equation
define a Poisson bracket, the so-called Peierls bracket. We show the
relevant continuity properties in order to apply the construction of
the Weyl algebra to this particular Poisson bracket. It is shown that
the resulting Poisson algebra and Weyl algebra relate to the canonical
Poisson algebra and Weyl algebra on the initial data of the field
equation. The result will be a local net of Poisson algebras or Weyl
algebras obeying a version of the Haag-Kastler axioms including the
time-slice axiom. On one hand this is a very particular case of the
Peierls bracket discussed in
\cite{brunetti.fredenhagen.ribeiro:2012a:pre}, on the other hand, we
provide a simple quantum theory with honestly converging star product
in this situation thereby going beyond the formal star products as
discussed in \cite{duetsch.fredenhagen:2001a,
  duetsch.fredenhagen:2001b}. It would be very interesting to see how
the much more general (and non-constant) Poisson structures in
\cite{brunetti.fredenhagen.ribeiro:2012a:pre} can be deformation
quantized with a convergent star product.

To conclude this introduction we take the opportunity to point out
some possible further questions arising with our approach to the Weyl
algebra:
\begin{itemize}
\item In finite dimensions it is always possible to choose a
    compatible almost complex structure for a given symplectic Poisson
    structure. Such a choice gives a star product of Wick type where
    the symmetric part of $\Lambda$ now consists of a suitable
    multiple of the compatible positive definite inner product. The
    Wick product enjoys the additional feature of being a
    \emph{positive} deformation \cite{bursztyn.waldmann:2000a}. In
    particular, the evaluation functionals at the points of the dual
    will become positive linear functionals on the Wick algebra. In
    \cite{beiser.roemer.waldmann:2007a} the corresponding GNS
    construction was investigate in detail and yields the usual
    Bargmann-Fock space representation for the canonical commutation
    relations. The case of a Hilbert space of arbitrary dimension will
    be the natural generalization for this. In general, the existence
    of a compatible almost complex structure having good continuity
    properties is far from being obvious. We will address these
    questions in a future project.
\item Closely related will be the question what the states of the
    locally convex Weyl algebra will be in general. While this
    question might be quite hard to attack in full generality, the
    more particular case of the Weyl algebra arising from the Peierls
    bracket will be already very interesting: here one has certain
    candidates of so-called Hadamard states from (quantum) field
    theory. It would be interesting to see whether and how they can be
    matched with compatible almost complex structures and evaluations
    at points in the dual.
\item In infinite dimensions there are important examples of bilinear
    forms which are not continuous but only separately continuous. It
    would be interesting to extend our analysis to this situation as
    well in such a way that one obtains a separately continuous star
    product. Yet another scenario would be to investigate a
    \emph{bornological} version of the Weyl algebra construction: many
    bilinear forms turn out to be compatible with naturally defined
    bornologies rather than locally convex topologies. Thus a
    bornological star product would be very desirable and has the
    potential to cover many more examples not yet available with our
    present construction. A good starting point might be
    \cite{voigt:2008a, meyer:2004a}. These questions will be addressed
    in a future project.
\item Finally, already in finite dimensions it will be very
    challenging to go beyond the geometrically trivial case of
    constant Poisson structures. One possible strategy is to use the
    completed nuclear Weyl algebra build on each tangent space of a
    symplectic manifold. This leads to a Weyl algebra bundle, now in
    our convergent setting. In a second step one should try to
    understand how the Fedosov construction \cite{fedosov:1996a} of a
    formal star product can be transferred to this convergent setting
    provided the curvature and its covariant derivatives of a suitably
    chosen symplectic connection satisfy certain (still to be found)
    bounds.
\end{itemize}

The paper is organized as follows: In Section~\ref{sec:Preliminaries}
we fix our notation and recall some well-known algebraic facts on
constant Poisson structures and their deformation quantizations. The
next section contains the core results of the paper. We first
construct several systems of seminorms on the tensor algebra and
investigate the continuity properties of the tensor product with
respect to them. The continuity of the Poisson bracket is then
established but the continuity of the star product requires a suitable
projective limit construction in addition. Nevertheless, the resulting
systems of seminorms are still described explicitly. This way, we
arrive at our definition of the Weyl algebra in
Definition~\ref{definition:WeylAlgebraWR} and show that it yields a
locally convex algebra in Theorem~\ref{theorem:WeylAlgebra}. We prove
that the star product converges absolutely, provides an entire
deformation, and enjoys good reality properties. In
Section~\ref{sec:BasesNuclearity} we show two main results: first that
if $V$ has an absolute Schauder basis the Weyl algebra also has an
absolute Schauder basis. Second, we prove in
Theorem~\ref{theorem:NuclearWeylAlgebra} that the Weyl algebra is
(strongly) nuclear iff $V$ is (strongly) nuclear.
Section~\ref{sec:SymmetriesEquivalences} is devoted to various
symmetries and equivalences. We prove that the algebraic symmetries
can be cast into the realm of the locally convex Weyl algebra, too,
and yield a good functoriality of the construction. If the convergence
parameter $R$ is less than $1$ then translations are shown to act by
inner automorphisms. Moreover, we show that the isomorphism class of
the Weyl algebra only depends on the antisymmetric part of the
bilinear form $\Lambda$. Finally, we relate our construction to the
one from \cite{omori.maeda.miyazaki.yoshioka:2000a} in
Proposition~\ref{proposition:OMMYcoincides} as well as to the version
from \cite{beiser.roemer.waldmann:2007a, beiser.waldmann:2011a:pre} in
Proposition~\ref{proposition:WickIsWeylForREins}.  The final and quite
large Section~\ref{sec:ExamplePeierlsFreeQFT} contains a first
nontrivial example: the canonical and the covariant Poisson structures
arising in non-interacting field theories on globally hyperbolic
spacetimes.  We recall the necessary preliminaries to define and
compare the two Poisson structures in detail. The continuity
properties of both allow to apply our general construction of the Weyl
algebra, leading to a detailed description in
Theorem~\ref{theorem:CovariantVSCanonical}.  As a first application we
show that both on the classical side as well as on the quantum side
the construction leads to a local net of observables satisfying the
time-slice axiom.

%
% Acknowledgements
%

\noindent
\textbf{Acknowledgements:} It is a pleasure to thank many colleagues
for their comments at various stages when preparing this
manuscript. In particular, I would like to thank Christian Bär, Klaus
Fredenhagen, and Frank Pfäffle for valuable discussions on the Peierls
bracket, Gandalf Lechner for pointing out
\cite{buchholz.grundling:2008a}, Ralf Meyer for the suggestion to
investigate the bornological versions of the Weyl algebra,
Karl-Hermann Neeb for many remarks and for drawing my attention to
\cite{borchers.yngvason:1976a}, and Pedro Lauridsen Ribeiro for
clarifying remarks.

%
% Preliminaries
%

\section{Algebraic Preliminaries}
\label{sec:Preliminaries}

In this section we collect some algebraic preliminaries on constant
Poisson brackets and their deformation quantization to fix our
conventions. In the following, let $\mathbb{k}$ be a field of
characteristic $0$ and let $V$ be a $\mathbb{k}$-vector space.

%
% The Symmetric Algebra $\Sym^\bullet(V)$ and Sign Conventions
%

\subsection{The Symmetric Algebra $\Sym^\bullet(V)$ and Sign Conventions}
\label{subsec:SymmetricAlgebraSignConventions}

In order to treat the symmetric and the Grassmann algebra on the same
footing, we assume that $V = V_{\mathbf{0}} \oplus V_{\mathbf{1}}$ is
$\mathbb{Z}_2$-graded. In many applications, $V$ is even
$\mathbb{Z}$-graded and the induced $\mathbb{Z}_2$-grading is then
given by the even and odd part of $V$. A vector $v \in V_{\mathbf{0}}$
is called homogeneous of parity $\mathbf{0}$ while a vector in
$V_{\mathbf{1}}$ is called homogeneous of parity
$\mathbf{1}$. Occasionally, we denote the parity of the vector $v$
with the same symbol $v \in \mathbb{Z}_2$ and we also shall refer to
even and odd parity.  In all what follows, we will make use of the
\emph{Koszul sign rule}, i.e.  if two things with parities $a, b \in
\mathbb{Z}_2$ are exchanged this gives an extra sign $(-1)^{ab}$.  The
homogeneous components of $v \in V$ will be denoted by $v =
v_{\mathbf{0}} + v_{\mathbf{1}}$.

In more detail, we will need to following signs for
\emph{symmetrization}. For homogeneous vectors $v_1, \ldots, v_n \in
V$ and a permutation $\sigma \in S_n$ one defines the sign
\begin{equation}
    \label{eq:GradedSign}
    \sign(v_1, \ldots, v_n; \sigma)
    =
    \prod_{i < j}
    \frac{\sigma(i) + (-1)^{v_{\sigma(i)}v_{\sigma(j)}} \sigma(j)}
    {i + (-1)^{v_iv_j}j}.
\end{equation}
Then $\sign(v_1, \ldots, v_n; \sigma) = 1$ if all the $v_1,
\ldots, v_n$ are even and $\sign(v_1, \ldots, v_n; \sigma) =
\sign(\sigma)$ is the usual signum of the permutation for all $v_1,
\ldots, v_n$ odd. It is then straightforward to check that
\begin{equation}
    \label{eq:SnAction}
    (v_1 \tensor \cdots \tensor v_n) \racts \sigma
    =
    \sign(v_1, \ldots, v_n; \sigma)
    v_{\sigma(1)} \tensor \cdots \tensor v_{\sigma(n)}
\end{equation}
extends to a well-defined right action of $S_n$ on $V^{\tensor n}$. We
use this right action to define the symmetrization operator
\begin{equation}
    \label{eq:Symmetrization}
    \Symmetrizer_n\colon V^{\tensor n}
    \ni v
    \; \mapsto \;
    \Symmetrizer_n(v) = \frac{1}{n!} \sum_{\sigma \in S_n}
    v \racts \sigma
    \in V^{\tensor n}.
\end{equation}
One has $\Symmetrizer_n \circ \Symmetrizer_n = \Symmetrizer_n$ since
\eqref{eq:SnAction} is an action. In the case where $V =
V_{\mathbf{0}}$, the operator $\Symmetrizer_n$ is the usual total
symmetrization, if $V = V_{\mathbf{1}}$ we get the total
antisymmetrization operator.

For later use it will be advantageous to define the symmetric algebra
not as a quotient algebra of the tensor algebra but as a subspace with
a new product. Thus we set
\begin{equation}
    \label{eq:SymnV}
    \Sym^n(V) = \image \Symmetrizer_n \subseteq V^{\tensor n}
    \quad
    \textrm{and}
    \quad
    \Sym^0(V) = \mathbb{k}.
\end{equation}
The elements in $\Sym^n(V)$ consist of the symmetric, i.e.  invariant
tensors with respect to the action \eqref{eq:SnAction}. Moreover, we
set
\begin{equation}
    \label{eq:SymmetricAlgebra}
    \Sym^\bullet(V) = \bigoplus_{n = 0}^\infty \Sym^n(V)
    \subseteq
    \Tensor^\bullet(V) = \bigoplus_{n=0}^\infty V^{\tensor n}.
\end{equation}
Alternatively, we can use the idempotent $\Symmetrizer =
\bigoplus_{n=0}^\infty \Symmetrizer_n$ where $\Symmetrizer_0 = \id$
and get $\Sym^\bullet(V) = \image \Symmetrizer$.

Next we define the symmetric tensor product $\mu\colon \Sym^\bullet(V)
\tensor \Sym^\bullet(V) \longrightarrow \Sym^\bullet(V)$ of $v, w \in
\Sym^\bullet(V)$ as usual by
\begin{equation}
    \label{eq:SymmetricTensorProduct}
    vw = \mu(v \tensor w) = \Symmetrizer(v \tensor w).
\end{equation}
Most of the time we shall omit the symbol $\mu$ for products. Then the
following statement is well-known:
\begin{lemma}
    \label{lemma:SymmetricAlgebra}%
    The symmetric tensor product turns $\Sym^\bullet(V)$ into an
    associative commutative unital algebra freely generated by $V$.
\end{lemma}
Moreover, it is $\mathbb{Z}$-graded with respect to the tensor
degree. Note however, that we do not use this degree for sign purposes
at all.  Freely generated means that a homogeneous map $\phi\colon V
\longrightarrow \algebra{A}$ of parity $\mathbf{0}$ into another
associative $\mathbb{Z}_2$-graded commutative unital algebra
$\algebra{A}$ has a unique extension $\Phi\colon \Sym^\bullet(V)
\longrightarrow \algebra{A}$ as unital algebra homomorphism.

Beside the symmetric algebra we will also need tensor products of
algebras. Thus let $\algebra{A}$ and $\algebra{B}$ be two associative
$\mathbb{Z}_2$-graded algebras. On their tensor product $\algebra{A}
\tensor \algebra{B}$ a new product is defined by linear extension of
\begin{equation}
    \label{eq:TensorProductAlgebras}
    (a \tensor b) (a' \tensor b')
    =
    (-1)^{ba'} aa' \tensor bb',
\end{equation}
where $a, a' \in \algebra{A}$ and $b, b' \in \algebra{B}$ are
homogeneous elements. This turns $\algebra{A} \tensor \algebra{B}$
again into an associative $\mathbb{Z}_2$-graded algebra.

Finally, we recall that the Koszul sign rule also applies to tensor
products of maps and evaluations, i.e.  for homogeneous maps
$\phi\colon V \longrightarrow W$ and $\psi\colon \tilde{V}
\longrightarrow \tilde{W}$ we define their tensor product $\phi
\tensor \psi\colon V \tensor \tilde{V} \longrightarrow W \tensor
\tilde{W}$ by
\begin{equation}
    \label{eq:phitensorpsiMap}
    (\phi \tensor \psi)(v \tensor w)
    =
    (-1)^{\psi v} \phi(v) \tensor \psi(w)
\end{equation}
on homogeneous vectors and extend linearly.

%
% Constant Poisson Structures and Formal Star Products
%

\subsection{Constant Poisson Structures and Formal Star Products}
\label{subsec:ConstantPoissonStructuresFormalStarProducts}

Recall that a homogeneous multilinear map of parity $\mathbf{0}$
\begin{equation}
    \label{eq:MultiDiffop}
    P\colon
    \Sym^\bullet(V) \times \cdots \times \Sym^\bullet(V)
    \longrightarrow
    \Sym^\bullet(V)
\end{equation}
is called a \emph{multiderivation} if for each argument it satisfies
the Leibniz rule
\begin{equation}
    \label{eq:LeibnizRule}
    \begin{split}
        &P(v_1, \ldots, v_{k-1}, v_k v_k', v_{k+1}, \ldots, v_n) \\
        &\quad=
        (-1)^{(v_1 + \cdots + v_{k-1})v_k}
        v_k P(v_1, \ldots, v_k', \ldots, v_n)
        +
        (-1)^{v_k'(v_{k+1} + \cdots + v_n)}
        P(v_1, \ldots, v_k, \ldots, v_n) v_k',
    \end{split}
\end{equation}
where we follow again the Koszul sign rule. Note that if one $v_i$ is
a constant, i.e. $v_i \in \Sym^0(V)$, then $P(v_1, \ldots, v_n) =
0$. Since $V$ generates $\Sym^\bullet(V)$, a multiderivation is
uniquely determined by its values on $V \times \cdots \times
V$. Conversely, any multilinear homogeneous map $V \times \cdots
\times V \longrightarrow \Sym^\bullet(V)$ of parity $\mathbf{0}$
extends to a multiderivation since $V$ generates $\Sym^\bullet(V)$
\emph{freely}. Though one can also consider odd multiderivations, we
shall not need them in the sequel.

For later estimates, we will need the following more explicit form of
this extension for the particular case of a \emph{bilinear}
homogeneous map
\begin{equation}
    \label{eq:Lambda}
    \Lambda\colon V \times V \longrightarrow \mathbb{k} = \Sym^0(V)
\end{equation}
of parity $\mathbf{0}$. To this end we first define the linear map
\begin{equation}
    \label{eq:PLambda}
    P_\Lambda\colon
    \Sym^\bullet(V) \tensor \Sym^\bullet(V)
    \longrightarrow
    \Sym^{\bullet-1}(V) \tensor \Sym^{\bullet-1}(V)
\end{equation}
to be the linear extension of
\begin{equation}
    \label{eq:PLambdaDef}
    \begin{split}
        &P_\Lambda(v_1 \dots v_n \tensor w_1 \cdots w_m) \\
        &\quad=
        \sum_{k=1}^n \sum_{\ell=1}^m
        (-1)^{
          v_k(v_{k+1} + \cdots + v_n)
          + w_\ell(w_1 + \cdots + w_{\ell-1})
        }
        \Lambda(v_k, w_\ell)
        v_1 \cdots \stackrel{k}{\wedge} \cdots v_n
        \tensor
        w_1 \cdots \stackrel{\ell}{\wedge} \cdots w_m.
    \end{split}
\end{equation}
The requirement on the $\mathbb{Z}$-grading implies that $P_\Lambda$
vanishes on tensors of the form $\Unit \tensor w$ or $v \tensor
\Unit$.

First note that $\Lambda$ is of parity $\mathbf{0}$ and thus
$\Lambda(v_k, w_\ell)$ is only nontrivial if $v_k$ and $w_\ell$ have
the same parity. Hence we can exchange the parities $v_k$ and $w_\ell$
in the above sign. Second, note that the map $P_\Lambda$ is indeed
well-defined since the right hand side is totally symmetric in $v_1,
\ldots, v_n$ and in $w_1, \ldots, w_m$. With respect to the algebra
structure \eqref{eq:TensorProductAlgebras} of $\Sym^\bullet(V) \tensor
\Sym^\bullet(V)$ we can characterize the map $P_\Lambda$ now as
follows:
\begin{lemma}
    \label{lemma:PLambdaMapProperties}%
    The map $P_\Lambda$ is the unique map with
    \begin{lemmalist}
    \item \label{item:PLambdaOnGenerators} $P_\Lambda(v \tensor w) =
        \Lambda(v, w) \Unit \tensor \Unit$ for all $v, w \in V$,
    \item \label{item:PLambdaLeibnizEins} $P_\Lambda(v \tensor wu) =
        P_\Lambda(v \tensor w) (\Unit \tensor u) + (-1)^{wv}(\Unit
        \tensor w) P_\Lambda(v \tensor u)$ for all $v, w, u \in
        \Sym^\bullet(V)$,
    \item \label{item:PLambdaLeibnizZwei} $P_\Lambda(vw \tensor u) =
        (v \tensor \Unit) P_\Lambda(w \tensor u) + (-1)^{wu}
        P_\Lambda(v \tensor u) (\Unit \tensor w)$ for all $v, w, u \in
        \Sym^\bullet(V)$.
    \end{lemmalist}
\end{lemma}
Note that the two Leibniz rules imply $P_\Lambda(v \tensor \Unit) = 0
= P_\Lambda(\Unit \tensor v)$ for all $v \in \Sym^\bullet(V)$.

In order to rewrite the Leibniz rules for $P_\Lambda$ in a more
conceptual way, we have to introduce the canonical flip operator
$\tau_{VW}\colon V \tensor W \longrightarrow W \tensor V$ for
$\mathbb{Z}_2$-graded vector spaces $V$ and $W$ by
\begin{equation}
    \label{eq:tauVW}
    \tau_{VW}(v \tensor w) = (-1)^{vw} w \tensor v
\end{equation}
on homogeneous elements and linear extension to all tensors. We
usually write $\tau$ if the reference to the underlying vector spaces
is clear. Using $\tau$ we define the operators
\begin{equation}
    \label{eq:PLambdaOperatorsEinsZweiDrei}
    P_\Lambda^{12}, P_\Lambda^{23}, P_\Lambda^{13}\colon
    \Sym^\bullet(V) \tensor \Sym^\bullet(V) \tensor \Sym^\bullet(V)
    \longrightarrow
    \Sym^\bullet(V) \tensor \Sym^\bullet(V) \tensor \Sym^\bullet(V)
\end{equation}
on the triple tensor product by
\begin{equation}
    \label{eq:PLambdaOperatorsEinsZweiDreiDef}
    P_\Lambda^{12} = P_\Lambda \tensor \id,
    \quad
    P_\Lambda^{23} = \id \tensor P_\Lambda,
    \quad
    \textrm{and}
    \quad
    P_\Lambda^{13}
    =
    (\id \tensor \tau)
    \circ (P_\Lambda \tensor \id)
    \circ (\id \tensor \tau).
\end{equation}
These operators have again parity $\mathbf{0}$ and change the tensor
degrees by $(-1, -1, 0)$, $(0, -1, -1)$, and by $(-1, 0, -1)$,
respectively.
\begin{lemma}
    \label{lemma:LeibnizRuleForPLambda}%
    The Leibniz rules for $P_\Lambda$ can be written as
    \begin{equation}
        \label{eq:LeibnizPLambdaEins}
        P_\Lambda \circ (\mu \tensor \id)
        =
        (\mu \tensor \id)
        \circ
        \left(P_\Lambda^{13} + P_\Lambda^{23}\right)
    \end{equation}
    and
    \begin{equation}
        \label{eq:LeibnizPLambdaZwei}
        P_\Lambda \circ (\id \tensor \mu)
        =
        (\id \tensor \mu)
        \circ
        \left(P_\Lambda^{12} + P_\Lambda^{13}\right).
    \end{equation}
\end{lemma}
Analogously, we have similar Leibniz rules for the operators
$P_\Lambda^{12}$, $P_\Lambda^{23}$, and $P_\Lambda^{13}$ which show
that they will be uniquely determined by their values on generators of
$\Sym^\bullet(V) \tensor \Sym^\bullet(V) \tensor
\Sym^\bullet(V)$. Hence the products $P_\Lambda^{12} \circ
P_\Lambda^{23}$ etc.\ will be uniquely determined by their values on
quadratic expressions in the generators. This will allow for a rather
straightforward computation leading to the following observation:
\begin{lemma}
    \label{lemma:PLambdasAllCommute}%
    The operators $P_\Lambda^{12}$, $P_\Lambda^{23}$, and
    $P_\Lambda^{13}$ commute pairwise.
\end{lemma}
This lemma together with the Leibniz rule in form of
Lemma~\ref{lemma:LeibnizRuleForPLambda} gives immediately the
following result, see e.g.  \cite[Sect.~6.2.4]{waldmann:2007a} for a
detailed proof:
\begin{proposition}
    \label{proposition:FormalStarProduct}%
    On $\Sym^\bullet(V)[[\nu]]$ one obtains a $\mathbb{Z}_2$-graded
    associative $\mathbb{k}[[\nu]]$-bilinear multiplication by
    \begin{equation}
        \label{eq:TheFormalStarProduct}
        v \formalstar w
        =
        \mu \circ \E^{\nu P_\Lambda} (v \tensor w),
    \end{equation}
    where all $\mathbb{k}$-multilinear maps are extended to be
    $\mathbb{k}[[\nu]]$-multilinear as usual.
\end{proposition}
\begin{remark}[Commuting derivations]
    \label{remark:CommutingDerivations}%
    A particularly simple case is obtained for $\Lambda = \varphi
    \tensor \psi$ with $\varphi, \psi \in V^*$ of equal parity. In
    this case, we denote by $X_\varphi, X_\psi\colon \Sym^\bullet(V)
    \longrightarrow \Sym^{\bullet -1}(V)$ the corresponding
    derivations of the same parity. It is then easy to see that
    $[X_\varphi, X_\psi] = 0$ and that $P_\Lambda = X_\varphi \tensor
    X_\psi$. This example (without $\mathbb{Z}_2$-grading) of a formal
    associative deformation via commuting derivations was first
    considered by Gerstenhaber in \cite[Thm.~8]{gerstenhaber:1968a}
    and was generalized in various ways ever since.
\end{remark}
Also the next proposition is folklore and easily verified:
\begin{proposition}
    \label{proposition:InducedPoissonBracket}%
    Let $\algebra{A}$ be an associative $\mathbb{Z}_2$-graded
    commutative algebra and let $\star$ be a formal associative
    deformation of it such that $(\algebra{A}[[\nu]], \star)$ is still
    $\mathbb{Z}_2$-graded. Then the first order of the
    $\star$-commutator defines a Poisson bracket on $\algebra{A}$.
\end{proposition}
Of course, we always take the $\mathbb{Z}_2$-graded commutators and
Poisson brackets. In our example, this leads to the following Poisson
bracket:
\begin{corollary}
    \label{corollary:ThePoissonBracket}%
    Let $\Lambda$ be as above and set $P^\opp_\Lambda = \tau \circ
    P_\Lambda \circ \tau$. Then
    \begin{equation}
        \label{eq:ThePoissonBracket}
        \Bracket{v, w}
        =
        \mu \circ (P_\Lambda - P^\opp_\Lambda) (a \tensor b)
    \end{equation}
    defines a Poisson bracket for $\Sym^\bullet(V)$.
\end{corollary}
Alternatively, we can also consider the \emph{symmetric} and
\emph{antisymmetric part}
\begin{equation}
    \label{eq:AntiAndSymmetricPartOfLambda}
    \Lambda_\pm
    =
    \frac{1}{2}\left(\Lambda \pm \Lambda \circ \tau\right)
\end{equation}
of $\Lambda$ such that $\Lambda = \Lambda_+ + \Lambda_-$. Then we note
that with $\Lambda^\opp = \Lambda \circ \tau$ we have
\begin{equation}
    \label{eq:PLambdaOpp}
    P_{\Lambda^\opp} = P_\Lambda^\opp
\end{equation}
and thus
\begin{equation}
    \label{eq:PLambdaAndPLambdaMinus}
    P_\Lambda - P^\opp_\Lambda = 2 P_{\Lambda_-}.
\end{equation}
Thus $\{v, w\}_{\Lambda} = 2 \mu \circ P_{\Lambda_-}(v \tensor w)$
depends only on the antisymmetric part. Nevertheless, the star product
$\star$ in \eqref{eq:TheFormalStarProduct} depends on $\Lambda$ and
not just on $\Lambda_-$.  It is this Poisson bracket for which
$\formalstar$ provides a formal deformation quantization.

In general, one requires only a formal star product but since our
Poisson bracket is rather particular, we can sharpen the deformation
result as follows:
\begin{corollary}
    \label{corollary:NonFormalDeformation}%
    The product $\formalstar$ restricts to $\Sym^\bullet(V)[\nu]$
    which becomes an associative $\mathbb{Z}_2$-graded algebra over
    $\mathbb{k}[\nu]$.
\end{corollary}
More precisely, for $v, w \in \Sym^\bullet(V)$ we have $P_\Lambda^n(v
\tensor w) = 0$ as soon as $n \in \mathbb{N}_0$ is larger than the
maximal symmetric degree in $v$ or $w$. It follows that in
$\Sym^\bullet(V)[\nu]$ we can replace the formal parameter $\nu$ by
any element of $\mathbb{k}$ and get a well-defined associative
multiplication from $\formalstar$.

Also the following result is well-known and obtained from an easy
induction: the elements of $V$ generate $\Sym^\bullet(V)[\nu]$ with
respect to $\formalstar$:
\begin{corollary}
    \label{corollary:VGeneratesWeylAlgebra}%
    The $\mathbb{k}[\nu]$-algebra $\Sym^\bullet(V)[\nu]$ is generated
    by $V$.
\end{corollary}

%
% Symmetries and Equivalences
%

\subsection{Symmetries and Equivalences}
\label{subsec:Symmetries}

The symmetric algebra $\Sym^\bullet(V)$ can be interpreted as the
polynomials on the ``predual'' of $V$, which, of course, needs not to
exist in infinite dimensions. Alternatively, $\Sym^\bullet(V)$ injects
as a subalgebra into the polynomials $\Pol^\bullet(V^*)$ on the dual
of $V$. We use this heuristic point of view now to establish some
symmetries of $\Bracket{\argument, \argument}$ and $\formalstar$ which
justify the term ``constant'' Poisson structure.

Let $\varphi \in V^*$ be an even linear functional, i.e.  $\varphi \in
V_{\mathbf{0}}^*$, then the linear map $v \mapsto v + \varphi(v)\Unit$
is even, too, and thus it extends uniquely to a unital algebra
homomorphism $\tau_\varphi^*\colon \Tensor^\bullet(V) \longrightarrow
\Tensor^\bullet(V)$. Clearly, the symmetry properties of the tensors
in $\Tensor^\bullet(V)$ are preserved by $\tau_\varphi^*$ and thus it
restricts to a unital algebra homomorphism
\begin{equation}
    \label{eq:Translation}
    \tau_\varphi^*\colon
    \Sym^\bullet(V) \longrightarrow \Sym^\bullet(V),
\end{equation}
now with respect to the symmetric tensor product.  One has $\tau_0^* =
\id$ and $\tau_\varphi^* \tau_\psi^* = \tau_{\varphi + \psi}^*$ for
all $\varphi, \psi \in V_{\mathbf{0}}^*$. Thus we get an action of the
abelian group $V_{\mathbf{0}}^*$ on $\Sym^\bullet(V)$ by
automorphisms. In the interpretation of polynomials these
automorphisms correspond to pull-backs with \emph{translations} via
$\varphi$, hence the above notation.

The other important symmetry emerges from the endomorphisms of $V$
itself. Let $A\colon V \longrightarrow V$ be an even linear map and
denote the extension as unital algebra homomorphism again by $A\colon
\Sym^\bullet(V) \longrightarrow \Sym^\bullet(V)$. This yields an
embedding of $\End_{\mathbf{0}}(V)$ into the unital algebra
endomorphisms of $\Sym^\bullet(V)$. In particular, we get a group
homomorphism of $\group{GL}_{\mathbf{0}}(V)$ into
$\Aut_{\mathbf{0}}(\Sym^\bullet(V))$. For $A \in
\group{GL}_{\mathbf{0}}(V)$ and $\varphi \in V_{\mathbf{0}}^*$ we have
the relation $A^{-1} \tau_\varphi^* A v = \tau_{A^*\varphi}^* v$ for
the generators $v \in V$ and hence also in general
\begin{equation}
    \label{eq:AtauA}
    A^{-1} \tau_\varphi^* A = \tau_{A^*\varphi}.
\end{equation}
This gives an action of the semidirect product
$\group{GL}_{\mathbf{0}}(V) \ltimes V^*$ on $\Sym^\bullet(V)$ via
unital algebra automorphisms.

For the bilinear map $\Lambda$ we consider the group of invertible
even endomorphisms of $V$ preserving it and denote this group by
\begin{equation}
    \label{eq:AutVLambda}
    \Aut(V, \Lambda)
    =
    \left\{
        A \in \group{GL}_{\mathbf{0}}(V)
        \; \big| \;
        \Lambda(Av, Aw) = \Lambda(v, w)
        \;
        \textrm{for all}
        \;
        v, w \in V
    \right\}.
\end{equation}
Note that such an automorphism preserves $\Lambda_+$ and $\Lambda_-$
separately. However, $\Lambda_-$ and $\Lambda_+$ might have a larger
invariance group than $\Aut(V, \Lambda)$.
\begin{remark}
    \label{remark:SymplecticGroup}%
    Suppose that $\Lambda = - \Lambda^\opp = \Lambda_-$ is already
    antisymmetric and non-degenerate. In the case where $V =
    V_{\mathbf{0}}$ consists of even vectors only, $\Lambda$ is a
    symplectic form and $\Aut(V, \Lambda)$ is the corresponding
    symplectic group. In the case where $V = V_{\mathbf{1}}$ is odd,
    $\Lambda$ corresponds to an inner product and $\Aut(V, \Lambda)$
    is the corresponding pseudo-orthogonal group. Note however, that
    we will also be interested in the case where $\Lambda$ is not
    necessarily antisymmetric and not necessarily non-degenerate.
\end{remark}
\begin{lemma}
    \label{lemma:SymmetriesOfPoissonStructure}%
    The subgroup $\Aut(V, \Lambda) \ltimes V^*$ acts on
    $\Sym^\bullet(V)$ as automorphisms of $\Bracket{\argument,
      \argument}$ and $\star$.
\end{lemma}
\begin{proof}
    First consider $A \in \Aut(V, \Lambda)$. Then on generators one
    sees that $P_\Lambda \circ (A \tensor A) = (A \tensor A) \circ
    P_\Lambda$, which therefore holds in general. From this we see
    that $A$ is an automorphism of both, the Poisson bracket and the
    star product. Analogously, for $\varphi \in V_{\mathbf{0}}^*$ one
    checks first on generators and then in general that $P_\Lambda
    \circ (\tau_\varphi^* \tensor \tau_\varphi^*) = (\tau_\varphi^*
    \tensor \tau_\varphi^*) \circ P_\Lambda$.
\end{proof}
In this sense, both the Poisson bracket and the star product are
\emph{constant}, i.e.  translation-invariant.

In a next step we discuss to what extend the automorphisms are
inner. We consider only the infinitesimal picture as the integrated
version will require analytic tools. The bilinear form $\Lambda$
induces a linear map into the dual $V^*$. More precisely, we need the
antisymmetric part $\Lambda_-$ of $\Lambda$ as it appears also in the
Poisson bracket \eqref{eq:ThePoissonBracket}. This defines an even
linear map
\begin{equation}
    \label{eq:sharpMap}
    \sharp\colon
    V \ni
    v \; \mapsto \; v^\sharp = \Lambda_-(v, \argument)
    \in V^*.
\end{equation}
\begin{lemma}
    \label{lemma:InnernessDerivations}%
    Let $\varphi \in V^*$ be homogeneous and denote by
    $X_\varphi\colon \Sym^\bullet(V) \longrightarrow \Sym^\bullet(V)$
    the homogeneous derivation extending $\varphi\colon V
    \longrightarrow \mathbb{k}$.
    \begin{lemmalist}
    \item \label{item:XvarphiPoissonDerivation}$X_\varphi$ is a
        Poisson derivation of parity $\varphi$, i.e.  we have
        \begin{equation}
            \label{eq:XvarphiPoissonDerivation}
            X_\varphi \Bracket{a, b}
            =
            \Bracket{X_\varphi(a), b}
            +
            (-1)^{\varphi a} \Bracket{a, X_\varphi(b)}
        \end{equation}
        for all homogeneous $a, b \in \Sym^\bullet(V)$.
    \item \label{item:XvarphiInnerPoissonDer} $X_\varphi$ is inner iff
        $\varphi \in \image \sharp$. In this case $X_\varphi =
        \Bracket{v, \argument}$ for any $v \in V$ with $2 v^\sharp =
        \varphi$.
    \item \label{item:XvarphiDerivationStar} $X_\varphi$ is a
        derivation of $\formalstar$, i.e.  we have
        \begin{equation}
            \label{eq:XvarphiDerivationStar}
            X_\varphi(a \formalstar b)
            =
            X_\varphi(a) \formalstar b
            +
            (-1)^{\varphi a} a \formalstar X_\varphi(b)
        \end{equation}
        for all homogeneous $a, b \in \Sym^\bullet(V)$.
    \item \label{item:XvarphiQuasiInnerDer} $X_\varphi$ is a
        quasi-inner derivation of $\formalstar$, i.e.  $X_\varphi =
        \frac{1}{\nu} [a, \argument]_{\formalstar}$ for some $a \in
        \Sym^\bullet(V)[[\nu]]$, iff $\varphi \in \image \sharp$. In
        this case $a = v \in V$ with $2 v^\sharp = \varphi$ will do
        the job.
    \end{lemmalist}
\end{lemma}
\begin{proof}
    Consider an even linear map $P\colon \Sym^\bullet(V) \tensor
    \Sym^\bullet(V) \longrightarrow \Sym^\bullet(V) \tensor
    \Sym^\bullet(V)$ satisfying the Leibniz rules from
    Lemma~\ref{lemma:PLambdaMapProperties},
    \refitem{item:PLambdaLeibnizEins} and
    \refitem{item:PLambdaLeibnizZwei}, and let $X$ be any homogeneous
    derivation of either even or odd parity. Then we claim that the
    operator
    \[
    D =
    P \circ (X \tensor \id + \id \tensor X)
    -
    (X \tensor \id + \id \tensor X) \circ P
    \]
    satisfies the Leibniz rules
    \[
    D(ab \tensor c)
    =
    (-1)^{bc} D(a \tensor c) (b \tensor \Unit)
    +
    (-1)^{Xa} (a \tensor \Unit) D(b \tensor c)
    \]
    and
    \[
    D(a \tensor bc)
    =
    D(a \tensor b) (\Unit \tensor c)
    +
    (-1)^{(X + a)b} (\Unit \tensor b) D(a \tensor c)
    \]
    for all homogeneous $a, b, c \in \Sym^\bullet(V)$. This is a
    simple verification and does not use that $P$ is (anti-)
    symmetric. In our case, we conclude that $D$ is uniquely
    determined by its values on the generators of $\Sym^\bullet(V)
    \tensor \Sym^\bullet(V)$. For $P = P_\Lambda$ and $X = X_\varphi$
    it is easy to check that $D = 0$ on generators and thus
    $P_\Lambda$ and $(X_\varphi \tensor \id + \id \tensor X_\varphi)$
    commute. But this implies the first as well as the third part. Now
    consider $\varphi \in \image \sharp$, i.e.  there is a $v \in V$
    with $\varphi = 2\Lambda_-(v, \argument)$. In this case we get for
    $w \in V$
    \[
    \Bracket{v, w}
    =
    \Lambda(v, w) \Unit - (-1)^{vw} \Lambda(w, v) \Unit
    =
    2 \Lambda_-(v, w) \Unit
    =
    \varphi(w) \Unit
    =
    X_\varphi(w).
    \]
    Since the derivation $X_\varphi$ is determined by its values on
    generators this implies $X_\varphi = \Bracket{v, \argument}$. For
    the converse, assume that $X_\varphi = \Bracket{v, \argument}$ for
    some $v \in \Sym^\bullet(V)$ which we write as $v = \sum_n v_n$
    with $v_n \in \Sym^n(V)$. Then for $w \in V$ we have
    $\Bracket{v_n, w} \in \Sym^{n-1}(V)$ while $X_\varphi(w) \in
    \Sym^0(V)$. Thus we necessarily have $X_\varphi(w) = \Bracket{v_1,
      w}$, i.e.  the higher order terms in $v$ are not necessary. But
    then $X_\varphi = \Bracket{v_1, \argument}$ follows, proving
    $\varphi = 2v_1^\sharp$. The fourth part is similar, since for $v
    \in V$ we have $v \formalstar a = va + \nu \mu \circ P_\Lambda(v
    \tensor a)$ without higher order terms. Thus $[v, a]_\star = \nu
    \Bracket{v, a}$ and we can argue as in the second part.
\end{proof}
\begin{remark}
    \label{remark:FormalStarExpSucks}%
    We see here a notorious difficulty of formal star products: the
    derivation generating a symmetry is only quasi-inner and not
    inner. Thus a naive exponentiation of the generating element would
    lead us outside the formal power series. In fact, algebraically it
    can not be well-defined at all. Thus the symmetry $\tau_\varphi^*$
    is an \emph{outer} automorphism. We shall come back to this when
    we have some more analytic framework.
\end{remark}

We can extend the results of
Lemma~\ref{lemma:SymmetriesOfPoissonStructure} in the following way:
suppose we have two vector spaces $V$ and $W$ with two bilinear forms
$\Lambda_V$ and $\Lambda_W$ on them. Then an even linear map $A\colon
V \longrightarrow W$ is called a \emph{Poisson map} if
\begin{equation}
    \label{eq:varphiPoissonMap}
    \Lambda_W(A(v), A(v')) = \Lambda_V(v, v')
\end{equation}
for all $v, v' \in V$. The induced map $A\colon \Sym^\bullet(V)
\longrightarrow \Sym^\bullet(W)$ is then easily shown to satisfy
$P_{\Lambda_W} \circ (A \tensor A) = (A \tensor A) \circ
P_{\Lambda_V}$, generalizing the computation in the proof of
Lemma~\ref{lemma:SymmetriesOfPoissonStructure} slightly. From this we
see that $A$ is a homomorphism of Poisson algebras and star product
algebras. Thus we arrive at the following simple functoriality
statement:
\begin{proposition}
    \label{proposition:Functorial}%
    The construction of $\Bracket{\argument, \argument}$ and
    $\formalstar$ is functorial with respect to Poisson maps.
\end{proposition}

Let us now discuss how we can change the star product by changing the
\emph{symmetric} part $\Lambda_+$ of $\Lambda$ as in
\eqref{eq:AntiAndSymmetricPartOfLambda}. Symmetry means that
$\Lambda_+(v, w) = (-1)^{vw} \Lambda_+(w, v)$ for homogeneous elements
in $V$.

Let $g\colon V \times V \longrightarrow \mathbb{k}$ be another
symmetric and even bilinear form, which we can think of as a
$\mathbb{Z}_2$-graded version of an inner product. We define now a
second order ``Laplacian'' associated to $g$ as follows. For
homogeneous vectors $v_1, \ldots, v_n \in V$ we set
\begin{equation}
    \label{eq:gLaplacian}
    \Delta_g(v_1 \cdots v_n)
    =
    \sum_{i < j}
    (-1)^{v_i(v_1 + \cdots + v_{i-1})}
    (-1)^{v_j(v_1 + \cdots + v_{i-1} + v_{i+1} + \cdots v_{j-1})}
    g(v_i, v_j)
    v_1 \cdots
    \stackrel{i}{\wedge} \cdots \stackrel{j}{\wedge}
    \cdots v_n,
\end{equation}
and extend this again by linearity to an operator
\begin{equation}
    \label{eq:DeltagOperator}
    \Delta_g\colon
    \Sym^\bullet(V) \longrightarrow \Sym^{\bullet-2}(V).
\end{equation}
Note that $\Delta_g$ has even parity since $g$ vanishes on vectors of
different parities. This is no longer a derivation but a second order
differential operator. More precisely, we have the following
``Leibniz rule'' for $\Delta_g$:
\begin{lemma}
    \label{lemma:LeibnizRuleLaplacian}%
    The operator $\Delta_g$ satisfies
    \begin{equation}
        \label{eq:LeibnizRuleLaplacian}
        \Delta_g \circ \mu
        =
        \mu \circ \left(
            \Delta_g \tensor \id + P_g + \id \tensor \Delta_g
        \right).
    \end{equation}
\end{lemma}
\begin{proof}
    On $v_1 \cdots v_n \tensor w_1 \cdots w_m$ with homogeneous
    vectors $v_1, \ldots, v_n$, $w_1, \ldots, w_m \in V$ this is just
    a straightforward computation.
\end{proof}
\begin{lemma}
    \label{lemma:DeltasAndPLambdasCommute}%
    Let $\Lambda, \Lambda', g\colon V \times V \longrightarrow
    \mathbb{k}$ be even bilinear maps and let $g$ be symmetric. Then
    the operators $\Delta_g \tensor \id$, $\id \tensor \Delta_g$,
    $P_\Lambda$, and $P_{\Lambda'}$ commute pairwise.
\end{lemma}
\begin{proof}
    Again, one just checks this on $v_1 \cdots v_n \tensor w_1 \cdots
    w_m$ for homogeneous vectors $v_1, \ldots, v_n$, $w_1, \ldots, w_m
    \in V$ which is a lengthy but straightforward computation, the
    details of which we shall omit.
\end{proof}
We use these commutation relations now to prove the following
equivalence statement: the isomorphism class of the deformation
depends only on the \emph{antisymmetric} part of $\Lambda$.
\begin{proposition}
    \label{proposition:FormalEquivalence}%
    Let $\Lambda, \Lambda'\colon V \times V \longrightarrow
    \mathbb{k}$ be two even bilinear forms on $V$ such that their
    antisymmetric parts $\Lambda_- = \Lambda'_-$ coincide. Then the
    corresponding star products $\formalstar$ and
    $\star_{\nu\Lambda'}$ are equivalent via the equivalence
    transformation
    \begin{equation}
        \label{eq:FormalEqiuvalenceBetweenStarStarPrime}
        \E^{\nu \Delta_g} (a \formalstar b)
        =
        \left(\E^{\nu \Delta_g} a\right)
        \star_{\nu\Lambda'}
        \left(\E^{\nu \Delta_g} b\right)
    \end{equation}
    for all $a, b \in \Sym^\bullet(V)[[\nu]]$ where $g = \Lambda' -
    \Lambda = \Lambda_+' - \Lambda_+$.
\end{proposition}
\begin{proof}
    The proof is now fairly easy. Analogously to
    \cite[Exercise~5.7]{waldmann:2007a} we have
    \begin{align*}
        \E^{\nu \Delta_g}(a \formalstar b)
        &=
        \E^{\nu \Delta_g}
        \circ \mu
        \circ \E^{\nu P_\Lambda} (a \tensor b) \\
        &=
        \mu
        \circ
        \E^{
          \nu (\Delta_g \tensor \id + P_g + \id \tensor \Delta_g)
        }
        \circ
        \E^{\nu P_\Lambda}
        (a \tensor b) \\
        &=
        \mu \circ
        \E^{\nu (P_\Lambda + P_g)}
        \circ
        \left(
            \E^{\nu \Delta_g} \tensor \E^{\nu \Delta_g}
        \right)(a \tensor b) \\
        &=
        \mu \circ
        \E^{\nu P_{\Lambda'}}
        \circ
        \left(
            \E^{\nu \Delta_g}a \tensor \E^{\nu \Delta_g} b
        \right),
    \end{align*}
    since $P_\Lambda + P_g = P_{\Lambda + g}$ and since $\Lambda + g =
    \Lambda'$. Note that $g$ is indeed symmetric.
\end{proof}

%
% Continuity of $\starLambda$
%

\section{Continuity of the Star Product}
\label{sec:Continuity}

In this section we establish a locally convex topology on
$\Sym^\bullet(V)$ for which the formal star product, after
substituting the formal parameter by a real or complex number $z$,
will be continuous. Starting point is a locally convex topology on
$V$, which we will assume to be Hausdorff, and a continuity assumption
on $\Lambda$. From now on the field of scalars $\mathbb{K}$ is either
$\mathbb{R}$ or $\mathbb{C}$.

%
% The Topology for $\Sym^\bullet(V)$
%

\subsection{The Topology for $\Sym^\bullet(V)$}
\label{subsec:TopologyForSymV}

Let $V$ be now a real or complex $\mathbb{Z}_2$-graded Hausdorff
locally convex vector space. We require that the grading is
\emph{compatible} with the topological structure, i.e.  the
projections onto the even and odd parts in $V = V_{\mathbf{0}} \oplus
V_{\mathbf{1}}$ are continuous. Thus we have for every continuous
seminorm $\halbnorm{p}$ on $V$ another continuous seminorm
$\halbnorm{q}$ with $\halbnorm{p}(v_{\mathbf{0}}),
\halbnorm{p}(v_{\mathbf{1}}) \le \halbnorm{q}(v)$ for all $v \in V$.
This implies that the even and odd part of $V$ constitute
complementary closed subspaces.

In principle, there are many interesting locally convex topologies on
$\Sym^\bullet(V)$ induced by the one on $V$. We shall construct now a
rather particular one.

First we will endow the tensor products $V^{\tensor n}$ with the
$\pi$-topology. Recall that for seminorms $\halbnorm{p}_1, \ldots,
\halbnorm{p}_n$ on $V$ one defines the seminorm $\halbnorm{p}_1
\tensor \cdots \tensor \halbnorm{p}_n$ on $V^{\tensor n}$ by
\begin{equation}
    \label{eq:pEinscdotspnDef}
    (\halbnorm{p}_1 \tensor \cdots \tensor \halbnorm{p}_n)(v)
    =
    \inf
    \left\{
        \sum\nolimits_i
        \halbnorm{p}_1\left(v^{(1)}_i\right)
        \cdots
        \halbnorm{p}_n\left(v^{(n)}_i\right)
        \; \Big| \;
        v
        =
        \sum\nolimits_i v^{(1)}_i \tensor \cdots \tensor v^{(n)}_i
    \right\},
\end{equation}
where the infimum is taken over all possibilities to write the tensor
$v$ as a linear combination of elementary (i.e.  factorizing)
tensors. One has $(\halbnorm{p}_1 \tensor \cdots \tensor
\halbnorm{p}_n) \tensor (\halbnorm{q}_1 \tensor \cdots \tensor
\halbnorm{q}_m) = \halbnorm{p}_1 \tensor \cdots \tensor \halbnorm{p}_n
\tensor \halbnorm{q}_1 \tensor \cdots \tensor \halbnorm{q}_m$ and on
factorizing tensors one gets $(\halbnorm{p}_1 \tensor \cdots \tensor
\halbnorm{p}_n)(v_1 \tensor \cdots \tensor v_n) = \halbnorm{p}_1(v_1)
\cdots \halbnorm{p}_n(v_n)$.  We shall use the abbreviation
$\halbnorm{p}^n = \halbnorm{p} \tensor \cdots \tensor \halbnorm{p}$
for $n$ copies of the same seminorm $\halbnorm{p}$, where by
definition $\halbnorm{p}^0$ is the usual absolute value on
$\mathbb{K}$.

The $\pi$-topology on $V^{\tensor n}$ is obtained by taking all
seminorms of the form $\halbnorm{p}_1 \tensor \cdots \tensor
\halbnorm{p}_n$ with $\halbnorm{p}_1, \ldots, \halbnorm{p}_n$ being
continuous seminorms on $V$. Equivalently, one can take all
$\halbnorm{p}^n$ with $\halbnorm{p}$ being a continuous seminorm on
$V$. Analogously, one defines the $\pi$-topology for the tensor
products of different locally convex spaces. We denote the tensor
product endowed with the $\pi$-topology by $\pitensor$. It is clear
that the induced $\mathbb{Z}_2$-grading of $V^{\tensor n}$ is again
compatible with the $\pi$-topology. More explicitly, we have the
following statement:
\begin{lemma}
    \label{lemma:ZzweiGrading}
    Let $\halbnorm{p}$ be a continuous seminorm on $V$ and choose a
    continuous seminorm $\halbnorm{q}$ such that
    $\halbnorm{p}(v_{\mathbf{0}}), \halbnorm{p}(v_{\mathbf{1}}) \le
    \halbnorm{q}(v)$. Then for all $v \in V^{\tensor n}$ one has
    \begin{equation}
        \label{eq:ContinuityOfZzweiGradingPiTop}
        \halbnorm{p}^n(v_{\mathbf{0}}),
        \halbnorm{p}^n(v_{\mathbf{1}})
        \le
        \halbnorm{q}^n(v).
    \end{equation}
\end{lemma}

For concrete estimates the following simple lemma is useful: it
suffices to check estimates on factorizing tensors only:
\begin{lemma}
    \label{lemma:TestOnFactorizingTensors}%
    Let $V_1, \ldots, V_n$, $W$ be vector spaces and let $\phi\colon
    V_1 \times \cdots \times V_n \longrightarrow W$ be an $n$-linear
    map, identified with a linear map $\phi\colon V_1 \tensor \cdots
    \tensor V_n \longrightarrow W$ as usual. If $\halbnorm{p}_1$,
    \ldots, $\halbnorm{p}_n$, $\halbnorm{q}$ are seminorms on $V_1$,
    \ldots, $V_n$, $W$, respectively, such that for all $v_1 \in V_1$,
    \ldots, $v_n \in V_n$ one has
    \begin{equation}
        \label{eq:qphiEstimatep}
        \halbnorm{q}\left(\phi(v_1, \ldots, v_n)\right)
        \le
        \halbnorm{p}_1(v_1) \cdots \halbnorm{p}_n(v_n),
    \end{equation}
    then one has for all $v \in V_1 \tensor \cdots \tensor V_n$
    \begin{equation}
        \label{eq:qphiTensorpEstimate}
        \halbnorm{q}\left(\phi(v)\right)
        \le
        (\halbnorm{p}_1 \tensor \cdots \tensor \halbnorm{p}_n)(v).
    \end{equation}
\end{lemma}

Since we view the symmetric powers $\Sym^n(V)$ as subspace of
$V^{\tensor n}$ we can inherit the $\pi$-topology also for
$\Sym^n(V)$, indicated by $\piSym^n(V)$. Then we get the following
simple properties of the symmetric tensor product:
\begin{lemma}
    \label{lemma:SymmetrizerpiContinuous}%
    Let $n, m \in \mathbb{N}_0$ and let $\halbnorm{p}$ be a continuous
    seminorm on $V$.
    \begin{lemmalist}
    \item \label{item:SymmetrizerContinuous} The symmetrizer
        $\Symmetrizer_n\colon V^{\pitensor n} \longrightarrow
        V^{\pitensor n}$ is continuous and for all $v \in V^{\pitensor
          n}$ one has
        \begin{equation}
            \label{eq:pSymp}
            \halbnorm{p}^n(\Symmetrizer_n(v)) \le \halbnorm{p}^n(v).
        \end{equation}
    \item \label{item:SymClosed} $\piSym^n(V) \subseteq V^{\pitensor
          n}$ is a closed subspace.
    \item \label{item:pnPlusmSymTensor} For $v \in \Sym^n(V)$ and $w
        \in \Sym^m(V)$ one has
        \begin{equation}
            \label{eq:pnPlusmSymTensor}
            \halbnorm{p}^{n+m}(vw)
            \le
            \halbnorm{p}^n(v)\halbnorm{p}^m(w).
        \end{equation}
    \end{lemmalist}
\end{lemma}
\begin{proof}
    The first part is clear for factorizing tensors and hence
    Lemma~\ref{lemma:TestOnFactorizingTensors} applies. The second
    follows as $\Sym^n(V) = \ker(\id - \Symmetrizer_n)$ by
    definition. The third is clear from the definition and from
    \eqref{eq:pSymp}.
\end{proof}

On the tensor algebra $\Tensor^\bullet(V)$ there are at least two
canonical locally convex topologies: the Cartesian product topology
inherited from $\prod_{n=0}^\infty V^{\pitensor n}$ and the direct sum
topology which is the inductive limit topology of the finite direct
sums. While the first is very coarse, the second is very
fine. Nevertheless, both of them induce the $\pi$-topology on each
subspace $V^{\tensor n}$. We are now searching for something in
between.

We fix a parameter $R \in \mathbb{R}$ and consider for a given
continuous seminorm $\halbnorm{p}$ on $V$ the new seminorm
\begin{equation}
    \label{eq:pREinsDef}
    \halbnorm{p}_R(v)
    =
    \sum_{n=0}^\infty \halbnorm{p}^n(v_n) n!^R
\end{equation}
on the tensor algebra $\Tensor^\bullet(V)$, where we write $v =
\sum_{n=0}^\infty v_n$ as the sum of its components with fixed tensor
degree $v_n \in V^{\tensor n}$. Analogously, we define
\begin{equation}
    \label{eq:pRinftyDef}
    \halbnorm{p}_{R, \infty}(v)
    =
    \sup_{n \in \mathbb{N}_0}\left\{\halbnorm{p}^n(v_n) n!^R\right\}.
\end{equation}
In principle, we have also $\ell^p$-versions for all $p \in [1,
\infty)$, but the above two extreme cases will suffice for the
following.

The seminorms control the growth of the contributions
$\halbnorm{p}^n(v_n)$ for $n \longrightarrow \infty$ compared to a
power of $n!$ which we can view as weights from a weighted counting
measure. The choice of the factorials as weights will become clear
later. We list some first elementary properties of these seminorms.
\begin{lemma}
    \label{lemma:pREinspRinftyEstimates}%
    Let $\halbnorm{p}$ and $\halbnorm{q}$ be seminorms on $V$ and $R,
    R' \in \mathbb{R}$.
    \begin{lemmalist}
    \item \label{item:pRinftyLEpREins} One has for all $v \in
        \Tensor^\bullet(V)$
        \begin{equation}
            \label{eq:pRinftyLEpREins}
            \halbnorm{p}_{R, \infty}(v)
            \le \halbnorm{p}_R(v)
            \le 2 (2\halbnorm{p})_{R, \infty}(v).
        \end{equation}
    \item \label{item:pRRestrictToVn} Both seminorms $\halbnorm{p}_R$
        and $\halbnorm{p}_{R, \infty}$ restrict to
        $n!^R\halbnorm{p}^n$ on $V^{\tensor n}$.
    \item \label{item:pLEqPRLEqR} If $\halbnorm{q} \le \halbnorm{p}$
        then $\halbnorm{q}_R \le \halbnorm{p}_R$.
    \item \label{item:SeminormsForDifferentRs} If $R' > R$ then
        $\halbnorm{p}_R(v) \le \halbnorm{p}_{R'}(v)$ for all $v \in
        \Tensor^\bullet(V)$.
    \end{lemmalist}
\end{lemma}
\begin{proof}
    The parts \refitem{item:pRRestrictToVn},
    \refitem{item:pLEqPRLEqR}, and
    \refitem{item:SeminormsForDifferentRs} are clear. Also the first
    estimate in \refitem{item:pRinftyLEpREins} is obvious. For the
    second, we note
    \[
    \halbnorm{p}_R(v)
    =
    \sum\nolimits_{n=0}^\infty \halbnorm{p}^n(v_n) n!^R
    =
    \sum\nolimits_{n=0}^\infty
    2^n \halbnorm{p}^n(v_n) n!^R \frac{1}{2^n}
    \le
    \sup_{n \in \mathbb{N}_0} 2^n \halbnorm{p}^n(v_n) n!^{R'}
    \sum_{n=0}^\infty \frac{1}{2^n}
    \]
    which is the second estimate in \eqref{eq:pRinftyLEpREins}.
\end{proof}
The seemingly trivial first part will have an important consequence
later when we discuss the nuclearity properties of the Weyl algebra.

We can use now all the seminorms $\halbnorm{p}_R$ for a fixed $R$ to
define a new locally convex topology on the tensor algebra. In
particular, the lemma shows that we can safely restrict to the
seminorms of the type $\halbnorm{p}_R$ as long as we take \emph{all}
continuous seminorms on $V$. It is clear from an analogous estimate
that also the $\ell^p$-versions would not yield anything new.
\begin{definition}
    \label{definition:TREinsTopology}%
    Let $R \in \mathbb{R}$. Then $\Tensor^\bullet_R(V)$ is the tensor
    algebra of $V$ equipped with the locally convex topology
    determined by all the seminorms $\halbnorm{p}_R$ with
    $\halbnorm{p}$ running through all continuous seminorms on $V$.
\end{definition}
In the following, we will mainly be interested in the case of positive
$R$ where we have a \emph{decay} of the numbers $\halbnorm{p}^n(v_n)$.
\begin{lemma}
    \label{lemma:TensorProductContinuous}%
    Let $R' > R \ge 0$.
    \begin{lemmalist}
    \item \label{item:TensorProductContinuous} The tensor product is
        continuous on $\Tensor^\bullet_R(V)$. More precisely, one
        has
        \begin{equation}
            \label{eq:TensorProductContinuous}
            \halbnorm{p}_R(v \tensor w)
            \le
            (2^R\halbnorm{p})_R(v)
            (2^R\halbnorm{p})_R(w)
        \end{equation}
        for all $v, w \in \Tensor^\bullet_R(V)$.
    \item \label{item:ComponentsContinuous} For all $n \in
        \mathbb{N}_0$ the projections and the inclusions
        \begin{equation}
            \label{eq:ProjectionsInclusionContinuous}
            \Tensor^\bullet_R(V)
            \longrightarrow
            V^{\pitensor n}
            \longrightarrow
            \Tensor^\bullet_R(V)
        \end{equation}
        are continuous.
    \item \label{item:CompletionTREins} The completion
        $\cTensor^\bullet_R(V)$ of $\Tensor^\bullet_R(V)$ can
        explicitly be described by
        \begin{equation}
            \label{eq:CompletionTREins}
            \cTensor^\bullet_R(V)
            =
            \left\{
                v = \sum\nolimits_{n=0}^\infty v_n
                \; \Big| \;
                \halbnorm{p}_R(v) < \infty
                \;
                \textrm{for all}
                \;
                \halbnorm{p}
            \right\}
            \subseteq \prod_{n=0}^\infty V^{\cpitensor n},
        \end{equation}
        where $\halbnorm{p}$ runs through all continuous seminorms of
        $V$ and we extend $\halbnorm{p}_R$ to the Cartesian product by
        allowing the value $+\infty$ as usual.
    \item \label{item:TReinsInclusions} We have a continuous
        inclusion
        \begin{equation}
            \label{eq:TREinsTRInftyInclusions}
            \cTensor^\bullet_{R'}(V)
            \longrightarrow
            \cTensor^\bullet_R(V).
        \end{equation}
    \item \label{item:ZzweiGradingContinuous} The
        $\mathbb{Z}_2$-grading of $\Tensor^\bullet_R(V)$ is
        continuous. More explicitly, if $\halbnorm{p}$ and
        $\halbnorm{q}$ are continuous seminorms with
        $\halbnorm{p}(v_{\mathbf{0}}), \halbnorm{p}(v_{\mathbf{1}})
        \le \halbnorm{q}(v)$ for all $v \in V$ then we have
        \begin{equation}
            \label{eq:EstimateZzweiGradingTReins}
            \halbnorm{p}_R(v_{\mathbf{0}}),
            \halbnorm{p}_R(v_{\mathbf{1}})
            \le
            \halbnorm{q}_R(v)
        \end{equation}
        for all $v \in \Tensor^\bullet_R(V)$.
    \end{lemmalist}
\end{lemma}
\begin{proof}
    The first part is a simple estimate; we have
    \begin{align*}
        \halbnorm{p}_R(v \tensor w)
        &=
        \sum\nolimits_{k=0}^\infty
        \halbnorm{p}^k\left(
            \sum\nolimits_{n+m=k}
            v_n \tensor w_m
        \right) k!^R \\
        &\le
        \sum\nolimits_{k=0}^\infty
        \sum\nolimits_{n+m=k}
        \halbnorm{p}^n(v_n) \halbnorm{p}^m(w_m)
        (n + m)!^R \\
        &\le
        \sum\nolimits_{n=0}^\infty
        \sum\nolimits_{m=0}^\infty
        \halbnorm{p}^n(v_n) \halbnorm{p}^m(w_m)
        2^{Rn} 2^{Rm} n!^R m!^R \\
        &=
        (2^R\halbnorm{p})_R(v)
        (2^R\halbnorm{p})_R(w),
    \end{align*}
    where we used $(n+m)! \le 2^{n+m} n!m!$. Since with $\halbnorm{p}$
    also $2^R\halbnorm{p}$ is a continuous seminorm on $V$, the
    continuity of $\tensor$ follows. The second and third part are
    standard, here we use the completed $\pi$-tensor product
    $V^{\cpitensor n}$ to achieve completeness at every fixed $n \in
    \mathbb{N}_0$. The fourth part is a consequence of
    Lemma~\ref{lemma:pREinspRinftyEstimates},
    \refitem{item:SeminormsForDifferentRs}. The last part follows from
    Lemma~\ref{lemma:ZzweiGrading}.
\end{proof}
\begin{remark}
    \label{remark:FreeLMCAlgebra}%
    The case $R = 0$ gives a well-known topology on
    $\Tensor^\bullet(V)$ which becomes the \emph{free locally
      multiplicatively convex} unital algebra generated by $V$ as
    discussed e.g.  by Cuntz in \cite{cuntz:1997a}. For $R > 0$ the
    completion $\cTensor^\bullet_R(V)$ behaves differently: it
    does not even have an entire holomorphic calculus. To see this
    take the entire function
    \begin{equation}
        \label{eq:fEpsilonFunction}
        f_\epsilon(z)
        =
        \sum_{n=0}^\infty \frac{z^n}{n!^\epsilon}
    \end{equation}
    for a parameter $\epsilon > 0$. If $R > \epsilon$ then for every
    nonzero $v \in V$ the series $f_\epsilon(v)$ does not converge in
    $\cTensor^\bullet_R(V)$.  In particular, the tensor algebra
    $\Tensor^\bullet_R(V)$ can \emph{not} be locally
    multiplicatively convex unless $R = 0$.
\end{remark}

We have an analogous statement for the symmetric algebra. We equip
$\Sym^\bullet(V)$ with the induced topology from
$\Tensor^\bullet_R(V)$ and denote it by $\Sym^\bullet_R(V)$. From
\eqref{eq:pSymp} we get immediately
\begin{equation}
    \label{eq:SymmetrizerContinuous}
    \halbnorm{p}_R(\Symmetrizer(v))
    \le
    \halbnorm{p}_R(v),
\end{equation}
which implies the continuity statement
\begin{equation}
    \label{eq:ContinuityOfSymmetricTensorProduct}
    \halbnorm{p}_R(vw)
    \le
    (2^R\halbnorm{p})_R(v)
    (2^R\halbnorm{p})_R(w)
\end{equation}
for all $v, w \in \Sym^\bullet(V)$. This shows that
$\Sym^\bullet_R(V)$ becomes a locally convex algebra, too. Again, it
will be locally multiplicatively convex only for $R = 0$ in which case
it is the \emph{free locally multiplicatively convex commutative}
unital algebra generated by $V$. The completion of $\Sym^\bullet_R(V)$
can be described analogously to \eqref{eq:CompletionTREins}. We also
have the continuous inclusions
\begin{equation}
    \label{eq:SymInclusionsRinftyREins}
    \cSym^\bullet_{R'}(V)
    \longrightarrow
    \cSym^\bullet_R(V)
\end{equation}
for $R' > R$. Finally, we have the continuous projections and
inclusions
\begin{equation}
    \label{eq:SymmetricProjectionsInclusionContinuous}
    \Sym^\bullet_R(V)
    \longrightarrow
    \Sym^n_\pi(V)
    \longrightarrow
    \Sym^\bullet_R(V)
\end{equation}
for all $n \in \mathbb{N}_0$. Thus it makes sense to speak of the
$n$-th component $v_n$ of a vector $v \in \cSym^\bullet_R(V)$
even after the completion. In fact, it is easy to see that the series
of components
\begin{equation}
    \label{eq:vAsaSeriesConverges}
    v = \sum_{n=0}^\infty v_n
\end{equation}
converges to $v \in \cSym^\bullet_R(V)$, even absolutely.

To get rid of the somehow arbitrary parameter $R$ we can pass to the
projective limit $R \longrightarrow \infty$. The resulting locally
convex algebras will be denoted by
\begin{equation}
    \label{eq:ProjectiveLimitTensorRtoInfty}
    \cTensor^\bullet_\infty(V)
    =
    \projlim_{R \longrightarrow \infty}
    \cTensor^\bullet_R (V)
    \quad
    \textrm{and}
    \quad
    \cSym^\bullet_\infty(V)
    =
    \projlim_{R \longrightarrow \infty}
    \cSym^\bullet_R (V)
\end{equation}
in the symmetric case.  We have a more explicit description of
$\cTensor^\bullet_\infty(V)$ and $\cSym^\bullet_\infty(V)$ as
consisting of those formal series $v = \sum_{n=0}^\infty v_n$ with
$v_n \in V^{\cpitensor n}$ or $v_n \in \cSym^n_\pi(V)$, respectively,
such that $\halbnorm{p}_R(v) < \infty$ for \emph{all} $R \ge 0$ and
for all continuous seminorms $\halbnorm{p}$ on $V$.  Note that these
completions will be rather small as we require a rather strong decay
of the coefficients $v_n$. We will not use these projective limits in
the sequel.

%
% The Continuity of the Star Product
%

\subsection{The Continuity of the Star Product}
\label{subsec:ContinuityStarProduct}

Let us now consider an even bilinear form $\Lambda\colon V \times V
\longrightarrow \mathbb{K}$ which we require to be continuous. Thus
there exists a continuous seminorm $\halbnorm{p}$ on $V$ such that
\begin{equation}
    \label{eq:LambdaContinuous}
    |\Lambda(v, w)| \le \halbnorm{p}(v) \halbnorm{p}(w)
\end{equation}
for all $v, w \in V$. Note that we require continuity and not just
separate continuity. Note also, that if $\halbnorm{p}$ satisfies
\eqref{eq:LambdaContinuous} then we also have the estimates
\begin{equation}
    \label{eq:LambdaAntiSymContinuous}
    |\Lambda_\pm(v, w)|
    \le
    \halbnorm{p}(v) \halbnorm{p}(w),
\end{equation}
showing the continuity of the antisymmetric and symmetric part of
$\Lambda$.  The continuity of $\Lambda$ implies the continuity of the
operator $P_\Lambda$ when restricted to fixed symmetric degrees:
\begin{lemma}
    \label{lemma:ContinuityOfPLambda}%
    Let $\halbnorm{p}$ be a continuous seminorm of $V$ satisfying
    \eqref{eq:LambdaContinuous}. Then for all $u \in \Sym^n(V) \tensor
    \Sym^m(V)$ one has
    \begin{equation}
        \label{eq:ContinuityOfPLambda}
        \left(\halbnorm{p}^{n-1} \tensor \halbnorm{p}^{m-1}\right)
        \left(P_\Lambda(u)\right)
        \le
        nm \halbnorm{p}^{n+m}(u).
    \end{equation}
    The same estimate holds for $P_{\Lambda_\pm}$.
\end{lemma}
\begin{proof}
    We work on the whole tensor algebra first. Thus consider
    homogeneous vectors $v_1, \ldots, v_n$ and $w_1, \ldots, w_m \in
    V$ and define $\tilde{P}_\Lambda\colon \Tensor^\bullet(V) \tensor
    \Tensor^\bullet(V) \longrightarrow \Tensor^\bullet(V) \tensor
    \Tensor^\bullet(V)$ by the linear extension of
    \begin{equation*}
        \begin{split}
            &\tilde{P}_\Lambda
            (v_1 \tensor \cdots \tensor  v_n
            \tensor
            w_1 \tensor \cdots \tensor w_m) \\
            &\quad=
            \sum_{k=1}^n \sum_{\ell=1}^m
            (-1)^{v_k(v_{k+1} + \cdots + v_n)}
            (-1)^{w_\ell(w_1 + \cdots + w_{\ell-1})}
            \Lambda(v_k, w_\ell)
            v_1 \tensor \cdots \stackrel{k}{\wedge} \cdots \tensor v_n
            \tensor
            w_1 \tensor \cdots \stackrel{\ell}{\wedge} \cdots \tensor w_m.
        \end{split}
    \end{equation*}
    Then we have $P_\Lambda \circ (\Symmetrizer_n \tensor
    \Symmetrizer_m) = (\Symmetrizer_{n-1} \tensor \Symmetrizer_{m-1})
    \circ \tilde{P}_\Lambda$. For general tensors of arbitrary degree
    this yields
    \begin{equation*}
        P_\Lambda \circ \left(\Symmetrizer \tensor \Symmetrizer\right)
        =
        \left(\Symmetrizer \tensor \Symmetrizer\right)
        \circ \tilde{P}_\Lambda.
        \tag{$*$}
    \end{equation*}
    For homogeneous vectors we get now the estimate
    \begin{align*}
        &\left(\halbnorm{p}^{n-1} \tensor \halbnorm{p}^{m-1}\right)
        \left(
            \tilde{P}_\Lambda
            \left(
                v_1 \tensor \cdots \tensor v_n
                \tensor
                w_1 \tensor \cdots \tensor w_m
            \right)
        \right)\\
        &\quad\le
        \sum_{k=1}^n \sum_{\ell=1}^m
        |\Lambda(v_k, w_\ell)|
        \halbnorm{p}(v_1)
        \cdots \stackrel{k}{\wedge} \cdots
        \halbnorm{p}(v_n)
        \halbnorm{p}(w_1)
        \cdots \stackrel{\ell}{\wedge} \cdots
        \halbnorm{p}(w_m) \\
        &\quad\le
        nm
        \halbnorm{p}(v_1) \cdots \halbnorm{p}(v_n)
        \halbnorm{p}(w_1) \cdots \halbnorm{p}(w_m) \\
        &\quad=
        nm
        \left(\halbnorm{p}^n \tensor \halbnorm{p}^m\right)
        \left(
            v_1 \tensor \cdots \tensor v_n
            \tensor
            w_1 \tensor \cdots \tensor w_m
        \right).
    \end{align*}
    By Lemma~\ref{lemma:TestOnFactorizingTensors} we conclude that for
    all $u \in \Tensor^n(V) \tensor \Tensor^m(V)$ we have
    \[
    \left(\halbnorm{p}^{n-1} \tensor \halbnorm{p}^{m-1}\right)
    \left(
        \tilde{P}_\Lambda(u)
    \right)
    \le
    nm
    \left(\halbnorm{p}^n \tensor \halbnorm{p}^m\right)(u).
    \]
    Finally, for $u \in \Sym^n(V) \tensor \Sym^m(V)$ we have
    $(\Symmetrizer_n \tensor \Symmetrizer_m) (u) = u$ and thus
    by ($*$) and \eqref{eq:pSymp}
    \begin{align*}
        \left(\halbnorm{p}^{n-1} \tensor \halbnorm{p}^{m-1}\right)
        \left(P_\Lambda(u)\right)
        &=
        \left(\halbnorm{p}^{n-1} \tensor \halbnorm{p}^{m-1}\right)
        \left(P_\Lambda
            (\Symmetrizer_n \tensor \Symmetrizer_m) (u)
        \right) \\
        &=
        \left(\halbnorm{p}^{n-1} \tensor \halbnorm{p}^{m-1}\right)
        \left(
            (\Symmetrizer_{n-1} \tensor \Symmetrizer_{m-1})
            \tilde{P}_\Lambda(u)
        \right) \\
        &\le
        \left(\halbnorm{p}^{n-1} \tensor \halbnorm{p}^{m-1}\right)
        \left(\tilde{P}_\Lambda(u)\right) \\
        &\le
        nm
        \left(\halbnorm{p}^n \tensor \halbnorm{p}^m\right)(u).
    \end{align*}
    The last statement follows from
    \eqref{eq:LambdaAntiSymContinuous}.
\end{proof}
In fact, this estimate just reflects the fact that $P_\Lambda$ is a
biderivation. If $\Lambda$ is nontrivial then it can not be improved
in general. It implies immediately the continuity of the Poisson
bracket $\Bracket{\argument, \argument}$:
\begin{proposition}
    \label{proposition:ContinuityOfPossionBracket}%
    Let $\Lambda$ be continuous. Then the Poisson bracket
    $\{\argument, \argument\}_{\Lambda}$ is continuous on
    $\Sym^\bullet_R(V)$ for every $R \ge 0$. More precisely, for $v, w
    \in \Sym^\bullet_R(V)$ and any continuous seminorm $\halbnorm{p}$
    on $V$ with \eqref{eq:LambdaContinuous} we have a constant $c > 0$
    such that
    \begin{equation}
        \label{eq:ContinuityPoissonBracket}
        \halbnorm{p}_R\left(
            \{v, w\}_{\Lambda}
        \right)
        \le
        \left(2^{R+1}\halbnorm{p}\right)_R(v)
        \left(2^{R+1}\halbnorm{p}\right)_R(w).
    \end{equation}
\end{proposition}
\begin{proof}
    Let $v, w \in \Sym^\bullet_R(V)$ with components $v = \sum_n
    v_n$ and $w = \sum_m w_m$ as usual. Then we have for a seminorm
    $\halbnorm{p}$ satisfying \eqref{eq:LambdaContinuous}
    \begin{align*}
        \halbnorm{p}_R\left(\Bracket{v, w}\right)
        &=
        \sum_{k=0}^\infty
        \halbnorm{p}^k\left(
            \sum_{n+m-2 = k} \Bracket{v_n, w_m}
        \right)
        k!^R \\
        &\le
        \sum_{k=0}^\infty
        \sum_{n+m-2 = k}
        \halbnorm{p}^{n+m-2}
        \left(
            2 \mu \circ P_{\Lambda_-}(v_n \tensor w_m)
        \right)
        k!^R
        \\
        &\le
        \sum_{k=0}^\infty
        \sum_{n+m-2 = k}
        2 nm
        \halbnorm{p}^n(v_n) \halbnorm{p}^m(w_m)
        (n+m-2)!^R \\
        &\le
        \sum_{n, m = 0}^\infty
        2 nm
        2^{(n+m)R}
        \halbnorm{p}^n(v_n) \halbnorm{p}^m(w_m)
        n!^Rm!^R \\
        &\le
        c
        \sum_{n, m = 0}^\infty
        2^{(n+m)(R+1)}
        \halbnorm{p}^n(v_n) \halbnorm{p}^m(w_m)
        n!^R m!^R ,
    \end{align*}
    where we have used Lemma~\ref{lemma:ContinuityOfPLambda} for
    $\Lambda_-$ and the standard estimate $(n+m)! \le 2^{n+m} n!m!$.
\end{proof}
In this sense, $\Sym^\bullet_R(V)$ becomes a \emph{locally convex
  Poisson algebra} for every $R \ge 0$. In particular, the Poisson
bracket extends to the completion $\cSym^\bullet_R(V)$ and still
obeys the continuity estimate \eqref{eq:ContinuityPoissonBracket} as
well as the algebraic properties of a Poisson bracket.

However, for the star product the situation is more complicated: first
we note that thanks to Corollary~\ref{corollary:NonFormalDeformation}
we can use the formal star product to get a well-defined
\emph{non-formal} star product $\starzLambda$ by replacing $\nu$ with
some real or complex number $z \in \mathbb{K}$, depending on our
choice of the underlying field. We fix $z$ in the following and
consider the dependence of $\starzLambda$ on $z$ later in
Subsection~\ref{subsec:DependenceOnz}. The next lemma provides the key
estimate for all continuity properties of $\starzLambda$: the main
point is that we have to limit the range of the possible values of
$R$:
\begin{lemma}
    \label{lemma:KeyLemma}%
    Let $R \ge \frac{1}{2}$. Then there exist constants $c, c' > 0$ such
    that for all $a, b \in \Sym^\bullet(V)$ and all seminorms
    $\halbnorm{p}$ with \eqref{eq:LambdaContinuous} we have
    \begin{equation}
        \label{eq:KeyEstimate}
        \halbnorm{p}_R(a \starzLambda b)
        \le
        c'
        (c\halbnorm{p})_R(a)
        (c\halbnorm{p})_R(b).
    \end{equation}
\end{lemma}
\begin{proof}
    Let $a, b \in \Sym^\bullet(V)$ be given and denote by $a_n$, $b_m$
    their homogeneous parts with respect to the tensor degree as
    usual. We have to distinguish several cases of the parameters. The
    non-trivial case is for $\frac{1}{2} \le R \le 1$ and $|z| \ge 1$,
    where we estimate
    \begin{align*}
        &\halbnorm{p}_R(a \starzLambda b)
        \le
        \sum_{k=0}^\infty \frac{|z|^k}{k!}
        \halbnorm{p}_R \left(
            \mu \circ P_\Lambda^k(a \tensor b)
        \right) \\
        &\quad\le
        \sum_{k=0}^\infty \frac{|z|^k}{k!}
        \sum_{k \le n, m}
        (n + m - 2k)!^R
        \halbnorm{p}^{n + m - 2k}
        \left(
            \mu \circ P_\Lambda^k (a_n \tensor b_m)
        \right) \\
        &\quad\stackrel{(a)}{\le}
        \sum_{k = 0}^\infty
        \frac{|z|^k}{k!}
        \sum_{k \le n, m}
        (n + m - 2k)!^R
        \frac{n!}{(n-k)!} \frac{m!}{(m-k)!}
        \halbnorm{p}^n(a_n) \halbnorm{p}^m(b_m) \\
        &\quad\stackrel{(b)}{\le}
        \sum_{k= 0}^\infty
        \sum_{k \le n, m}
        \frac{|z|^k 2^{R(n+m-2k)}}{k!}
        \frac{n!^{1-R}}{(n-k)!^{1-R}}
        \frac{m!^{1-R}}{(m-k)!^{1-R}}
        n!^R \halbnorm{p}^n(a_n) m!^R \halbnorm{p}^m(b_m) \\
        &\quad\stackrel{(c)}{\le}
        \sum_{k= 0}^\infty
        \sum_{k \le n, m}
        \frac{|z|^k 2^{R(n+m-2k)}}{k!}
        2^{(1-R)n} 2^{(1-R)m}
        k!^{2(1-R)}
        n!^R \halbnorm{p}^n(a_n)
        m!^R \halbnorm{p}^m(b_m) \\
        &\quad\stackrel{(d)}{\le}
        \sum_{k= 0}^\infty
        \sum_{k \le n, m}
        \frac{1}{|z|^k k!^{2R -1} 2^{2Rk}}
        n!^R (2|z|)^n \halbnorm{p}^n(a_n)
        m!^R (2|z|)^m \halbnorm{p}^m(b_m) \\
        &\quad\stackrel{(e)}{\le}
        \left(
            \sum_{k = 0}^\infty
            \frac{1}{|z|^k k!^{2R -1} 2^{2Rk}}
        \right)
        \left(
            \sum_{n=0}^\infty
            n!^R (2|z|)^n \halbnorm{p}^n(a_n)
        \right)
        \left(
            \sum_{m=0}^\infty
            m!^R (2|z|)^m \halbnorm{p}^m(b_m)
        \right) \\
        &\quad=
        c'
        (2|z|\halbnorm{p})_R(a)
        (2|z|\halbnorm{p})_R(b)
    \end{align*}
    where in ($a$) we used $k$-times the estimate from
    Lemma~\ref{lemma:KeyLemma}, in ($b$) we used $(n + m -2k)! \le
    2^{n+m-2k} (n-k)!(m-k)!$, in ($c$) we used $n! \le 2^n (n-k)!k!$
    and $m! \le 2^m (m-k)! k!$ together with the assumption $R \le 1$,
    in ($d$) we use $|z| \ge 1$ as well as $k \le n, m$, and finally,
    in ($e$) we note that the series over $k$ converges to a constant
    thanks to $R \ge \frac{1}{2}$ while the remaining series over $n$
    and $m$ give the seminorm $(2|z|\halbnorm{p})_R$ for the
    rescaled seminorm $2 |z|\halbnorm{p}$. If instead $|z| < 1$ then
    we continue instead of ($d$) by
    \[
    \halbnorm{p}_R(a \starzLambda b)
    \le \ldots
    \stackrel{(d')}{\le}
    \left(
        \sum_{k=0}^\infty \frac{|z|^k 2^{-2Rk}}{k!^{2R -1}}
    \right)
    (2\halbnorm{p})_R(a)
    (2\halbnorm{p})_R(b),
    \]
    where the series over $k$ always converges since $R \ge
    \frac{1}{2}$ and $|z| < 1$. Finally, if $R > 1$ then we continue
    instead of ($c$) with
    \begin{align*}
        \halbnorm{p}_R(a \starzLambda b)
        \le \ldots &\stackrel{(c')}{\le}
        \sum_{k= 0}^\infty
        \sum_{k \le n, m}
        \frac{|z|^k 2^{-2Rk}}{k!}
        n!^R 2^{Rn}\halbnorm{p}^n(a_n)
        m!^R 2^{Rm}\halbnorm{p}^m(b_m) \\
        &\le
        \left(
            \sum_{k=0}^\infty
            \frac{|z|^k 2^{-2Rk}}{k!}
        \right)
        (2^R\halbnorm{p})_R(a)
        (2^R\halbnorm{p})_R(b),
    \end{align*}
    where again the series over $k$ converges. In total, we always get
    an estimate for all $R \ge \frac{1}{2}$ and all $z$ as claimed.
\end{proof}
\begin{remark}
    \label{remark:REinhalbIstScharf}%
    We also note that the limiting case $R = \frac{1}{2}$ is sharp in
    the following sense: consider the most simple nontrivial situation
    $V = \mathbb{R}^2$ with basis vectors $q$ and $p$ as well as the
    bilinear form $\Lambda_{\scriptscriptstyle \mathrm{std}}(p, q) =
    1$ and zero for the other combinations. The corresponding Poisson
    bracket is the canonical Poisson bracket and the star product is
    the \emph{standard-ordered star product} if we take $z =
    \frac{\hbar}{\I}$, see
    e.g. \cite[Sect.~5.2.4]{waldmann:2007a}. Identifying elements in
    $\Sym^\bullet(\mathbb{R}^2)$ with polynomials in $q$ and $p$ we
    have the more explicit formula
    \begin{equation}
        \label{eq:StandardOrderedProduct}
        f \star_{\scriptscriptstyle\mathrm{std}} g
        =
        \sum_{k=0}^\infty
        \frac{(-\I\hbar)^k}{k!}
        \frac{\partial^k f}{\partial p^k}
        \frac{\partial^k g}{\partial q^k}.
    \end{equation}
    Using again the function $f_\epsilon$ from
    \eqref{eq:fEpsilonFunction} we see that $f_\epsilon(q)$ and
    $f_\epsilon(p)$ belong to $\cSym^\bullet_R(\mathbb{R}^2)$ as
    soon as $R < \epsilon$. However, for the star product we get
    (formally)
    \begin{equation}
        \label{eq:fepsStarfeps}
        f_\epsilon(p)
        \star_{\scriptscriptstyle\mathrm{std}}
        f_\epsilon(q)
        =
        \sum_{n, m, k \le m, n}
        \frac{(-\I\hbar)^k}{k!}
        \frac{n!^{1-\epsilon}}{(n-k)!}
        \frac{n!^{1-\epsilon}}{(n-k)!}
        q^{n-k} p^{n-k}.
    \end{equation}
    Since the projection $\cSym^\bullet_R(V) \longrightarrow
    \piSym^n(V)$ is continuous for all $R \ge 0$ we consider the
    coefficient of \eqref{eq:fepsStarfeps} in $\Sym^0(V)$ which is
    obtained for $n = m = k$, i.e.
    \begin{equation}
        \label{eq:DivergentZeroCoeff}
        \sum_{\ell = 0}^\infty \frac{(-\I\hbar)^\ell}{\ell!}
        \ell!^{1-\epsilon} \ell!^{1-\epsilon}
        =
        \sum_{\ell = 0}^\infty (-\I\hbar)^\ell \ell!^{1-2\epsilon}.
    \end{equation}
    This clearly diverges for $\epsilon < \frac{1}{2}$ unless $\hbar =
    0$. Thus for $R < \frac{1}{2}$ we can not expect a continuous star
    product.
\end{remark}

%
% The Weyl Algebra $\WeylR(V, \starzLambda)$
%

\subsection{The Weyl Algebra $\WeylR(V, \starzLambda)$}
\label{subsec:WeylAlgebra}

The estimate from Lemma~\ref{lemma:KeyLemma} shows that the star
product will be continuous for the topology of $\Sym^\bullet_R(V)$
\emph{provided} the parameter $R$ satisfies $R \ge \frac{1}{2}$. This
will motivate the following definition of the Weyl algebra. However,
we will also give an alternative definition for later use, where we
want to compare with the results from
\cite{beiser.roemer.waldmann:2007a, beiser.waldmann:2011a:pre}.
\begin{definition}[Weyl algebra]
    \label{definition:WeylAlgebraWR}%
    For $R \in \mathbb{R}$ we endow $\Sym^\bullet_R(V)$ with the
    product $\starzLambda$ and call the resulting algebra the Weyl
    algebra $\WeylR(V, \starzLambda)$. Moreover, for $R > \frac{1}{2}$
    we set
    \begin{equation}
        \label{eq:WeylProjective}
        \WeylRMinus(V) = \projlim_{\epsilon \longrightarrow 0}
        \Sym^\bullet_{R - \epsilon}(V)
    \end{equation}
    and endow $\WeylRMinus(V)$ with the Weyl product $\starzLambda$,
    too.
\end{definition}
This way, we arrive at two possible definitions of the Weyl
algebra. The projective limit can be made more explicitly, since the
underlying vector space is always the same: we use \emph{all}
seminorms $\halbnorm{p}_{R-\epsilon}$ for $\epsilon > 0$ and
$\halbnorm{p}$ a continuous seminorm on $V$ for $\WeylRMinus(V)$ and
have $\WeylRMinus(V) = \Sym^\bullet(V)$ as a linear space as before.
It will turn out that this projective limit enjoys some more
interesting properties when it comes to strong nuclearity.

The completion $\cWeylR(V)$ will be given as those formal series $v =
\sum_{n = 0}^\infty v_n$ with $v_n \in \cpiSym^n(V)$ such that
\emph{all} seminorms $\halbnorm{p}_R(v)$ are finite for all continuous
seminorms $\halbnorm{p}$ on $V$. Correspondingly, for the completion
of the projective limit we have to have finite seminorms
$\halbnorm{p}_{R-\epsilon}(v)$ for all $\epsilon > 0$ and all
continuous seminorms $\halbnorm{p}$ on $V$.

A last option is to take the projective limit $R \longrightarrow
\infty$. Most of the following statements will therefore also be
available for the case $R = \infty$. However, we will not be too much
interested in this case as the completion $\cWeylInfty(V) =
\cSym^\bullet_\infty(V)$ is rather small.
\begin{remark}
    \label{remark:TensorRDef}%
    For later use, we can endow also the tensor algebra with the
    projective topology of all the seminorms $\halbnorm{p}_{R -
      \epsilon}$ for a fixed $R \in \mathbb{R}$ and all $\epsilon >
    0$. The resulting tensor algebra will be denoted by
    $\Tensor^\bullet_{R^-}(V)$. Then $\WeylRMinus(V) \subseteq
    \Tensor^\bullet_{R^-}(V)$ is a closed subspace.
\end{remark}
\begin{remark}
    \label{remark:PoissonBracketContinuousForWeylTop}%
    From Proposition~\ref{proposition:ContinuityOfPossionBracket} it
    follows immediately that the Poisson bracket $\Bracket{\argument,
      \argument}$ is still continuous for the projective limit
    topology. Thus $\WeylRMinus(V)$ becomes a locally convex Poisson
    algebra.
\end{remark}

We start now to collect some basic features of $\WeylRMinus(V)$. From
Lemma~\ref{lemma:pREinspRinftyEstimates} we get immediately the
following statement:
\begin{lemma}
    \label{lemma:ComponentsAreContinuous}%
    Let $R' \ge R \ge 0$.
    \begin{lemmalist}
    \item \label{item:InducedTopologyOnSn} For all $n \in
        \mathbb{N}_0$ the induced topology on $\Sym^n(V) \subseteq
        \WeylRMinus(V)$ is the $\pi$-topology.
    \item \label{item:ComponentsAreContinuous} The projection and the
        inclusion maps
        \begin{equation}
            \label{eq:WeylSnWeyl}
            \WeylRMinus(V)
            \longrightarrow
            \piSym^n(V)
            \longrightarrow
            \WeylRMinus(V)
        \end{equation}
        are continuous for all $n \in \mathbb{N}_0$.
    \item \label{item:DifferentWeylRIntoWeylR} The inclusion map
        $\mathcal{W}_{R'^-}(V) \longrightarrow \WeylRMinus(V)$ is
        continuous.
    \item \label{item:ZzweiStetigFuerWeyl} The $\mathbb{Z}_2$-grading
        is continuous for $\WeylRMinus(V)$.
    \end{lemmalist}
\end{lemma}
The analogous statements for $\WeylR(V)$ hold for trivial reasons: we
have discussed them already for the symmetric algebra
$\Sym^\bullet_R(V)$.

This lemma has the important consequence that also after completion of
$\WeylRMinus(V)$ to $\cWeylRMinus(V)$ we can speak of the $n$-th
component $a_n \in \cpiSym^n(V)$ of an element $a \in \cWeylRMinus(V)$
in a meaningful way. More precisely, $a$ can be expressed as a
convergent series in its components of fixed tensor degree:
\begin{lemma}
    \label{lemma:aInWeylRAsSuman}%
    Let $R \in \mathbb{R}$ and let $a \in \cWeylRMinus(V)$ with
    components $a_n \in \cpiSym^n(V)$ for $n \in \mathbb{N}_0$. Then
    \begin{equation}
        \label{eq:aSeriesan}
        a = \sum_{n=0}^\infty a_n
    \end{equation}
    converges absolutely.
\end{lemma}
\begin{proof}
    Identifying $a_n \in \cpiSym^n(V)$ with its image in
    $\cWeylRMinus(V)$ we get for every continuous seminorm
    $\halbnorm{p}$ on $V$ the equation
    \[
    \halbnorm{p}_{R-\epsilon}(a)
    =
    \sum_{n=0}^\infty \halbnorm{p}^n(a_n) n!^{R-\epsilon}
    =
    \sum_{n=0}^\infty \halbnorm{p}_{R-\epsilon}(a_n),
    \]
    from which the statement follows immediately.
\end{proof}

In particular, the direct sum $\bigoplus_{n=0}^\infty \cpiSym^n(V)$ of
the completed symmetric $\pi$-tensor powers of $V$ is
\emph{sequentially} dense in $\cWeylRMinus(V)$. Again, this statement
is also true for the case $\cWeylR(V)$, see \eqref{eq:vAsaSeriesConverges}.

The first main result is now that $\WeylR(V, \starzLambda)$ as well as
$\WeylRMinus(V, \starzLambda)$ are indeed locally convex algebras
provided $R$ is suitably chosen:
\begin{theorem}
    \label{theorem:WeylAlgebra}%
    Let $R \ge \frac{1}{2}$.  The Weyl algebra $\WeylR(V,
    \starzLambda)$ is a locally convex algebra. Moreover,
    $\WeylRMinus(V, \starzLambda)$ is a locally convex algebra for $R
    > \frac{1}{2}$. In both cases, the Weyl algebra is first countable
    iff $V$ is first countable.
\end{theorem}
\begin{proof}
    The continuity of the product for $\WeylR(V)$ is just
    Lemma~\ref{lemma:KeyLemma}. This gives also the continuity in the
    case of $\WeylRMinus(V)$.  Note that we need $R > \frac{1}{2}$ for
    the second case. A locally convex space $V$ is first countable iff
    we can choose a sequence of continuous seminorms
    $\halbnorm{p}^{(1)} \le \halbnorm{p}^{(2)} \le \cdots$ such that
    for every other continuous seminorm $\halbnorm{q}$ on $V$ we have
    some $n \in \mathbb{N}$ with $\halbnorm{q} \le
    \halbnorm{p}^{(n)}$. Then Lemma~\ref{lemma:pREinspRinftyEstimates}
    shows that the seminorms $\halbnorm{p}^{(n)}_R$ will do the job
    for $\WeylR(V)$ while for $\WeylRMinus(V)$ we can take the
    seminorms $\left(\halbnorm{p}^{(n)}\right)_{R - \frac{1}{n}}$ to
    determine the topology. The converse is obvious from
    Lemma~\ref{lemma:ComponentsAreContinuous},
    \refitem{item:InducedTopologyOnSn}.
\end{proof}
\begin{remark}
    \label{remark:FrechetIsOptimal}%
    Note that for a Banach space $V$ a Fréchet algebra $\cWeylR(V,
    \starzLambda)$ is the best we can hope for in general since the
    canonical commutation relations $[q,
    p]_{\star_{\scriptscriptstyle\mathrm{std}}} = \I \hbar \Unit$ can
    not be implemented in a Banach algebra for $\hbar \ne 0$. So if
    the even part of $V$ and $\Lambda\at{V_{\mathbf{0}}}$ are
    nontrivial, a Banach algebra structure is excluded since the
    standard-ordered star product from
    Remark~\ref{remark:REinhalbIstScharf} comprises always a
    subalgebra in this case. Even worse, in this case every
    submultiplicative seminorm is necessarily trivial: $\cWeylR(V,
    \starzLambda)$ is very far from being locally multiplicatively
    convex.
\end{remark}

As a first application we show that in the completion $\cWeylR(V)$ we
have exponentials of every vector in $V$, provided $R$ is small
enough:
\begin{proposition}
    \label{proposition:WeylAlgebraHasExp}%
    Assume $V_{\mathbf{0}} \ne \{0\}$.
    \begin{propositionlist}
    \item \label{item:expIffRkleinerEins} One has $\exp(v) \in
        \cWeylR(V)$ for every non-zero $v \in V$ iff $R < 1$.
    \item \label{item:expEntire} Let $R < 1$ and $v \in V$. The map
        $\mathbb{K} \ni t \mapsto \exp(t v) \in \cWeylR(V)$ is
        real-analytic in the case $\mathbb{K} = \mathbb{R}$ with
        radius of convergence $\infty$ and entire in the case
        $\mathbb{K} = \mathbb{C}$. The Taylor series converges
        absolutely.
    \end{propositionlist}
\end{proposition}
\begin{proof}
    Let $v \in V_{\mathbf{0}}$ be non-zero and choose a seminorm
    $\halbnorm{p}$ with $\halbnorm{p}(v) > 1$ which is possible thanks
    to the Hausdorff property and by an appropriate rescaling of
    $v$. Then $\halbnorm{p}^n(v^n) = (\halbnorm{p}(v))^n$ since $v
    \tensor \cdots \tensor v = v \cdots v$ in this case. The
    exponential series therefore gives
    \[
    \halbnorm{p}_R(\exp(v))
    =
    \sum_{n=0}^\infty \frac{\halbnorm{p}^n(v^n)}{n!} n!^R
    =
    \sum_{n=0}^\infty (\halbnorm{p}(v))^n n!^{R - 1},
    \]
    which converges iff $R < 1$, showing the first part. The second
    part is clear from Lemma~\ref{lemma:aInWeylRAsSuman} since the
    homogeneous components of $\exp(t v)$ are given by $\frac{t^n
      v^v}{n!}$.
\end{proof}
\begin{remark}
    \label{remark:WeylMinusExp}%
    In the projective case we do not have to exclude $R = 1$: here the
    analogous statement is that $\exp(v) \in \cWeylRMinus(V)$ for
    every non-zero $v \in V$ iff $R \le 1$. The reason is that we have
    to use the seminorms $\halbnorm{p}_{R - \epsilon}$ instead.  Note
    also that in the case where $V_{\mathbf{0}} = \{0\}$ the
    exponential series of all elements $v \in V$ always converges
    since $\exp(v) = 1 + v$ for odd vectors $v \in V_{\mathbf{1}}$.
\end{remark}

%
% Dependence on $z$
%

\subsection{Dependence on $z$}
\label{subsec:DependenceOnz}

Up to now we have established the continuity of $\starzLambda$ on
$\WeylR(V)$ and thus we can conclude that $\starzLambda$ has a unique
extension to a continuous product $\starzLambda$ on the completion
$\cWeylR(V)$. We shall now re-interpret the proof of
Lemma~\ref{lemma:KeyLemma} to get the more specific statement that
also the formula for $\starzLambda$ stays valid:
\begin{proposition}
    \label{proposition:StarAbsConvergent}%
    Let $R \ge \frac{1}{2}$ and let $a, b \in \cWeylR(V,
    \starzLambda)$. Then
    \begin{equation}
        \label{eq:TheWeylProductConverges}
        a \starzLambda b
        =
        \mu \circ \E^{z P_\Lambda}(a \tensor b)
        =
        \sum_{k=0}^\infty
        \frac{z^k}{k!} \mu \circ P_\Lambda^k (a \tensor b)
    \end{equation}
    converges absolutely in $\cWeylR(V)$. The same statement holds for
    the projective limit version $\cWeylRMinus(V, \starzLambda)$ and
    $R > \frac{1}{2}$.
\end{proposition}
\begin{proof}
    We have to show that for all seminorms $\halbnorm{p}_R$ of the
    defining system of seminorms the series $\sum_{k=0}^\infty
    \frac{|z|^k}{k!}  \halbnorm{p}_R(\mu \circ P_\Lambda^k(a \tensor
    b))$ converges. But this was exactly what we did in the proof of
    Lemma~\ref{lemma:KeyLemma}. The projective limit case is
    analogous.
\end{proof}

This proposition also allows us to discuss the dependence on the
deformation parameter $z$: here we have the best possible scenario. In
the real case, $z \in \mathbb{R}$, we have a real-analytic dependence
on $z$ in $a \starzLambda b$ with an explicit Taylor expansion around
$z = 0$ given by the absolutely convergent series
\eqref{eq:TheWeylProductConverges}. In the complex case, $z \in
\mathbb{C}$, we have an entire dependence, again by
\eqref{eq:TheWeylProductConverges}. Note that it is important for such
statements that the topology of the Weyl algebra is actually
independent of $z$. Holomorphic deformations were introduced and
studied in detail in \cite{pflaum.schottenloher:1998a}, mainly in the
context of Hopf algebra deformations.
\begin{proposition}
    \label{proposition:HolomorphicDeformation}%
    Let $R \ge \frac{1}{2}$.
    \begin{propositionlist}
    \item \label{item:RealAnalyticDeformation} If $\mathbb{K} =
        \mathbb{R}$ then for every $a, b \in \cWeylR(V, \starzLambda)$
        the map
        \begin{equation}
            \label{eq:RealAnalyticDeformation}
            \mathbb{R} \ni z
            \; \mapsto \;
            a \starzLambda b \in \cWeylR(V, \starzLambda)
        \end{equation}
        is real-analytic with Taylor expansion around $z = 0$ given by
        \eqref{eq:TheWeylProductConverges}.
    \item \label{item:HolomorphicDeformation} If $\mathbb{K} =
        \mathbb{C}$ then for every $a, b \in \cWeylR(V, \starzLambda)$
        the map
        \begin{equation}
            \label{eq:EntireDeformation}
            \mathbb{C} \ni z
            \; \mapsto \;
            a \starzLambda b \in \cWeylR(V, \starzLambda)
        \end{equation}
        is holomorphic (even entire) with Taylor expansion around $z =
        0$ given by \eqref{eq:TheWeylProductConverges}. The collection
        of Weyl algebras $\{\cWeylR(V, \starzLambda)\}_{z \in
          \mathbb{C}}$ provides a holomorphic (even entire)
        deformation of $\cWeylR(V, \mu)$, where $\mu =
        \starzLambda\at{z=0}$ is the symmetric tensor product.
    \end{propositionlist}
\end{proposition}

%
% Reality and the $^*$-Involution
%

\subsection{Reality and the $^*$-Involution}
\label{subsec:RealityInvolution}

Up to now we treated the real and complex case on equal
footing. However, for applications in physics one typically needs an
additional structure, both for the classical Poisson algebra as well
as for the quantum algebra: a reality structure in form of a
$^*$-involution.

The following two structures are well-known to be equivalent. We
recall their relation in order to establish some notation: either we
can start with a real vector space $V_{\mathbb{R}}$ and complexify it
to $V_{\mathbb{C}} = V_{\mathbb{R}} \tensor \mathbb{C}$. This gives us
an antilinear involutive automorphism of $V_{\mathbb{C}}$, the complex
conjugation, denoted by $v_{\mathbb{R}} \tensor z \mapsto
\cc{v_{\mathbb{R}} \tensor z} = v_{\mathbb{R}} \tensor \cc{z}$ for
$v_{\mathbb{R}} \in V_{\mathbb{R}}$ and $z \in \mathbb{C}$. We can
recover $V_{\mathbb{R}}$ as the real subspace of $V_{\mathbb{C}}$ of
those vectors $v \in V_{\mathbb{C}}$ with $\cc{v} = v$. Or,
equivalently, we can start with a complex vector space
$V_{\mathbb{C}}$ and an antilinear involutive automorphism, still
denoted by $v \mapsto \cc{v}$. Then $V_{\mathbb{C}} \cong
V_{\mathbb{R}} \tensor \mathbb{C}$ with $V_{\mathbb{R}}$ consisting
again of the real vectors in $V_{\mathbb{C}}$. In this situation the
symmetric algebra $\Sym^\bullet(V_{\mathbb{C}})$ is a $^*$-algebra
with respect to the complex conjugation, i.e.  we have for homogeneous
$a, b \in \Sym^\bullet(V_{\mathbb{C}})$
\begin{equation}
    \label{eq:ccabInvolution}
    \cc{ab} = (-1)^{ab} \cc{b} \, \cc{a} = \cc{a} \, \cc{b},
\end{equation}
where in the second equation we use the commutativity of the symmetric
tensor product.
\begin{remark}
    \label{remark:NotQuiteTheRightInvolution}%
    There are various reasons why this might not be what one really
    wants in the case of $\mathbb{Z}_2$-graded algebras. Instead, an
    honest $^*$-involution \emph{without} signs might be more
    desirable, i.e.  $(ab)^* = b^*a^*$ for all $a$ and $b$, no matter
    what parity they have. However, this requires some extra structure
    which we will not discuss in the sequel.
\end{remark}

If in addition $V_{\mathbb{R}}$ is locally convex we can extend a
continuous seminorm $\halbnorm{p}_{\mathbb{R}}$ on $V_{\mathbb{R}}$ to
a seminorm $\halbnorm{p}_{\mathbb{C}}$ on $V_{\mathbb{C}}$ by setting
$\halbnorm{p}_{\mathbb{C}}(v \tensor z) =
|z|\halbnorm{p}_{\mathbb{R}}(v)$. This makes $V_{\mathbb{C}}$ a
locally convex space such that the complex conjugation is
continuous. In fact, $\halbnorm{p}_{\mathbb{C}}(\cc{v}) =
\halbnorm{p}_{\mathbb{C}}(v)$ for the seminorms of the form
$\halbnorm{p}_{\mathbb{C}}$. Conversely, if $V_{\mathbb{C}}$ is a
locally convex complex vector space with a continuous complex
conjugation then for every continuous seminorm $\halbnorm{q}$ also
$\halbnorm{p}(v) = \frac{1}{2}(\halbnorm{q}(v) +
\halbnorm{q}(\cc{v}))$ is continuous, now satisfying $\halbnorm{p}(v)
= \halbnorm{p}(\cc{v})$. Clearly, these seminorms still determine the
topology of $V_{\mathbb{C}}$. Finally, for $\halbnorm{p}_{\mathbb{R}}
= \halbnorm{p}\at{V_{\mathbb{R}}}$ we get
$(\halbnorm{p}_{\mathbb{R}})_{\mathbb{C}} = \halbnorm{p}$. Thus also
in the locally convex situation the two structures are equivalent.

Now let $V_{\mathbb{R}}$ be a real locally convex vector space and set
$V = V_{\mathbb{C}} = V_{\mathbb{R}} \tensor \mathbb{C}$, always
endowed with the above locally convex topology and the canonical
complex conjugation.
\begin{lemma}
    \label{lemma:ccContinuousOnWeylR}%
    Let $R \in \mathbb{R}$ and let $\halbnorm{p}$ be a continuous
    seminorm on $V$ with $\halbnorm{p}(v) =
    \halbnorm{p}(\cc{v})$. Then
    \begin{equation}
        \label{eq:pREinspREinftyCC}
        \halbnorm{p}_R(a)
        =
        \halbnorm{p}_R(\cc{a})
    \end{equation}
    for all $a \in \Sym^\bullet(V)$. Thus the complex conjugation
    extends to a continuous antilinear involutive endomorphism of
    $\cWeylR(V)$.
\end{lemma}

For a continuous $\mathbb{C}$-bilinear form $\Lambda\colon V \times V
\longrightarrow \mathbb{C}$ we define $\cc{\Lambda}\colon V \times V
\longrightarrow \mathbb{C}$ by
\begin{equation}
    \label{eq:ccLambda}
    \cc{\Lambda}(v, w) = \cc{\Lambda(\cc{v}, \cc{w})}
\end{equation}
as usual, and set $\RE(\Lambda) = \frac{1}{2}(\Lambda + \cc{\Lambda})$
as well as $\IM(\Lambda) = \frac{1}{2\I}(\Lambda - \cc{\Lambda})$.
Then $\cc{\Lambda}$, $\RE(\Lambda)$, and $\IM(\Lambda)$ are again
continuous $\mathbb{C}$-bilinear forms.

In view of applications in quantum physics we rescale our deformation
parameter $z \in \mathbb{C}$ to $z = \frac{\I\hbar}{2}$ and consider
only \emph{real} (or even positive) values for $\hbar$. Thus the star
product becomes
\begin{equation}
    \label{eq:Starhbar}
    a \starhLambda b
    =
    \mu \circ \E^{\frac{\I\hbar}{2} \Lambda} (a \tensor b).
\end{equation}
The next result clarifies under which conditions the complex
conjugation is a $^*$-involution for $\starhLambda$:
\begin{proposition}
    \label{proposition:Involution}%
    Let $R \ge \frac{1}{2}$ and $\hbar \in \mathbb{R} \setminus
    \{0\}$. Then the following statements are equivalent:
    \begin{propositionlist}
    \item \label{item:ccIsInvolution} The complex conjugation is a
        $^*$-involution with respect to $\starhLambda$, i.e.  we have
        for homogeneous $a, b \in \cWeylR(V, \starhLambda)$
        \begin{equation}
            \label{eq:ccInvolution}
            \cc{a \starhLambda b} =
            (-1)^{ab}
            \,
            \cc{b} \starhLambda \cc{a}.
        \end{equation}
    \item \label{item:LambdaPlusImaginaryLambdaMinusReal}
        $\cc{\Lambda}_+ = - \Lambda_+$ and $\cc{\Lambda}_- =
        \Lambda_-$.
    \item \label{item:LambdaPlusImaginaryPoissonStarAlg}
        $\cc{\Lambda}_+ = - \Lambda_+$ and $\cWeylR(V)$ is a Poisson
        $^*$-algebra in the sense that for all $a, b \in \cWeylR(V)$
        \begin{equation}
            \label{eq:PoissonInvolution}
            \cc{\Bracket{a, b}} = \Bracket{\cc{a}, \cc{b}}.
        \end{equation}
    \end{propositionlist}
\end{proposition}
\begin{proof}
    First we note that by continuity it suffices to work on $\WeylR(V,
    \starhLambda)$ instead of the completion. Suppose
    \refitem{item:ccIsInvolution} and consider homogeneous $v, w \in
    V$. Then
    \begin{align*}
        0
        &=
        \cc{v \starhLambda w} - (-1)^{vw} \cc{w} \starhLambda \cc{v}
        \\
        &=
        - \frac{\I\hbar}{2}\left(
            \cc{\Lambda}(\cc{v}, \cc{w})
            +
            (-1)^{vw}
            \Lambda(\cc{w}, \cc{v})
        \right) \\
        &=
        - \frac{\I\hbar}{2}\left(
            \cc{\Lambda}_+(\cc{v}, \cc{w}) + \Lambda_+(\cc{w}, \cc{v})
            +
            \cc{\Lambda}_-(\cc{v}, \cc{w}) - \Lambda_-(\cc{w}, \cc{v})
        \right),
    \end{align*}
    since the star product gives only zeroth and first order terms and
    $\Lambda_+$ is symmetric while $\Lambda_-$ is antisymmetric. Now
    $\cc{\Lambda}_+ + \Lambda_+$ is still symmetric and
    $\cc{\Lambda}_- - \Lambda_-$ is still antisymmetric. Hence their
    contributions have to vanish separately which implies
    \refitem{item:LambdaPlusImaginaryLambdaMinusReal}. Next, assume
    \refitem{item:LambdaPlusImaginaryLambdaMinusReal}. Then
    \begin{equation*}
        \cc{P_{\Lambda_\pm}(a \tensor b)}
        = \mp P_{\Lambda_\pm}(\cc{a} \tensor \cc{b})
        \tag{$*$}
    \end{equation*}
    follows immediately. For the symmetric tensor product $\mu$ we
    $\cc{\mu(a \tensor b)} = \mu(\cc{a} \tensor \cc{b})$ which
    combines to give \refitem{item:LambdaPlusImaginaryPoissonStarAlg}
    at once. Conversely,
    \refitem{item:LambdaPlusImaginaryPoissonStarAlg} implies
    \refitem{item:LambdaPlusImaginaryLambdaMinusReal} by evaluating on
    $v, w \in V$. Finally, assume
    \refitem{item:LambdaPlusImaginaryLambdaMinusReal}. In general the
    (anti-) symmetry of $\Lambda_\pm$ implies
    \begin{equation*}
        P_{\Lambda_\pm} = \pm \; \tau \circ P_{\Lambda_\pm} \circ \tau,
        \tag{$**$}
    \end{equation*}
    as this follows either by a direct computation or by verifying the
    Leibniz rules for $\tau \circ P_{\Lambda_\pm} \circ \tau$ and then
    applying the uniqueness result from
    Lemma~\ref{lemma:PLambdaMapProperties}. Combining now ($*$) and
    ($**$) with the commutativity $\mu = \mu \circ \tau$ of the
    symmetric tensor product gives \refitem{item:ccIsInvolution} by a
    computation analogously to \cite[Prop.~5.2.19]{waldmann:2007a}.
\end{proof}

Thus we need a real Poisson bracket to start the deformation and an
imaginary symmetric part $\Lambda_+$ in the star product
$\starhLambda$ to have the complex conjugation as $^*$-involution. In
this case $\cWeylR(V, \starhLambda)$ is a locally convex $^*$-algebra.

%
% Bases and Nuclearity
%

\section{Bases and Nuclearity}
\label{sec:BasesNuclearity}

In this section we collect some additional properties of the Weyl
algebra $\WeylR(V, \starzLambda)$ which are inherited from $V$.

%
% Absolute Schauder Bases
%

\subsection{Absolute Schauder Bases}
\label{subsec:AsoluteSchauderBases}

Suppose that $V$ has an absolute Schauder basis, i.e.  there exists a
linearly independent set $\{e_i\}_{i \in I}$ of vectors in $V$
together with continuous coefficient functionals $\{\varphi^i\}_{i \in
  I}$ in $V'$ such that $\varphi^i(e_j) = \delta^i_j$ for $i, j \in I$
and
\begin{equation}
    \label{eq:SchauderBasis}
    v = \sum_{i \in I} \varphi^i(v) e_i
\end{equation}
converges. More precisely, for an \emph{absolute} Schauder basis one
requires that for all continuous seminorms $\halbnorm{p}$ on $V$ there
exists a continuous seminorm $\halbnorm{q}$ such that
\begin{equation}
    \label{eq:AbsoluteConvergesBasis}
    \sum_{i \in I} |\varphi^i(v)| \halbnorm{p}(e_i)
    \le
    \halbnorm{q}(v)
\end{equation}
for all $v \in V$, i.e.  the series in \eqref{eq:SchauderBasis}
converges absolutely for all continuous seminorms and can be estimated
by a continuous seminorm. In particular, at most countably many
contributions $\varphi^i(v)\halbnorm{p}(e_i)$ can be different from
$0$ for a given $v \in V$. Typically, $I$ will be countable itself.
In the following we assume to have such an absolute Schauder basis
$\{e_i\}_{i \in I}$ for $V$ with coefficient functionals
$\{\varphi^i\}_{i \in I}$.

The projective topology on $V^{\tensor n}$ is known to be compatible
with absolute Schauder bases. More precisely, we have the following
folklore lemma:
\begin{lemma}
    \label{lemma:BasisOnTensorProduct}%
    The vectors $\{e_{i_1} \tensor \cdots \tensor e_{i_n}\}_{i_1,
      \ldots, i_n \in I}$ form an absolute Schauder basis for
    $V^{\pitensor n}$ with coefficient functionals $\{\varphi^{i_1}
    \tensor \cdots \tensor \varphi^{i_n}\}_{i_1, \ldots, i_n \in I}$.
    If $\halbnorm{p}$ and $\halbnorm{q}$ are continuous seminorms on
    $V$ with \eqref{eq:AbsoluteConvergesBasis} then one has for all $v
    \in V^{\pitensor n}$
    \begin{equation}
        \label{eq:SchauderBasisTensorVn}
        \sum_{i_1, \ldots, i_n \in I}
        \left|
            \left(
                \varphi^{i_1} \tensor \cdots \tensor \varphi^{i_n}
            \right)
            (v)
        \right|
        \halbnorm{p}^n(e_{i_1} \tensor \cdots \tensor e_{i_n})
        \le
        \halbnorm{q}^n(v).
    \end{equation}
\end{lemma}

In a next step we consider the whole tensor algebra
$\Tensor^\bullet_R(V)$ endowed with the topology from
Definition~\ref{definition:TREinsTopology} for some fixed $R \ge
0$. We claim that the collection $\{e_{i_1} \tensor \cdots \tensor
e_{i_n}\}_{n \in \mathbb{N}_0, i_1, \ldots, i_n \in I}$ provides an
absolute Schauder basis for $\Tensor^\bullet_R(V)$ with corresponding
coefficient functionals $\{\varphi^{i_1} \tensor \cdots \tensor
\varphi^{i_n}\}_{n \in \mathbb{N}_0, i_1, \ldots, i_n \in I}$. Here
for $n = 0$ we take the standard basis vector $1 \in \mathbb{K}$ with
the corresponding coefficient functional.  First we note that the
linear functionals $\varphi^{i_1} \tensor \cdots \tensor
\varphi^{i_n}$ are continuous on $T^\bullet_R(V)$ when they are
extended by $0$ to the tensor degrees different from $n$. Moreover, we
have
\begin{equation}
    \label{eq:pReinsRestrictspnFac}
    \halbnorm{p}_R(e_{i_1} \tensor \cdots \tensor e_{i_n})
    =
    \halbnorm{p}^n(e_{i_1} \tensor \cdots \tensor e_{i_n}) n!^R
\end{equation}
by Lemma~\ref{lemma:pREinspRinftyEstimates},
\refitem{item:pRRestrictToVn}. This results in the following
statement:
\begin{lemma}
    \label{lemma:SchauderForTReins}%
    The vectors $\{e_{i_1} \tensor \cdots \tensor e_{i_n}\}_{n \in
      \mathbb{N}_0, i_1, \ldots, i_n \in I}$ form an absolute Schauder
    basis for $\Tensor^\bullet_R(V)$ with coefficient functionals
    $\{\varphi^{i_1} \tensor \cdots \tensor \varphi^{i_n}\}_{n \in
      \mathbb{N}_0, i_1, \ldots, i_n \in I}$.  If $\halbnorm{p}$ and
    $\halbnorm{q}$ are continuous seminorms on $V$ with
    \eqref{eq:AbsoluteConvergesBasis} then one has for all $v \in
    \Tensor^\bullet_R(V)$
    \begin{equation}
        \label{eq:EstimateAbsoluteSchauderTensorAlgebra}
        \sum_{n=0}^\infty
        \sum_{i_1, \ldots, i_n \in I}
        \left|
            \left(
                \varphi^{i_1} \tensor \cdots \tensor \varphi^{i_n}
            \right)
            (v)
        \right|
        \halbnorm{p}_R(e_{i_1} \tensor \cdots \tensor e_{i_n})
        \le
        \halbnorm{q}_R(v).
    \end{equation}
\end{lemma}
\begin{remark}
    \label{remark:AbsoluteSchauderToCompletion}
    We note that an \emph{absolute} Schauder basis stays an absolute
    Schauder basis after completion, see
    \cite[Prop.~14.7.7]{jarchow:1981a}.
\end{remark}

The absolute Schauder basis descends now to the symmetric algebra by
symmetrizing. If $v \in \Sym^\bullet(V)$ then we have $v =
\Symmetrizer v$ with the continuous symmetrization map from
\eqref{eq:SymmetrizerContinuous}. Applying $\Symmetrizer$ twice, this
shows the convergence
\begin{equation}
    \label{eq:SchauderForSym}
    v
    =
    \sum_{n=0}^\infty \sum_{i_1, \ldots, i_n \in I}
    \left(\varphi^{i_1} \cdots \varphi^{i_n}\right)(v)
    e_{i_1} \cdots e_{i_n}
\end{equation}
for all $v \in \Sym^\bullet_R(V)$, where $\varphi^{i_1} \cdots
\varphi^{i_n} = (\varphi^{i_1} \tensor \cdots \tensor \varphi^{i_n})
\circ \Symmetrizer$. Moreover, using again $v = \Symmetrizer v$ we get
from the estimate \eqref{eq:SymmetrizerContinuous} and
\eqref{eq:EstimateAbsoluteSchauderTensorAlgebra} the estimate
\begin{equation}
    \label{eq:SchauderForSymIsAbsolute}
    \sum_{n=0}^\infty \sum_{i_1, \ldots, i_n \in I}
    \left|
        \left(\varphi^{i_1} \cdots \varphi^{i_n}\right)(v)
    \right|
    \halbnorm{p}_R(e_{i_1} \cdots e_{i_n})
    \le
    \halbnorm{q}_R(v).
\end{equation}
So we have all we need for an absolute Schauder basis \emph{except}
that the symmetrizations $e_{i_1} \cdots e_{i_n}$ will no longer be
linearly independent: some of them will be zero if they contain twice
the same odd vector and some of them will differ by signs. So we only
have to single out a maximal linearly independent subset, the choice
of which might be personal taste:
\begin{proposition}
    \label{proposition:SchauderBasisForWeylR}%
    Let $R \ge 0$ and let $\{e_i\}_{i \in I}$ be an absolute Schauder
    basis of $V$ of homogeneous vectors with coefficient functionals
    $\{\varphi^i\}_{i \in I}$. Then any choice of a maximal linearly
    independent subset of $\{e_{i_1} \cdots e_{i_n}\}_{n \in
      \mathbb{N}_0, i_1, \ldots, i_n \in I}$ will give an absolute
    Schauder basis of $\Sym^\bullet_R(V)$ and of $\WeylRMinus(V)$ with
    coefficient functionals given by the corresponding subset of
    $\{\varphi^{i_1} \cdots \varphi^{i_n}\}_{n \in \mathbb{N}_0, i_1,
      \ldots, i_n \in I}$. One has the estimate
    \eqref{eq:SchauderForSymIsAbsolute} whenever $\halbnorm{p}$ and
    $\halbnorm{q}$ satisfy \eqref{eq:AbsoluteConvergesBasis}.
\end{proposition}

Having an absolute Schauder basis is a very strong property for a
locally convex space. In fact, they are completely known: after
completion one obtains a Köthe (sequence) space where the index set
for the ``sequences'' is $I$, see e.g. \cite[Sect.~1.7E and
Sect.~14.7]{jarchow:1981a} for a detailed description. More precisely,
the Köthe matrix $K_V$ of $V$ is obtained from $K_V = (\lambda_{i,
  \halbnorm{p}})$ with $\lambda_{i, \halbnorm{p}} = \halbnorm{p}(e_i)$
where $i \in I$ and $\halbnorm{p}$ ranges over a defining system of
continuous seminorms of $V$. Thus the corresponding Köthe matrix of
the tensor algebra $\Tensor^\bullet_R(V)$ is given by
$K_{\Tensor^\bullet_R(V)} = (\lambda_{(n, i_1, \ldots, i_n),
  \halbnorm{p}})$ with
\begin{equation}
    \label{eq:KoetheMatrixTensorAlgebra}
    \lambda_{(n, i_1, \ldots, i_n), \halbnorm{p}}
    =
     n!^R
     \lambda_{i_1, \halbnorm{p}} \cdots \lambda_{i_n, \halbnorm{p}}.
\end{equation}
Thus we have an \emph{explicit} description in terms of the Köthe
matrix of $V$. Analogously, one can proceed for the Weyl algebra
$\WeylR(V)$. For Köthe spaces many properties are (easily) encoded in
their Köthe matrix, so we see here that the appearance of $n!^R$ will
play a prominent role when passing from $V$ to $\Tensor^\bullet_R(V)$
or $\WeylR(V)$.

%
% Nuclearity of $\WeylR(V)$ and Strong Nuclearity of $\WeylRMinus(V)$
%

\subsection{Nuclearity of $\WeylR(V)$ and Strong Nuclearity of $\WeylRMinus(V)$}
\label{subsec:Nuclearity}

Let us now discuss nuclearity properties of the Weyl algebra
$\WeylR(V)$ originating from those of $V$: since $V \subseteq
\WeylR(V)$ is a closed subspace inheriting the original topology from
the one of $\WeylR(V)$, we see that nuclearity of $\WeylR(V)$ implies
the nuclearity of $V$. The aim of this subsection is to show the
converse.

To this end, it will be convenient to work with the tensor algebra
$\Tensor^\bullet_R(V)$ instead of the symmetric algebra
$\Sym^\bullet_R(V)$ since we do not have to take care of the
combinatorics of symmetrization.

Let $U \subseteq V$ be a subspace and denote by $\langle U \rangle
\subseteq \Tensor^\bullet(V)$ the two-sided ideal generated by
$U$. Then the quotient algebra $\Tensor^\bullet(V) \big/ \langle U
\rangle$ is still $\mathbb{Z}$-graded by the tensor degree since $U$
has homogeneous generators of tensor degree one. The map
\begin{equation}
    \label{eq:CanonicalIsoTVmodU}
    \iota\colon
    \Tensor^\bullet(V) \big/\langle U\rangle
    \longrightarrow
    \Tensor^\bullet\left(V \big/ U\right)
\end{equation}
determined by $\iota([v_1 \tensor \cdots \tensor v_n]) = [v_1] \tensor
\cdots \tensor{} [v_n]$ turns out to be an isomorphism of graded
algebras. We shall now show that $\iota$ also respects the seminorms
$\halbnorm{p}_R$. First recall that for a seminorm $\halbnorm{p}$
on $V$ one defines a seminorm $[\halbnorm{p}]$ on $V \big/ U$ by
\begin{equation}
    \label{eq:QuotientSeminorm}
    [\halbnorm{p}]([v])
    =
    \inf\left\{\halbnorm{p}(v + u) \; | \; u \in U\right\}.
\end{equation}
Then the locally convex quotient topology on $V \big/ U$ is obtained
from all the seminorms $[\halbnorm{p}]$ where $\halbnorm{p}$ runs
through all the continuous seminorms of $V$.
\begin{lemma}
    \label{lemma:SeminormQuotientIota}%
    Let $U \subseteq V$ be a subspace and let $\halbnorm{p}$ be a
    seminorm on $V$. Then for all $R \in \mathbb{R}$ one has
    \begin{equation}
        \label{eq:SeminormQuotientIota}
        [\halbnorm{p}_R]
        =
        [\halbnorm{p}]_R \circ \iota.
    \end{equation}
\end{lemma}
\begin{proof}
    Let $\langle U \rangle_n = \sum_{\ell = 1}^n V \tensor \cdots
    \tensor U \tensor \cdots \tensor V \subseteq V^{\tensor n}$ with
    $U$ being at the $\ell$-th position. This is the $n$-th
    homogeneous part of $\langle U \rangle$. Then $V^{\tensor n} \big/
    \langle U \rangle_n$ gives the $n$-th homogeneous part of the
    graded algebra $\Tensor^\bullet(V) \big/\langle U \rangle$. For a
    seminorm $\halbnorm{p}$ on $V$ and the induced isomorphism $\iota$
    restricted to $V^{\tensor n} \big/ \langle U \rangle_n$ a simple
    argument shows
    \[
    [\halbnorm{p}^n] = [\halbnorm{p}]^n \circ \iota.
    \]
    From this, \eqref{eq:SeminormQuotientIota} follows at once.
\end{proof}

This simple lemma has an important consequence which we formulate in
two ways:
\begin{corollary}
    \label{corollary:QuotientIotaIsomorphism}%
    Let $U \subseteq V$ be a subspace and $R \in \mathbb{R}$.
    \begin{corollarylist}
    \item \label{item:TRmodGenUTRVmodU} The isomorphism
        \eqref{eq:CanonicalIsoTVmodU} induces an isomorphism
        \begin{equation}
            \label{eq:iotaTRModGenUIsomorphism}
            \iota\colon
            \Tensor_R^\bullet(V) \big / \langle U \rangle
            \longrightarrow
            \Tensor_R^\bullet \left(V \big/ U\right)
        \end{equation}
        of locally convex algebras if the left hand side as well as $V
        \big/ U$ carry the locally convex quotient topologies.
    \item \label{item:iotaIsomProjLims} The isomorphism
        \eqref{eq:CanonicalIsoTVmodU} induces an isomorphism
        \begin{equation}
            \label{eq:iotaProjLimIsomorphism}
            \iota\colon
            \Tensor^\bullet_{R^-}(V)
            \big/ \langle U \rangle
            \longrightarrow
            \Tensor^\bullet_{R^-} \left(V \big/ U\right)
        \end{equation}
        of locally convex algebras.
    \end{corollarylist}
\end{corollary}
\begin{proof}
    For the second part, we note that the seminorms $[\halbnorm{p}_{R
      - \epsilon, 1}]$ and $[\halbnorm{p}]_{R - \epsilon, 1}$ for
    $\epsilon > 0$ and $\halbnorm{p}$ a continuous seminorm on $V$
    constitute a defining system of seminorms for the projective limit
    topologies.
\end{proof}

Let us now assume that $V$ is nuclear. There are many equivalent ways
to characterize nuclearity, see e.g.  \cite[Chap.~21]{jarchow:1981a},
we shall use the following very basic one: for a given continuous
seminorm $\halbnorm{p}$ on $V$ we consider $V \big/ \ker \halbnorm{p}$
with the quotient seminorm $[\halbnorm{p}]$. This is now a normed
space as we have divided by $\ker \halbnorm{p}$. Thus we can complete
$V \big/ \ker \halbnorm{p}$ to a Banach space denoted by
$V_{\halbnorm{p}}$. Then $V$ is called nuclear if for every continuous
seminorm $\halbnorm{p}$ there is a another continuous seminorm
$\halbnorm{q} \ge \halbnorm{p}$ such that the canonical map
$\iota_{\halbnorm{q}\halbnorm{p}}\colon V_{\halbnorm{q}}
\longrightarrow V_{\halbnorm{p}}$ is a nuclear map. This means that
there are vectors $e_i \in V_{\halbnorm{p}}$ and continuous linear
functionals $\epsilon^i \in V_{\halbnorm{q}}'$ such that
\begin{equation}
    \label{eq:NuclearMap}
    \iota_{\halbnorm{q}\halbnorm{p}}(v)
    =
    \sum_{i=1}^\infty \epsilon^i(v) e_i
    \quad
    \textrm{with}
    \quad
    \sum_{i=1}^\infty
    \norm{\epsilon^i}_{\halbnorm{q}} \norm{e_i}_{\halbnorm{p}}
    < \infty,
\end{equation}
where we use the notation $\norm{\argument}_{\halbnorm{p}} =
[\halbnorm{p}]$ for the Banach norms on $V_{\halbnorm{p}}$ and
\begin{equation}
    \label{eq:FunctionalNorm}
    \norm{\epsilon^i}_{\halbnorm{q}}
    =
    \sup_{v \ne 0}
    \frac{|\epsilon^i(v)|}{\norm{v}_{\halbnorm{q}}}
\end{equation}
denotes the functional norm as usual. The following lemma is
well-known:
\begin{lemma}
    \label{lemma:BanachTensorFunctionalNorms}%
    Let $(V, \norm{\argument})$ be a Banach space and let $\varphi_1,
    \ldots, \varphi_n \in V'$. Then $\varphi_1 \tensor \cdots \tensor
    \varphi_n \in (V^{\pitensor n})'$ with
    \begin{equation}
        \label{eq:FunctionalNormFactorizes}
        \norm{\varphi_1 \tensor \cdots \tensor \varphi_n}
        =
        \norm{\varphi_1} \cdots \norm{\varphi_n},
    \end{equation}
    where on $V^{\pitensor n}$ we use the norm $\norm{\argument}
    \tensor \cdots \tensor \norm{\argument}$ as usual.
\end{lemma}
The next lemma shows how $\ker \halbnorm{p} \subseteq V$ is related to
$\ker \halbnorm{p}_R \subseteq T_R^\bullet(V)$.
\begin{lemma}
    \label{lemma:kerpGenerateskerpReins}%
    Let $R \in \mathbb{R}$ and let $\halbnorm{p}$ be a continuous
    seminorm on $V$. Then
    \begin{equation}
        \label{eq:kerpGenerateskerpReins}
        \langle \ker \halbnorm{p} \rangle
        =
        \ker \halbnorm{p}_R.
    \end{equation}
\end{lemma}
\begin{proof}
    Since $[\halbnorm{p}]$ is a norm on $V \big/ \ker \halbnorm{p}$,
    also $[\halbnorm{p}]^n$ is a norm on $(V \big/ \ker
    \halbnorm{p})^{\tensor n}$. It follows that $[\halbnorm{p}]_{R,
      1}$ is a norm on $\Tensor^\bullet(V \big/ \ker \halbnorm{p})$ as
    well, implying that $[\halbnorm{p}_R]$ is a norm on
    $\Tensor(V) \big/ \langle \ker \halbnorm{p} \rangle$ according to
    Lemma~\ref{lemma:SeminormQuotientIota}. Thus $[v] \in \ker
    \halbnorm{p}_R$ iff $[\halbnorm{p}_R]([v]) = 0$ iff $[v]
    = 0$ iff $v \in \langle \ker \halbnorm{p} \rangle$.
\end{proof}
The following lemma is the key to understand nuclearity:
\begin{lemma}
    \label{lemma:NuclearRepresentationTensor}%
    Let $\halbnorm{p} \le \halbnorm{q}$ be continuous seminorms such
    that the canonical map $\iota_{\halbnorm{q}\halbnorm{p}}\colon
    V_{\halbnorm{q}} \longrightarrow V_{\halbnorm{p}}$ has a nuclear
    representation
    \begin{equation}
        \label{eq:NuclearRepresentation}
        \iota_{\halbnorm{q}, \halbnorm{p}}
        =
        \sum_{i = 1}^\infty \epsilon^i \tensor e_i
    \end{equation}
    with $\epsilon^i \in V_{\halbnorm{q}}'$ and $e_i \in
    V_{\halbnorm{p}}$. Then there is a constant $c > 0$ such that the
    canonical map
    \begin{equation}
        \label{eq:iotaqReinspReins}
        \iota_{(c\halbnorm{q})_R, \halbnorm{p}_R}\colon
        \Tensor_R^\bullet(V)_{(c\halbnorm{q})_R}
        \longrightarrow
        \Tensor_R^\bullet(V)_{\halbnorm{p}_R}
    \end{equation}
    has the nuclear representation
    \begin{equation}
        \label{eq:NuclearRepresentationTensor}
        \iota_{(c\halbnorm{q})_R, \halbnorm{p}_R}
        =
        \sum_{n=0}^\infty \sum_{i_1, \ldots, i_n = 1}^\infty
        \left(
            \epsilon^{i_1} \tensor \cdots \tensor \epsilon^{i_n}
        \right)
        \tensor
        \left(
            e_{i_1} \tensor \cdots \tensor e_{i_n}
        \right).
    \end{equation}
\end{lemma}
\begin{proof}
    By rescaling the seminorm $\halbnorm{q}$ we can achieve that the
    numerical value
    \begin{equation*}
        x
        =
        \sum_{i=1}^\infty
        \norm{\epsilon^i}_{\halbnorm{q}} \norm{e_i}_{\halbnorm{p}}
        < 1
        \tag{$*$}
    \end{equation*}
    of the convergence condition in \eqref{eq:NuclearMap} is not only
    finite but actually as small as we need: rescaling of
    $\halbnorm{q}$ by $c$ shrinks the value of ($*$) by $\frac{1}{c}$
    since we need the \emph{dual} norm
    $\norm{\argument}_{\halbnorm{q}}$. Thus we may assume ($*$)
    without restriction defining the possibly necessary rescaling
    factor $c$.  Next we note that
    $\big(\Tensor^\bullet_R(V)\big)_{\halbnorm{q}_R}$ is the Banach
    space completion of $\Tensor^\bullet_R(V) \big/ \ker
    \halbnorm{q}_R \cong \Tensor^\bullet_R(V \big/ \ker \halbnorm{q})$
    with respect to the norm $[\halbnorm{q}_R] = [\halbnorm{q}]_R
    \circ \iota$, according to Lemma~\ref{lemma:SeminormQuotientIota},
    and analogously for
    $\left(\Tensor^\bullet_R(V)\right)_{\halbnorm{p}_R}$. In this
    sense we have $e_{i_1} \tensor \cdots \tensor e_{i_n} \in
    \left(\Tensor^\bullet_R(V)\right)_{\halbnorm{p}_R}$ as well as
    $\epsilon^{i_1} \tensor \cdots \tensor \epsilon^{i_n} \in
    \Big(\big(\Tensor^\bullet_R(V)\big)_{\halbnorm{q}_R}\Big)'$. For
    the norms of these vectors and linear functionals we have
    \[
    \halbnorm{p}_R(e_{i_1} \tensor \cdots \tensor e_{i_n})
    =
    n!^R[\halbnorm{p}]^n(e_{i_1} \tensor \cdots \tensor e_{i_n})
    =
    n!^R
    \norm{e_{i_1}}_{\halbnorm{p}}
    \cdots
    \norm{e_{i_n}}_{\halbnorm{p}}
    \]
    and
    \[
    \norm{
      \epsilon^{i_1}
      \tensor \cdots \tensor
      \epsilon^{i_n}
    }_{\halbnorm{q}_R}
    =
    \frac{1}{n!^R}
    \norm{
      \epsilon^{i_1}
      \tensor \cdots \tensor
      \epsilon^{i_n}
    }_{\halbnorm{q}^n}
    =
    \frac{1}{n!^R}
    \norm{\epsilon^{i_1}}_{\halbnorm{q}}
    \cdots
    \norm{\epsilon^{i_n}}_{\halbnorm{q}}.
    \]
    Again, due to dualizing, the prefactor $n!^R$ appears now in
    the denominator. Combining these results we have
    \[
    \sum_{n=0}^\infty \sum_{i_1, \ldots, i_n = 0}^\infty
    \norm{
      \epsilon^{i_1}
      \tensor \cdots \tensor
      \epsilon^{i_n}
    }_{\halbnorm{q}_R}
    \halbnorm{p}_R(e_{i_1} \tensor \cdots \tensor e_{i_n})
    =
    \sum_{n=0}^\infty x^n < \infty,
    \]
    since we arranged $x < 1$. Finally, it is clear that
    \eqref{eq:NuclearRepresentationTensor} holds before the completion
    and thus also afterwards by continuity.
\end{proof}
\begin{theorem}
    \label{theorem:NuclearWeylAlgebra}%
    Let $R \ge 0$. Then the following statements are equivalent:
    \begin{theoremlist}
    \item \label{item:Vnuclear} $V$ is nuclear.
    \item \label{item:TVnuclear} $\Tensor_R^\bullet(V)$ is
        nuclear.
    \item \label{item:WeylRnuclear} $\WeylR(V)$ is nuclear.
    \end{theoremlist}
\end{theorem}
\begin{proof}
    We consider the tensor algebra $\Tensor^\bullet_R(V)$. Then
    $\WeylR(V) = \Sym^\bullet_R(V)$ is a closed subspace of
    $\Tensor^\bullet_R(V)$ and $V$ is a closed subspace of
    $\WeylR(V)$. Hence it will suffice to show that
    $\Tensor^\bullet_R(V)$ is nuclear whenever $V$ is nuclear. Since
    the topology of $\Tensor^\bullet_R(V)$ is determined by all the
    seminorms $\halbnorm{p}_R$,
    Lemma~\ref{lemma:NuclearRepresentationTensor} gives us the nuclear
    representation of $\iota_{(c\halbnorm{q})_R, \halbnorm{p}_R}$
    whenever we have one for $\iota_{\halbnorm{q},
      \halbnorm{p}}$. Hence $\Tensor^\bullet_{R, 1}(V)$ is nuclear.
\end{proof}

For the projective limit version $\WeylRMinus(V)$ we can argue either
along the same line of proof as above or use the above result and rely
on the general fact that projective limits of nuclear spaces are again
nuclear. However, the following statement is less obvious and shines
some new light on the projective version of the Weyl algebra:
\begin{theorem}
    \label{theorem:StronglyNuclear}%
    Let $R \ge 0$. Then the following statements are equivalent:
    \begin{theoremlist}
    \item \label{item:VStronglyNuclear} $V$ is strongly nuclear.
    \item \label{item:TensorStronglyNuclear} $\Tensor_{R^-}(V)$ is
        strongly nuclear.
    \item \label{item:WeylStronglyNuclear} $\WeylRMinus(V)$ is
        strongly nuclear.
    \end{theoremlist}
\end{theorem}
\begin{proof}
    Again, since closed subspaces inherit strong nuclearity, we only
    have to show the implication \refitem{item:VStronglyNuclear}
    $\implies$ \refitem{item:TensorStronglyNuclear}. Thus let
    $\halbnorm{p}$ be a continuous seminorm on $V$ with a matching
    continuous seminorm $\halbnorm{q}$ such that \eqref{eq:NuclearMap}
    holds with
    \[
    x(\alpha)
    =
    \sum_{i=1}^\infty
    \norm{\epsilon^i}_{\halbnorm{q}}^\alpha
    \norm{e_i}_{\halbnorm{p}}^\alpha
    < \infty
    \]
    for all $\alpha > 0$. Now we have to take $\halbnorm{q}_{R -
      \epsilon'}$ for $\halbnorm{p}_{R - \epsilon}$ with some $0
    < \epsilon' < \epsilon$. Then the series
    \[
    \sum_{n=0}^\infty \sum_{i_1, \ldots, i_n = 0}^\infty
    \left(
        \norm{
          \epsilon^{i_1}
          \tensor \cdots \tensor
          \epsilon^{i_n}
        }_{\halbnorm{q}_{R - \epsilon'}}
        \halbnorm{p}_{R - \epsilon}(e_{i_1} \tensor \cdots \tensor
        e_{i_n})
    \right)^\alpha
    =
    \sum_{n=0}^\infty
    \frac{1}{n!^{\alpha(\epsilon - \epsilon')}}
    x(\alpha)^{\alpha}
    \]
    still converges for all $\alpha > 0$ showing the strong nuclearity
    of $\Tensor_{R^-}(V)$.
\end{proof}

Note that it is crucial to have some (small) inverse power of $n!$ at
hand: without this option we can not succeed in showing the strong
nuclearity as we need the $\alpha$-summability for \emph{all} $\alpha
> 0$. Thus, concerning strong nuclearity, the projective limit version
$\WeylRMinus(V)$ turns out to behave nicer than the more direct
version $\WeylR(V)$.
\begin{example}
    \label{example:FiniteDimV}%
    For $V$ finite-dimensional we get a strongly nuclear Weyl algebra
    $\WeylRMinus(V, \starzLambda)$. Here we can either use the above
    theorem since $V$ being finite-dimensional is strongly nuclear, or
    we can rely on the explicit description of $V$ and hence of
    $\WeylRMinus(V, \starzLambda)$ as a Köthe space: since for a
    finite-dimensional vector space it suffices to take a single norm,
    the Köthe matrix is finite and hence its entries are bounded, say
    by $c > 0$. Thus the Köthe matrix
    \eqref{eq:KoetheMatrixTensorAlgebra} has entries bounded by $c^n
    n!^{R - \epsilon}$ where $\epsilon > 0$. To this result one can
    apply the Grothendieck-Pietsch criterion and conclude strong
    nuclearity directly, see \cite[Prop.~21.8.2]{jarchow:1981a}. Of
    course, for a finite-dimensional $V$ the Weyl algebra $\WeylR(V,
    \starzLambda)$ is still a nuclear space.
\end{example}

%
% Symmetries and Equivalences
%

\section{Symmetries and Equivalences}
\label{sec:SymmetriesEquivalences}

In this section we discuss how the algebraic symmetries and
equivalences translate into our locally convex framework.

%
% Functoriality of $\cWeylR(V, \starzLambda)$
%

\subsection{Functoriality of $\cWeylR(V, \starzLambda)$}
\label{subsec:Functoriality}

Suppose that $V$ and $W$ are two $\mathbb{Z}_2$-graded locally convex
Hausdorff spaces and let $\Lambda_V$ and $\Lambda_W$ be continuous
even bilinear forms on $V$ and $W$, respectively. We want to extend
the functoriality statement from
Proposition~\ref{proposition:Functorial}. The following estimates are
obvious:
\begin{lemma}
    \label{lemma:ContinuityA}%
    Let $A\colon V \longrightarrow W$ be an even linear map and let
    $\halbnorm{p}$ and $\halbnorm{q}$ be seminorms on $V$ and $W$ such
    that $\halbnorm{q}(A(v)) \le \halbnorm{p}(v)$ for all $v \in
    V$. Then
    \begin{equation}
        \label{eq:pReinsAEstimates}
        \halbnorm{q}_R(A(v))
        \le
        \halbnorm{p}_R(v)
    \end{equation}
    for all $v \in \Tensor^\bullet(V)$.
\end{lemma}
\begin{proof}
    We clearly have $\halbnorm{q}^n(A^{\tensor n}(v)) \le
    \halbnorm{p}^n(v)$ for all $v \in \Tensor^n(V)$. From this,
    \eqref{eq:pReinsAEstimates} is clear by the definition of
    $\halbnorm{p}_R$.
\end{proof}
\begin{proposition}
    \label{proposition:FunctorialityWeylR}%
    Let $R \ge \frac{1}{2}$ and $z \in \mathbb{C}$. Then the Weyl
    algebra $\WeylR(V, \starzLambda)$ as well as its completion
    $\cWeylR(V, \starzLambda)$ depend functorially on $(V, \Lambda)$
    with respect to continuous Poisson maps.
\end{proposition}
In particular, the \emph{continuous} Poisson automorphisms in $\Aut(V,
\Lambda)$ act on the Weyl algebra $\WeylR(V, \starzLambda)$ as well as
on its completion $\cWeylR(V, \starzLambda)$ by continuous
automorphisms. The analogous statement holds for the projective
version $\WeylRMinus(V, \starzLambda)$ and $R > \frac{1}{2}$.

%
% Translation Invariance
%

\subsection{Translation Invariance}
\label{subsec:TranslationInvariance}

Let us now investigate the action of the translations by linear forms
on $V$ as done algebraically in \eqref{eq:Translation}. We discuss the
continuity of the translations $\tau_\varphi^*$ in two ways: first
directly for a general even continuous $\varphi \in V'$ and second for
a $\varphi$ in the image of $\sharp$ from \eqref{eq:sharpMap}: since
$\Lambda$ is continuous, an element $\varphi \in \image \sharp
\subseteq V^*$ is clearly continuous as well. In this more special
situation we show a much stronger statement, namely that
$\tau_\varphi^*$ is an inner automorphism.

We start with the following basic estimate for the continuity of the
translation operators $\tau_\varphi^*$:
\begin{lemma}
    \label{lemma:tauContinuous}%
    Let $\varphi \in V'$ be even and let $\halbnorm{p}$ be a
    continuous seminorm on $V$ such that $|\varphi(v)| \le
    \halbnorm{p}(v)$ for all $v \in V$. Then for $R \ge 0$ we have for
    all $v \in \Tensor^\bullet_R(V)$
    \begin{equation}
        \label{eq:pREinsTau}
        \halbnorm{p}_R(\tau_\varphi^* v)
        \le
        (2\halbnorm{p})_R (v).
    \end{equation}
\end{lemma}
\begin{proof}
    We write $v = \sum_{n=0}^\infty v_n \in \Tensor^\bullet(V)$ with
    its homogeneous components $v_n \in V^{\tensor n}$, all of which
    are zero except finitely many. Moreover, we write
    \begin{equation*}
        v_n = \sum\nolimits_i v_i^{(1)} \tensor \cdots \tensor v_i^{(n)}
        \tag{$*$}
    \end{equation*}
    as usual. Then the homomorphism property of $\tau_\varphi^*$ gives
    \[
    \tau_\varphi^* v
    =
    \sum_{n=0}^\infty \sum\nolimits_i
    \left(v_i^{(1)} + \varphi\left(v_i^{(1)}\right) \Unit\right)
    \tensor \cdots \tensor
    \left(v_i^{(n)} + \varphi\left(v_i^{(n)}\right) \Unit\right).
    \]
    For every $n$ we get now various contributions in all tensor
    degrees $k \le n$. The contributions in the tensor degree $k$
    consists linear combinations of a choice of $k$ vectors among the
    $v_i^{(1)}, \ldots, v_i^{(n)}$, taking their tensor product,
    applying $\varphi$ to the remaining $n-k$ vectors, and multiplying
    everything together in the end. For a fixed index $i$ there are
    $\binom{n}{k}$ possibilities to distribute $n-k$ copies of
    $\varphi$ to the $n$ vectors $v_i^{(1)}, \ldots, v_i^{(n)}$.
    Finally, using the estimate $|\varphi(w)| \le \halbnorm{p}(w)$ for
    all $w \in V$ we obtain that the contributions to $\halbnorm{p}_R$
    from these terms can be estimated by
    \[
    \halbnorm{p}_R
    \left(
        \left(v_i^{(1)} + \varphi\left(v_i^{(1)}\right)\Unit \right)
        \tensor \cdots \tensor
        \left(v_i^{(n)} + \varphi\left(v_i^{(n)}\right)\Unit \right)
    \right)
    \le
    \sum_{k=0}^n
    \binom{n}{k} k!^R
    \halbnorm{p}\left(v_{i}^{(1)}\right)
    \cdots
    \halbnorm{p}\left(v_i^{(n)}\right).
    \]
    In total, we get the estimate
    \[
    \halbnorm{p}_R(\tau_\varphi^*v)
    \le
    \sum_{n=0}^\infty \sum\nolimits_i
    \sum_{k=0}^n
    \binom{n}{k} k!^R
    \halbnorm{p}\left(v_{i}^{(1)}\right)
    \cdots
    \halbnorm{p}\left(v_i^{(n)}\right)
    \le
    \sum_{n=0}^\infty \sum\nolimits_i
    2^n n!^R
    \halbnorm{p}\left(v_{i}^{(1)}\right)
    \cdots
    \halbnorm{p}\left(v_i^{(n)}\right).
    \]
    Since the decomposition ($*$) was arbitrary, we can take the
    infimum over all such decompositions resulting in
    \eqref{eq:pREinsTau}.
\end{proof}

From this estimate we get immediately the following continuity
statements:
\begin{proposition}
    \label{proposition:ContinuityOfTau}%
    Let $\varphi \in V'$ be an even continuous linear functional and
    let $R \ge 0$.
    \begin{propositionlist}
    \item \label{item:TauTReinsContinuous} The algebra automorphism
        $\tau_\varphi^*\colon \Tensor^\bullet_R(V)
        \longrightarrow \Tensor^\bullet_R(V)$ is continuous.
    \item \label{item:TauWeylRContinuous} For $R \ge \frac{1}{2}$, the
        algebra automorphism $\tau_\varphi^*\colon \WeylR(V,
        \starzLambda) \longrightarrow \WeylR(V, \starzLambda)$ is
        continuous.
    \item \label{item:TauWeylRMinusContinuous} For $R > \frac{1}{2}$,
        the algebra automorphism $\tau_\varphi^*\colon \WeylRMinus(V,
        \starzLambda) \longrightarrow \WeylRMinus(V, \starzLambda)$ is
        continuous.
    \end{propositionlist}
\end{proposition}
In particular, $\tau_\varphi^*$ extends in all three cases to the
corresponding completions and yields a continuous automorphism for the
completions, too.

In a next step we want to understand which of the $\tau_\varphi^*$ are
inner automorphisms. Heuristically, this is a well-known statement: if
the linear functional $\varphi$ is in the image of $\sharp$ then
$\tau_\varphi^*$ is inner via the star-exponential of a pre-image of
$\varphi$ with respect to $\sharp$. Also the heuristic formula for the
star-exponential is folklore. Our main point here is that we have an
analytic framework where the star-exponential actually makes sense:
this is in so far nontrivial as we know that the canonical commutation
relations do not allow for a general entire calculus, see the
discussion in Remark~\ref{remark:FreeLMCAlgebra} and
Remark~\ref{remark:FrechetIsOptimal}. Thus the existence of an
exponential has to be shown by hand.

In the following it will be crucial to have $R \le 1$ in view of
Proposition~\ref{proposition:WeylAlgebraHasExp}. We start with some
basic properties of the exponential series:
\begin{lemma}
    \label{lemma:Exp}%
    Let $R < 1$ and $w \in V_{\mathbf{0}}$ be an even vector.
    \begin{lemmalist}
    \item \label{item:expStarv} For all $v \in V$ we have
        \begin{equation}
            \label{eq:expwStarv}
            \exp(w) \starzLambda v
            =
            \exp(w) \left(v + z \Lambda(w, v)\right),
        \end{equation}
        and
        \begin{equation}
            \label{eq:vStarexpw}
            v \starzLambda \exp(w)
            =
            \exp(w) \left(v + z \Lambda(v, w)\right).
        \end{equation}
    \item \label{item:DDtExp} For all $t \in \mathbb{K}$ one has
        \begin{equation}
            \label{eq:DDtExp}
            \frac{\D}{\D t}
            \E^{tw + \frac{t^2 z}{2} \Lambda(w, w) \Unit}
            =
            \E^{tw + \frac{t^2 z}{2} \Lambda(w, w) \Unit}
            \starzLambda w
            =
            w \starzLambda
            \E^{t w + \frac{t^2 z}{2} \Lambda(w, w) \Unit}.
        \end{equation}
    \item \label{item:StarExponentialw} The star-exponential series
        for $w \in V$ converges absolutely in $\WeylR(V,
        \starzLambda)$ and
        \begin{equation}
            \label{eq:StarExpw}
            \Exp_{\starzLambda}(tw)
            =
            \sum_{n = 0}^\infty
            \frac{t^n}{n!} w \starzLambda \cdots \starzLambda w
            =
            \E^{tw + \frac{t^2 z}{2} \Lambda(w, w) \Unit}.
        \end{equation}
    \end{lemmalist}
\end{lemma}
\begin{proof}
    We use the continuity of $\starzLambda$ and the (absolute)
    convergence of the exponential series to get
    \begin{align*}
        \exp(w) \starzLambda v
        &=
        \sum_{n=0}^\infty \frac{1}{n!} w^n \starzLambda v \\
        &=
        \sum_{n=0}^\infty \frac{1}{n!}
        \left(
            w^n v + z \mu \circ P_\Lambda (w^n \tensor v) + 0
        \right) \\
        &=
        \sum_{n=0}^\infty \frac{1}{n!}
        \left(
            w^n v + n z w^{n-1} \Lambda(w, v)
        \right) \\
        &=
        \exp(w) \left(v + z \Lambda(w, v)\right),
    \end{align*}
    since $\mu \circ P_\Lambda(\argument, v)$ is a derivation of the
    undeformed symmetric tensor product and $w$ is even. The second
    equation is analogous. For the second part, we first note that $t
    \mapsto \exp(tw + \frac{t^2z}{2}\Lambda(w, w)\Unit)$ is
    real-analytic (entire in the case $\mathbb{K} = \mathbb{C}$) with
    convergent Taylor expansion around $0$ for all $t \in \mathbb{K}$
    thanks to Proposition~\ref{proposition:WeylAlgebraHasExp},
    \refitem{item:expEntire}. We compute the derivative
    \[
    \frac{\D}{\D t}
    \E^{tw + \frac{t^2 z}{2} \Lambda(w, w) \Unit}
    =
    \E^{tw + \frac{t^2 z}{2} \Lambda(w, w) \Unit}
    \left(w + z \Lambda(tw, w)\right)
    =
    \E^{tw + \frac{t^2 z}{2} \Lambda(w, w) \Unit}
    \starzLambda w,
    \]
    using the first part and the fact that $\Lambda(w, w)\Unit$ is
    central. Analogously, we can write $w \starzLambda$ in front. This
    shows the second part. Together, this gives
    \[
    \frac{\D}{\D t}\At{t=0}
    \E^{tw + \frac{t^2 z}{2} \Lambda(w, w) \Unit}
    =
    w \starzLambda \cdots \starzLambda w
    \]
    for the Taylor coefficients of the real-analytic (entire) function
    $t \mapsto \E^{tw + \frac{t^2 z}{2} \Lambda(w, w) \Unit}$. Since
    its Taylor series converges absolutely, the last part follows.
\end{proof}

With this preparation the following statement is now an easy
computation:
\begin{proposition}
    \label{proposition:TranslationInnerAutos}%
    Let $R < 1$ and let $\varphi \in V'$ be even. If $\varphi$ is in
    the image of $\sharp$ then $\tau_\varphi^*$ is an inner
    automorphism of $\cWeylR(V, \starzLambda)$ for all $z \ne 0$. In
    fact,
    \begin{equation}
        \label{eq:TauVarphiInner}
        \tau_\varphi^*(a)
        =
        \Exp_{\starzLambda}(w) \starzLambda a \starzLambda \Exp(-w)
    \end{equation}
    for all $a \in \cWeylR(V, \starzLambda)$ where $w \in
    V_{\mathbf{0}}$ is such that $2z w^\sharp = \varphi$.
\end{proposition}
\begin{proof}
    First we note that the star-exponential function gives a
    one-parameter group of invertible elements in $\cWeylR(V,
    \starzLambda)$ with respect to the star product
    $\starzLambda$. This is clear from the absolute convergence of the
    star-exponential series. Thus the right hand side of
    \eqref{eq:TauVarphiInner} defines an inner automorphism of the
    Weyl algebra. Now consider $v \in V$. Using Lemma~\ref{lemma:Exp}
    we compute for $v \in V$
    \begin{align*}
        \frac{\D}{\D t}
        \Exp_{\starzLambda}(tw)
        \starzLambda v \starzLambda
        \Exp_{\starzLambda}(-tw)
        &=
        \Exp_{\starzLambda}(tw)
        \starzLambda (w \starzLambda v - v \starzLambda w) \starzLambda
        \Exp_{\starzLambda}(-tw) \\
        &=
        \Exp_{\starzLambda}(tw)
        \starzLambda
        (z \Lambda (w, v)\Unit - z \Lambda (v, w) \Unit)
        \starzLambda
        \Exp_{\starzLambda}(-tw) \\
        &=
        2z \Lambda_-(w, v) \Unit,
    \end{align*}
    where we use that $\Exp_{\starzLambda}(-tw)$ is the
    $\starzLambda$-inverse of $\Exp_{\starzLambda}(tw)$. On the other
    hand, $t  \mapsto \tau_{t\varphi}^* (v) = v + t \varphi(v)
    \Unit$ has the derivative
    \[
    \frac{\D}{\D t} \tau_{t\varphi}^*(v) = \varphi(v) \Unit.
    \]
    Thus taking $w$ such that $2w \Lambda_-(w, \argument) = \varphi$,
    i.e. $2zw^\sharp = \varphi$, shows \eqref{eq:TauVarphiInner} for
    $a = v$. Now both sides are automorphisms and hence both sides
    coincide on all $\starzLambda$-polynomials in elements from
    $V$. But $V$ together with $\Unit$ generates $\WeylR(V,
    \starzLambda)$ according to
    Corollary~\ref{corollary:VGeneratesWeylAlgebra}. Thus the two
    automorphisms coincide on $\WeylR(V, \starzLambda)$. Since both
    are continuous, they also coincide on the completion $\cWeylR(V,
    \starzLambda)$.
\end{proof}
\begin{remark}
    \label{remark:NeededContinuityOfInnerStuff}%
    Note that we need the continuity of $\tau_{\varphi}^*$ in order to
    show that it coincides with the (obviously continuous) inner
    automorphism on the right hand side \eqref{eq:TauVarphiInner} on
    the completion. Note also that for $\cWeylRMinus(V)$ we can extend
    the above statement also to $R = 1$, see
    Remark~\ref{remark:WeylMinusExp}.
\end{remark}

%
% Continuous Equivalences
%

\subsection{Continuous Equivalences}
\label{subsec:ContinuousEquivalences}

In the formal setting we have seen that the same antisymmetric part
$\Lambda_-$ yields equivalent deformations, no matter what the
symmetric part $\Lambda_+$ of $\Lambda$ is. We extend this now to the
analytic framework. The following lemma shows the continuity of the
equivalence transformation from
Proposition~\ref{proposition:FormalEquivalence}:
\begin{lemma}
    \label{lemma:DeltaContinuous}%
    Let $g\colon V \times V \longrightarrow \mathbb{K}$ be an even
    symmetric bilinear form. Let $R \ge \frac{1}{2}$ and let
    $\halbnorm{p}$ be a seminorm on $V$ with
    \begin{equation}
        \label{eq:gContinuity}
        |g(v, w)| \le \halbnorm{p}(v) \halbnorm{p}(w)
    \end{equation}
    for all $v, w \in V$. Then we have for all $a \in \Sym^\bullet(V)$
    \begin{equation}
        \label{eq:gContinuous}
        \halbnorm{p}_R(\Delta_g a) \le (2\halbnorm{p})_R(a).
    \end{equation}
    Moreover, there are constants $c, c' > 0$ with
    \begin{equation}
        \label{eq:expDeltagContinuous}
        \halbnorm{p}(\E^{t\Delta_g} a)
        \le
        c'(c\halbnorm{p})_R(a).
    \end{equation}
\end{lemma}
\begin{proof}
    First we extend the operator $\Delta_g$ to the whole tensor
    algebra $\Tensor^\bullet(V)$ as usual by setting
    \[
    \tilde{\Delta}_g (v_1 \tensor \cdots \tensor v_n)
    =
    \sum_{i<j} (-1)^{v_i(v_1 + \cdots + v_{i-1})}
    (-1)^{v_j(v_1 + \cdots v_{i-1} + v_{i+1} + \cdots v_{j-1})}
    g(v_i, v_j)
    v_1 \tensor
    \cdots \stackrel{i}{\wedge} \cdots \stackrel{j}{\wedge} \cdots
    \tensor v_n
    \]
    on factorizing homogeneous tensor and extending linearly.  Then we
    have
    \begin{equation*}
        \Symmetrizer \circ \tilde{\Delta}_g
        =
        \Delta_g \circ \Symmetrizer
        \tag{$*$}
    \end{equation*}
    as already for $P_\Lambda$. With an analogous estimate as for the
    Poisson bracket we get
    \[
    \halbnorm{p}^{n-2}\left(
        \tilde{\Delta}_g (v_1 \tensor \cdots \tensor v_n)
    \right)
    \le
    \sum_{i < j} |g(v_i, v_j)|
    \halbnorm{p}(v_1)
    \cdots \stackrel{i}{\wedge} \cdots \stackrel{j}{\wedge} \cdots
    \halbnorm{p}(v_n)
    \le
    \frac{n(n-1)}{2} \halbnorm{p}(v_1) \cdots \halbnorm{p}(v_n).
    \]
    This implies for all $a_n \in \Tensor^n(V)$ the estimate
    \[
    \halbnorm{p}^{n-2}(\tilde{\Delta}_g a_n)
    \le
    \frac{n(n-1)}{2}
    \halbnorm{p}^n(a_n).
    \]
    Thanks to ($*$) we get the same estimate for $a_n \in \Sym^n(V)$
    and $\Delta_g$ in place of $\tilde{\Delta}_g$. By induction this
    results in
    \begin{equation*}
        \halbnorm{p}^{n - 2k}(\Delta_g^k a_n)
        \le
        \frac{n!}{2^k (n-2k)!} \halbnorm{p}^n(a_n)
        \tag{$**$}
    \end{equation*}
    as long as $n - 2k \ge 0$ and $\Delta_g^k a_n = 0$ for $2k > n$.
    This gives
    \[
    \halbnorm{p}_R(\Delta_g a)
    =
    \sum_{n=2}^\infty
    (n-2)!^R \halbnorm{p}^{n-2}(\Delta_g a_n)
    \le
    \sum_{n=2}^\infty
    (n-2)!^R \frac{n(n-1)}{2} \halbnorm{p}^n(a_n)
    \le
    \sum_{n=0}^\infty n!^R 2^n \halbnorm{p}^n(a_n),
    \]
    which is the first estimate \eqref{eq:gContinuous}. For the second
    we have to be slightly more efficient with the estimates as a
    simple iteration of \eqref{eq:gContinuous} would not suffice. We
    have for $|t| \ge 1$
    \begin{align*}
        \halbnorm{p}_R(\E^{t \Delta_g} a)
        &\le
        \sum_{k = 0}^\infty
        \frac{|t|^k}{k!}
        \halbnorm{p}_R(\Delta_g^k a) \\
        &\le
        \sum_{n, k = 0}^\infty
        \frac{|t|^k}{k!}
        \halbnorm{p}_R(\Delta_g^k a_n)
        \\
        &\stackrel{\mathclap{(**)}}{\le} \;
        \sum_{\substack{n, k = 0 \\ n \ge 2k}}^\infty
        \frac{|t|^k}{k!}
        (n - 2k)!^R \frac{n!}{2^k (n - 2k)!}
        \halbnorm{p}^n(a_n)
        \\
        &=
        \sum_{\substack{n, k = 0 \\ n \ge 2k}}^\infty
        \frac{|t|^k}{2^k k!}
        \left(
            \frac{n!}{(n-2k)!}
        \right)^{1-R}
        n!^R
        \halbnorm{p}^n(a_n)
        \\
        &\stackrel{(a)}{\le}
        \sum_{\substack{n, k = 0 \\ n \ge 2k}}^\infty
        \frac{2^{2k(1 - R)} |t|^k}{2^k}
        \frac{1}{k!^{1- 2(1-R)}}
        n!^R
        2^n
        \halbnorm{p}^n(a_n)
        \\
        &\stackrel{(b)}{\le}
        \sum_{\substack{n, k = 0 \\ n \ge 2k}}^\infty
        \frac{1}{|t|^k k!^{2R - 1}}
        n!^R
        (2|t|)^n
        \halbnorm{p}^n(a_n) \\
        &\le
        \left(
            \sum_{k = 0}^\infty
            \frac{1}{|t|^k k!^{2R - 1}}
        \right)
        (2|t|\halbnorm{p})_R (a_n),
    \end{align*}
    where in ($a$) we use $n! \le 2^n (n-2k)! (2k)!$ and $(2k)! \le
    k!^2 2^{2k}$, and in ($b$) we use the assumption $|t| \ge 1$ as
    well as $k \le 2k \le n$ and $R \ge \frac{1}{2}$. If instead $|t|
    < 1$ then we proceed in ($b$) by
    \[
    \halbnorm{p}_R(\E^{t \Delta_g} a)
    \le \cdots
    \stackrel{(b')}{\le}
    \left(
        \sum_{k = 0}^\infty
        \frac{|t|^k}{k!^{2R - 1}}
    \right)
    (2\halbnorm{p})_R (a_n).
    \]
    In both cases the series over $k$ converge as long as $R \ge
    \frac{1}{2}$ and yield constants $c'$ as required for the estimate
    \eqref{eq:expDeltagContinuous}. The constant $c$ can be taken as
    the maximum of $2|t|$ and $2$.
\end{proof}

We see that the idea of this estimate is rather similar to the one in
Lemma~\ref{lemma:KeyLemma}. These estimates provide now the key to
establish the equivalences also in the analytic framework:
\begin{proposition}
    \label{proposition:EquivalenceContinuous}%
    Let $R \ge \frac{1}{2}$ and let $\Lambda, \Lambda'\colon V \times V
    \longrightarrow \mathbb{K}$ be even continuous bilinear forms such
    that their antisymmetric parts coincide. Then the Weyl algebras
    $\cWeylR(V, \starzLambda)$ and $\cWeylR(V, \star_{z\Lambda'})$ are
    isomorphic via the continuous equivalence transformation
    \begin{equation}
        \label{eq:ContinuousEquivalence}
        \E^{z\Delta_g} (a \starzLambda b)
        =
        \left(\E^{z\Delta_g} a\right)
        \star_{z\Lambda'}
        \left(\E^{z\Delta_g} b\right),
    \end{equation}
    where $g = \Lambda' - \Lambda = \Lambda_+' - \Lambda_+$ and $a, b
    \in \cWeylR(V)$.
\end{proposition}
\begin{proof}
    First we note that the continuity of $\Lambda$ and $\Lambda'$
    implies the continuity of $g$. Moreover, to test the continuity of
    the map $\E^{z\Delta_g}$ it clearly suffices to consider only
    those seminorms $\halbnorm{p}_R$ of $\cWeylR(V, \starzLambda)$
    with $\halbnorm{p}$ being a seminorm such that
    \eqref{eq:gContinuity} holds. Thus we can apply
    Lemma~\ref{lemma:DeltaContinuous} to conclude that
    $\E^{z\Delta_g}$ is continuous on $ \WeylR(V, \starzLambda)$ and
    hence extends to a continuous endomorphism of $\cWeylR(V,
    \starzLambda)$ as well. Then the relation
    \eqref{eq:ContinuousEquivalence} holds for all $a, b \in \WeylR(V,
    \starzLambda)$ by Proposition~\ref{proposition:FormalEquivalence}
    and hence for all $a, b \in \cWeylR(V, \starzLambda)$ by
    continuity. Finally, for a fixed $a \in \cWeylR(V, \starzLambda)$
    the exponential series
    \[
    \E^{t\Delta_g} a = \sum_{k=0}^\infty \frac{t^k}{k!} \Delta_g^k a
    \]
    converges absolutely in the topology of $\cWeylR(V,
    \starzLambda)$. Indeed, this follows from the estimate in the
    proof of Lemma~\ref{lemma:DeltaContinuous}. Thus for $z, w \in
    \mathbb{K}$ we get $\E^{z \Delta_g} \circ \E^{w \Delta_g} =
    \E^{(z+w)\Delta_g}$ at once. This shows that $\E^{z \Delta_g}$ is
    indeed invertible and hence a continuous isomorphism with
    continuous inverse $\E^{-z \Delta_g}$.
\end{proof}
\begin{remark}
    \label{remark:AlsoWeylREquivalent}%
    Note that $\E^{z\Delta_g}$ maps $\Sym^\bullet_R(V)$ into
    itself. Hence also the (non-completed) Weyl algebras $\WeylR(V,
    \starzLambda)$ and $\WeylR(V, \star_{z\Lambda'})$ are isomorphic
    via $\E^{z\Delta_g}$. Analogously, the equivalence transformation
    $\E^{z\Delta_g}$ is a continuous algebra isomorphism between
    $\cWeylRMinus(V, \starzLambda)$ and $\cWeylRMinus(V,
    \star_{z\Lambda'})$ for $R > \frac{1}{2}$. This continuity is
    clearly included in the above estimate
    \eqref{eq:expDeltagContinuous}.
\end{remark}
\begin{remark}
    \label{remark:IsoClassesWeylR}%
    Putting together the results from
    Proposition~\ref{proposition:FunctorialityWeylR} and
    Proposition~\ref{proposition:EquivalenceContinuous} we see that
    the isomorphism class of the Weyl algebra $\cWeylR(V,
    \starzLambda)$ is determined by the isomorphism class of the
    \emph{antisymmetric} part $\Lambda_-$ alone.
\end{remark}
\begin{remark}
    \label{remark:StarEquivalence}%
    In the situation of Proposition~\ref{proposition:Involution} we
    see that also in the $^*$-algebra case, two Weyl algebras
    $\WeylR(V, \starhLambda)$ and $\WeylR(V,
    \star_{\frac{\I\hbar}{2}\Lambda'})$ are $^*$-equivalent via an
    equivalence transformation which is now a $^*$-isomorphism
    provided that the antisymmetric parts of $\Lambda$ and $\Lambda'$
    coincide. The reason is that the symmetric parts are necessarily
    purely imaginary.
\end{remark}

%
% Finite Dimensional $V$: Comparison with existing results
%

\subsection{Finite Dimensional $V$: Comparison with existing results}
\label{subsec:FiniteDimensionalV}

In the finite-dimensional case the situation is very simple: first we
note that there is only one Hausdorff locally convex topology on $V$
and all bilinear maps are continuous. In this situation we get a
defining system of continuous seminorms for the topology of
$\WeylR(V)$ and $\WeylRMinus(V)$ very easily:
\begin{lemma}
    \label{lemma:FiniteDimVSeminorms}%
    Let $V$ be finite-dimensional and let $\halbnorm{p}$ be a norm on
    $V$. Then the norms $\{(c\halbnorm{p})_R\}_{c > 0}$ yield a
    defining system of seminorms for $\WeylR(V)$.
\end{lemma}
\begin{proof}
    Let $\halbnorm{q}$ be an arbitrary seminorm on $V$. Then there is
    a constant $c > 0$ with $\halbnorm{q} \le c \halbnorm{p}$ since
    $\halbnorm{p}$ is a norm and we are in finite dimensions. Then we
    have $\halbnorm{q}^n \le c^n \halbnorm{p}^n$. Hence we also get
    $\halbnorm{q}_R \le (c\halbnorm{p})_R$ from which the
    claim follows.
\end{proof}
\begin{lemma}
    \label{lemma:FiniteDimensionsProCase}%
    Let $V$ be finite-dimensional and let $\halbnorm{p}$ be a norm on
    $V$. Then the norms $\{\halbnorm{p}_{R - \epsilon}\}_{\epsilon
      > 0}$ yield a defining system of seminorms for $\WeylRMinus(V)$.
\end{lemma}
\begin{proof}
    Here we do not even need the multiples of $\halbnorm{p}$. As
    before, for a given seminorm $\halbnorm{q}$ on $V$ there is a
    constant $c > 0$ with $\halbnorm{q} \le c \halbnorm{p}$ and hence
    $\halbnorm{q}^n \le c^n \halbnorm{p}^n$. Now fix $\epsilon' > 0$
    with $\epsilon' < \epsilon$ and let $C > 0$ be a constant such
    that $c^n \le C n!^{\epsilon - \epsilon'}$ for all $n \in
    \mathbb{N}$. Then we have
    \[
    \halbnorm{q}_{R - \epsilon}(a)
    =
    \sum_{n=0}^\infty n!^{R - \epsilon} \halbnorm{q}^n(a_n)
    \le
    \sum_{n=0}^\infty n!^{R - \epsilon} c^n \halbnorm{p}^n(a_n)
    \le
    C \sum_{n=0}^\infty n!^{R - \epsilon'} \halbnorm{p}^n(a_n)
    =
    C \halbnorm{p}_{R - \epsilon'}(a).
    \]
    This shows that we can estimate every seminorm of the form
    $\halbnorm{q}_{R - \epsilon}$ by a suitable $\halbnorm{p}_{R -
      \epsilon'}$.
\end{proof}

Let $V = V_{\mathbf{0}} \oplus V_{\mathbf{1}}$ be finite-dimensional
and real. Moreover, let $\Lambda\colon V \times V \longrightarrow
\mathbb{R}$ be antisymmetric and even. Then $\Lambda =
\Lambda_{\mathbf{0}} + \Lambda_{\mathbf{1}}$ with
\begin{equation}
    \label{eq:LambdaBosonicFermionicPart}
    \Lambda_{\mathbf{0}}\colon
    V_{\mathbf{0}} \times V_{\mathbf{0}} \longrightarrow \mathbb{R}
    \quad
    \textrm{and}
    \quad
    \Lambda_{\mathbf{1}}\colon
    V_{\mathbf{1}} \times V_{\mathbf{1}} \longrightarrow \mathbb{R},
\end{equation}
such that $\Lambda_{\mathbf{0}}$ is an antisymmetric bilinear form on
$V_{\mathbf{0}}$ and $\Lambda_{\mathbf{1}}$ is a symmetric bilinear
form on $V_{\mathbf{1}}$. By the linear Darboux Theorem we can find a
basis $q_1, \ldots, q_d, p_1, \ldots, p_d, c_1, \ldots, c_k$ of
$V_{\mathbf{0}}$ such that the only nontrivial pairing is
\begin{equation}
    \label{eq:DarbouxBasis}
    \Lambda_{\mathbf{0}}(q_i, p_j)
    =
    \delta_{ij}
    =
    - \Lambda_{\mathbf{0}}(p_j, q_i).
\end{equation}
For the odd part, we can find a basis $e_1, \ldots, e_r, f_1,
\ldots, f_s, x_1, \ldots, x_t$ with the only nontrivial pairings
\begin{equation}
    \label{eq:OrthonormalBasis}
    \Lambda_{\mathbf{1}}(e_i, e_j) = \delta_{ij}
    \quad
    \textrm{and}
    \quad
    \Lambda_{\mathbf{1}}(f_i, f_j) = - \delta_{ij}.
\end{equation}
Here $\dim V_{\mathbf{0}} = 2d+k$ and $\dim V_{\mathbf{1}} =
r+s+t$. Then $\Lambda_{\mathbf{0}}$ is symplectic iff $k = 0$ and
$\Lambda_{\mathbf{1}}$ is an (indefinite) inner product iff $t = 0$,
its signature is then given by $(r, s)$. If we use $\Lambda$ directly
for building the star product $\starzLambda$ in this case then we
obtain the usual Weyl-Moyal star product for the even part and a
Clifford multiplication for the odd part. Thus the numbers $d, k, r,
s, t$ encode the isomorphism class of $\WeylR(V, \starzLambda)$ as
well as those of $\WeylRMinus(V, \starzLambda)$ in the
finite-dimensional case. The complex case is analogous.

Let us now use this simple classification to compare our general
construction with a previous construction in finite dimensions: first
we want to relate our construction to the one of
\cite{omori.maeda.miyazaki.yoshioka:2000a}, where the Weyl-Moyal type
star product was considered, i.e. no symmetric contribution to the
symplectic antisymmetric part $\Lambda$.

In the approach of \cite{omori.maeda.miyazaki.yoshioka:2000a}, the
relevant topology on the complexified symmetric algebra
$\Sym^\bullet(V_{\mathbb{R}}) \tensor \mathbb{C} = \mathbb{C}[z^1,
\ldots, z^d, \cc{z}^1, \ldots, \cc{z}^d]$ over $V_{\mathbb{R}} =
\mathbb{R}^{2d}$ is obtained as follows. For a parameter $0 < p \le 2$
the topology is defined by the seminorms
\begin{equation}
    \label{eq:SupSeminorms}
    \norm{a}_{p, s}
    =
    \sup_{x \in \mathbb{C}^{2d}}
    \{\abs{a(x)} \E^{-s \abs{x}^p}\},
\end{equation}
where we denote by $x \in \mathbb{C}^{2d}$ a point in the complexified
vector space and use the obvious extension of $a \in
\Sym^\bullet(V_{\mathbb{R}}) \tensor \mathbb{C}$ as a function
(polynomial) on $\mathbb{C}^{2d}$. Moreover, $\abs{x}$ denotes the
euclidean norm of $x$.  Then the locally convex topology used is the
one determined by all the seminorms $\norm{\argument}_{p, s}$ for all
$s > 0$. In fact, in \cite{omori.maeda.miyazaki.yoshioka:2000a} only
the case of $d = 1$ is considered, but it is clear that everything can
be done in higher (finite) dimensions as well. The following
proposition clarifies the relation between the two approaches, the
proof of which is implicitly contained already in
\cite{omori.maeda.miyazaki.yoshioka:2000a}.
\begin{proposition}
    \label{proposition:OMMYcoincides}%
    Let $0 < p \le 2$. Then the locally convex topology on
    $\Sym^\bullet(V_{\mathbb{R}})\tensor \mathbb{C}$ induced by the
    seminorms $\norm{\argument}_{p, s}$ for $s > 0$ coincides with the
    topology of $\Sym^\bullet_R(V_{\mathbb{R}}) \tensor
    \mathbb{C}$ if we set $R = \frac{1}{p}$.
\end{proposition}
\begin{proof}
    We have to find mutual estimates for the two families of
    seminorms. We begin with some preparatory material. In view of
    Lemma~\ref{lemma:FiniteDimVSeminorms} we are free to chose the
    following $\ell^1$-like norm
    \begin{equation}
        \label{eq:ellEinsHalbnorm}
        \halbnorm{p}(a) = \sum_\alpha \abs{a_\alpha}
    \end{equation}
    with respect to the canonical basis $e_1, \ldots, e_{2d}$ on
    $\mathbb{R}^{2d} \tensor \mathbb{C}$. The reason to chose this
    $\ell^1$-norm is that it behaves most nicely for the tensor
    product. We write a polynomial as
    \begin{equation}
        \label{eq:Polynomiala}
        a = \sum_{n = 0}^\infty \sum_{\alpha_1, \ldots, \alpha_n = 1}^{2d}
        a_{\alpha_1 \ldots \alpha_n} x_{\alpha_1} \cdots x_{\alpha_n},
    \end{equation}
    where $x_\alpha$ are the coordinate functions in $x = \sum_\alpha
    x_\alpha e_\alpha$. The components $a_{\alpha_1 \ldots \alpha_n}$
    are totally symmetric. For the seminorm $\halbnorm{p}^n$ of the
    homogeneous part $a_n$ of $a$ of degree $n$ we get
    \begin{equation*}
        \halbnorm{p}^n(a_n)
        =
        \sum_{\alpha_1, \ldots, \alpha_n = 1}^{2d}
        \abs{a_{\alpha_1 \ldots \alpha_n}}.
        \tag{$*$}
    \end{equation*}
    For a homogeneous monomial we estimate the seminorm
    $\norm{\argument}_{p, s}$, yielding
    \begin{equation*}
        \sup_{x \in \mathbb{C}^{2d}}
        \left\{
            |\abs{x_{\alpha_1} \cdots x_{\alpha_n}} \E^{-s |x|^p}
        \right\}
        \le
        |x|^n \E^{-s |x|^p}
        \le
        \left(\frac{n}{sp}\right)^{\frac{n}{p}} \E^{- \frac{n}{p}},
        \tag{$**$}
    \end{equation*}
    by explicitly computing the maximum value of the scalar function
    $g(r) = r^n \E^{-sr^p}$ for $r \ge 0$, see also \cite[Proof of
    Prop.~6.1]{omori.maeda.miyazaki.yoshioka:2000a}.  With this
    preparation and noting $n^n \le \frac{1}{\E} n! \E^n$ we get
    \[
    \norm{a}_{p, s}
    \le
    \sum_{n=0}^\infty \sum_{\alpha_1, \ldots, \alpha_n = 0}^{2d}
    \abs{a_{\alpha_1 \ldots \alpha_n}}
    \norm{x_{\alpha_1} \cdots x_{\alpha_n}}_{p, s}
    \stackrel{(*), (**)}{\le}
    \sum_{n=0}^\infty
    \halbnorm{p}^n(a_n)
    \left(\frac{n}{sp}\right)^{\frac{n}{p}} \E^{- \frac{n}{p}}
    \le
    \sum_{n=0}^\infty
    \halbnorm{p}^n(a_n)
    n!^{\frac{1}{p}} c^n
    \]
    where $c = \left(\frac{p}{s}\right)^{\frac{1}{p}}$. This gives the
    estimate $\norm{a}_{p, s} \le (c\halbnorm{p})_{\frac{1}{p}}(a)$.
    For the converse estimate we first note that we can apply the
    (multi-variable) Cauchy formula for the polynomial $a$. This gives
    an estimate for the Taylor coefficients following
    \cite{omori.maeda.miyazaki.yoshioka:2000a}: first we have for a
    fixed $r > 0$ the estimate
    \[
    \sum_{\alpha_1, \ldots, \alpha_n = 1}^{2d}
    \abs{a_{\alpha_1 \ldots \alpha_n}}
    \le
    \frac{1}{r^n} \E^{sr^p} \norm{a}_{p, s}.
    \]
    Using the minimum value $(sp)^n n^{-\frac{n}{p}} \E^{\frac{n}{p}}$
    of the scalar function $r^n \E^{sr^p}$ this gives for a fixed $s >
    0$ and $R = \frac{1}{p}$ the estimate
    \begin{align*}
        (c\halbnorm{p})_R(a)
        &=
        \sum_{n=0}^\infty \sum_{\alpha_1, \ldots, \alpha_n = 1}^{2d}
        \abs{a_{\alpha_1 \ldots \alpha_n}} n!^R c^n\\
        &\le
        \sum_{n=0}^\infty n!^R \frac{1}{r^n} \E^{sr^p}
        \norm{a}_{p, s} c^n \\
        &\le
        \sum_{n=0}^\infty
        n!^R (sp)^n n^{-\frac{n}{p}} \E^{\frac{n}{p}}
        \norm{a}_{p, s} c^n \\
        &\le
        \sum_{n=0}^\infty
        \E^R n^R n^{Rn} n^{-Rn} (sp)^n \E^{nR}
        \norm{a}_{p, s} c^n \\
        &=
        \sum_{n=0}^\infty
        \E^R n^R (spc \sqrt[p]{E})^n
        \norm{a}_{p, s}.
    \end{align*}
    Now choosing $s > 0$ sufficiently small such that $spc\sqrt[p]{\E}
    < 1$ gives a converging series over $n$ and hence a suitable
    constant $c' > 0$ with $(c\halbnorm{p})_R(a) \le c' \norm{a}_{p,
      s}$ for $R = \frac{1}{p}$ and the above chosen $s$. Thus the two
    topologies coincide.
\end{proof}
\begin{remark}
    \label{remark:OMMY}%
    Thus our general construction can be seen as a way to extend the
    results of \cite{omori.maeda.miyazaki.yoshioka:2000a} beyond the
    finite-dimensional case. It also shows some additional functorial
    behaviour of the construction which, even in finite dimensions, is
    not immediate from the description using the seminorms in
    \eqref{eq:SupSeminorms}. Finally, several additional and important
    features like the nuclearity and the existence of an absolute
    Schauder basis are now available for the Weyl algebra of
    \cite{omori.maeda.miyazaki.yoshioka:2000a}. The follow-up articles
    \cite{omori.maeda.miyazaki.yoshioka:2002a,
      omori.maeda.miyazaki.yoshioka:2007a} indicate in which direction
    a further investigation of the Weyl algebra should go also in the
    infinite dimensional setting: the exponentiation of quadratic
    elements gives some interesting new phenomena which are worth to
    be studied in the infinite-dimensional setting as
    well. Ultimately, these results will be related to a quantum
    theory of free time evolutions, i.e. for quadratic Hamiltonians.
\end{remark}

The second situation which we shall relate our general Weyl algebra
$\WeylR(V)$ to is the convergent Wick star product as in
\cite{beiser.roemer.waldmann:2007a, beiser.waldmann:2011a:pre}. Here
we are again in the real symplectic situation with $V_{\mathbb{R}} =
\mathbb{R}^{2d}$ and its canonical symplectic form. Using the same
notation as above, the Wick star product
\begin{equation}
    \label{eq:WickStarProduct}
    f \starwick g
    =
    \sum_{N=0}^\infty \frac{(2\hbar)^{|N|}}{N!}
    \frac{\partial^{|N| f}}{\partial z^N}
    \frac{\partial^{|N| g}}{\partial \cc{z}^N}
\end{equation}
from \cite{beiser.roemer.waldmann:2007a, beiser.waldmann:2011a:pre}
can then be written as $\starwick = \starhLambda$ with $\Lambda$ given
by
\begin{equation}
    \label{eq:LambdaForWick}
    \Lambda(z^k, \cc{z}^\ell) = \frac{4}{\I}\delta_{k\ell},
\end{equation}
and all other pairings trivial. In
\cite[Thm.~3.19]{beiser.waldmann:2011a:pre} it was shown that the
previously constructed locally convex topology for the Wick star
product from \cite{beiser.roemer.waldmann:2007a} can be described as
follows: write $a \in \Sym^\bullet(V)$ as Taylor polynomial
\begin{equation}
    \label{eq:TaylorExpansion}
    a
    =
    \sum_{I, J = 0}^\infty
    a_{IJ} \frac{z^I \cc{z}^J}{I! J!}.
\end{equation}
Then the defining system of seminorms is given by
\begin{equation}
    \label{eq:SubFactorial}
    \norm{a}_\epsilon
    =
    \sup_{I, J} \frac{|a_{IJ}|}{|I + J|!^\epsilon},
\end{equation}
where $\epsilon > 0$. With other words, the Taylor coefficients
$a_{IJ}$ have \emph{sub-factorial} growth with respect to the
multiindices $I$ and $J$. Note that in
\cite{beiser.waldmann:2011a:pre} an addition factor $(2\hbar)^{|I| +
  |J|}$ is present in the denominator in
\eqref{eq:TaylorExpansion}. But clearly such an exponential
contribution will not change the sub-factorial growth properties at
all. Therefore the seminorms \eqref{eq:SubFactorial} give the same
topology as the one in \cite{beiser.waldmann:2011a:pre}.
\begin{proposition}
    \label{proposition:WickIsWeylForREins}%
    The locally convex topology on $\Sym^\bullet(V)$ induced by the
    seminorms $\{\norm{\argument}_\epsilon\}_{\epsilon > 0}$ coincides
    with the topology of the Weyl algebra $\WeylRMinus(V)$ for $R =
    1$.
\end{proposition}
\begin{proof}
    To get the combinatorics of the Taylor and tensor coefficients
    right, we note that for a homogeneous $a_n \in \Sym^n(V) \subseteq
    \Tensor^n(V)$ written as in \eqref{eq:Polynomiala} we have
    \[
    \frac{a_{IJ}}{I!J!}
    =
    \sum_{\alpha \in (I, J)} a_{\alpha_1 \ldots \alpha_n},
    \]
    where the summation runs over all those $n$-tuples $\alpha =
    (\alpha_1, \ldots, \alpha_n)$ containing $i_1$ times the index
    $1$, \ldots, $i_d$ times the index $d$, $j_1$ times the index
    $\cc{1}$, \ldots, and $j_d$ times the index $\cc{d}$. Since the
    coefficients are totally symmetric we get
    \begin{equation*}
        \frac{|a_{IJ}|}{I!J!}
        =
        \sum_{\alpha \in (I, J)} |a_{\alpha_1 \ldots \alpha_n}|.
        \tag{$*$}
    \end{equation*}
    The first estimate we need is now
    \begin{align*}
        \norm{a}_\epsilon
        &=
        \sup_{I, J} \frac{|a_{IJ}|}{|I + J|!^\epsilon} \\
        &=
        \sup_n \frac{1}{n!^\epsilon}
        \sup_{\substack{I, J \\ |I + J| = n}} \sum_{\alpha \in (I, J)}
        |a_{\alpha_1 \ldots \alpha_n}| I! J! \\
        &\le
        \sum_{n=0}^\infty \frac{1}{n!^\epsilon}
        \sum_{\alpha_1, \ldots, \alpha_n}
        |a_{\alpha_1 \ldots \alpha_n}| n! \\
        &=
        \halbnorm{p}_{1 - \epsilon}(a),
    \end{align*}
    where $\halbnorm{p}$ is again the norm \eqref{eq:ellEinsHalbnorm}.
    For the other direction we take the same norm $\halbnorm{p}$ and
    chose again a $0 < \epsilon' < \epsilon$. Then
    \begin{align*}
        \halbnorm{p}_{1 - \epsilon}(a)
        &=
        \sum_{n=0}^\infty n!^{1-\epsilon}
        \sum_{\alpha_1, \ldots, \alpha_n}
        |a_{\alpha_1 \ldots \alpha_n}| \\
        &=
        \sum_{n=0}^\infty n!^{1-\epsilon}
        \sum_{|I + J| = n} \frac{|a_{IJ}|}{I!J!} \\
        &\le
        \sum_{n=0}^\infty n!^{1-\epsilon}
        \sum_{|I + J| = n}
        \frac{\norm{a}_{\epsilon'} n!^{\epsilon'}}{I!J!} \\
        &=
        \norm{a}_{\epsilon'}
        \sum_{n=0}^\infty \frac{(2d)^n}{n!^{\epsilon - \epsilon'}},
    \end{align*}
    which gives an estimate $\halbnorm{p}_{R - \epsilon}(a) \le c
    \norm{a}_{\epsilon'}$ with $c$ being the above convergent
    series. Since by Lemma~\ref{lemma:FiniteDimensionsProCase} it suffices
    to consider one norm $\halbnorm{p}$ for $V$, this finishes the
    proof.
\end{proof}
\begin{remark}
    \label{remark:WickIsCool}%
    The construction in \cite{beiser.roemer.waldmann:2007a,
      beiser.waldmann:2011a:pre} still had the flavour of being rather
    ad-hoc. Thanks to Proposition~\ref{proposition:WickIsWeylForREins}
    and the more conceptual definition of the topology of the Weyl
    algebra $\WeylRMinus(V)$ in
    Definition~\ref{definition:WeylAlgebraWR} this flaw
    disappears. Also, the choice of the Wick star product is not
    essential as on $V_{\mathbb{R}} = \mathbb{R}^{2d}$ there is only
    one symplectic form up to isomorphism, see
    Remark~\ref{remark:IsoClassesWeylR}. Finally, together with
    Proposition~\ref{proposition:OMMYcoincides}, this clarifies the
    relation between the two constructions of
    \cite{omori.maeda.miyazaki.yoshioka:2000a} on one side and
    \cite{beiser.roemer.waldmann:2007a, beiser.waldmann:2011a:pre} on
    the other side.
\end{remark}

%
% An Example: the Peierls Bracket and Free QFT
%

\section{An Example: the Peierls Bracket and Free QFT}
\label{sec:ExamplePeierlsFreeQFT}

In this last section we discuss a first example in infinite
dimensions: the Poisson bracket and the corresponding Weyl algebra
underlying a free, i.e. linear, field theory. Our main focus here is
the precise definition of the relevant locally convex topologies as
well as the global aspects of the construction. One has essentially
two possibilities for the Poisson structure: the canonical Poisson
structure built on a Hamiltonian formulation using the initial value
problem and the covariant Poisson structure, also called the Peierls
bracket, built on a Lagrangian approach. In the following, we will
exclusively work in the smooth category, all manifolds and bundles
will be $\Cinfty$.

The material on the Cauchy problem on globally hyperbolic spacetimes
is standard and can be found e.g. in the book
\cite{baer.ginoux.pfaeffle:2007a}. The comparison of the canonical and
the covariant Poisson brackets is folklore and was taken from
\cite[Sect.~4.4]{waldmann:2012a:pre}. For a much more far-reaching
discussion of the Peierls bracket including also non-linear field
equations we refer to \cite{brunetti.fredenhagen.ribeiro:2012a:pre}:
in fact, it would be a very interesting project to combine the results
from \cite{brunetti.fredenhagen.ribeiro:2012a:pre} on the classical
side with the nuclear Weyl algebra quantization to obtain the
corresponding quantum side.

%
% The Geometric Framework
%

\subsection{The Geometric Framework}
\label{subsec:GeometricFramework}

We consider an $n$-dimensional connected Lorentz manifold $(M, g)$
with a Lorentz metric $g$ of signature $(+, -, \ldots, -)$. The
important concept we need is the causal structure: first we require
that $(M, g)$ is time-orientable and time-oriented. Then one defines
the causal future $J_M^+(p)$ of a point $p \in M$ to be the set of
those points which can be reached by a future-directed causal
curve. Analogously, $J_M^-(p)$ denotes the causal past of $p$. For two
points $p, q \in M$ one defines the diamond $J_M(p, q) = J_M^+(p) \cap
J_M^-(q)$. For an arbitrary subset $A \subseteq M$ we set
\begin{equation}
    \label{eq:JpmMDiverse}
    J_M^\pm(A) = \bigcup_{p \in A} J_M^\pm(A)
    \quad
    \textrm{and}
    \quad
    J_M(A) = J_M^+(A) \cup J_M^-(A).
\end{equation}

A time-oriented Lorentz manifold is called \emph{globally hyperbolic}
if it is causal, i.e. there are no closed causal loops, and if all
diamonds are compact. A first consequence is that $J_M^\pm(p)$ is
always a closed subset of $M$. In the following we always assume that
$(M, g)$ is globally hyperbolic.  A celebrated theorem of Bernal and
Sánchez, refining a topological statement of Geroch, states that this
is equivalent to the existence of a smooth spacelike Cauchy surface
\begin{equation}
    \label{eq:CauchySurface}
    \iota\colon \Sigma \longrightarrow M
\end{equation}
together with a smooth Cauchy temporal function $t \in \Cinfty(M)$,
i.e. the gradient of $t$ is future directed and time-like everywhere
and the level sets of $t$ are smooth spacelike Cauchy surfaces for all
times. Moreover, $M$ is diffeomorphic to the product manifold
$\mathbb{R} \times \Sigma$ with the metric being
\begin{equation}
    \label{eq:MetricFactorizes}
    g = \beta \D t^2 - g_t,
\end{equation}
where $\beta \in \Cinfty(\mathbb{R} \times \Sigma)$ is positive and
$g_t$ is a Riemannian metric on $\Sigma$ smoothly depending on
$t$. Finally, the Cauchy temporal function $t$ can be chosen in such a
way that $\Sigma$ is the $t=0$ level set. For a detailed discussion
see the review \cite{minguzzi.sanchez:2008a}.

Since $M$ is diffeomorphic to $\mathbb{R} \times \Sigma$ we get a
global vector field $\frac{\partial}{\partial t}$ on $M$. Normalizing
this to a unit vector field gives
\begin{equation}
    \label{eq:NormalVectorField}
    \mathfrak{n}
    =
    \frac{1}{\sqrt{\beta}} \frac{\partial}{\partial t}
    \in \Secinfty(TM),
\end{equation}
which is a future-directed time-like unit vector field such that it is
normal to every level surface of the Cauchy temporal function. In
particular $\iota^\# \mathfrak{n} \in \Secinfty(TM_\Sigma)$ will be
normal to the Cauchy surface $\Sigma$. Here $TM_\Sigma = \iota^\# TM$
is the restriction (pull-back via $\iota$) of the tangent bundle to
$\Sigma$ and $\iota^\# \mathfrak{n} = \mathfrak{n} \at{\Sigma}$ is the
pull-back of $\mathfrak{n}$ to $\Sigma$.

The fields we are interested in will be modeled by sections of a
vector bundle over $M$: we require the vector bundle $E
\longrightarrow M$ to be real and equipped with a fiber metric $h$,
not necessarily positive definite but non-degenerate. The dynamics of
the field is now governed by a second order differential operator $D
\in \Diffop^2(E)$ with the following property: there is a metric
connection $\nabla^E$ for $(E, h)$ and a zeroth order differential
operator $B \in \Diffop^0(E) = \Secinfty(\End(E))$ such that
\begin{equation}
    \label{eq:DdAlembert}
    D = \dAlembert^\nabla + B,
\end{equation}
where $\dAlembert^\nabla$ denotes the d'Alembert operator obtained
from $\nabla^E$ and paring with the metric $g$. In particular, $D$ is
\emph{normally hyperbolic}. Conversely, note that for a normally
hyperbolic differential operator there is a unique connection
$\nabla^E$ and a unique $B \in \Secinfty(\End(E))$ such that
\eqref{eq:DdAlembert} holds, see
e.g. \cite[Lem.~1.5.5]{baer.ginoux.pfaeffle:2007a}. Thus the only
additional requirement we need is that $\nabla^E$ is also metric with
respect to $h$.

Let $\mu_g \in \Secinfty(|\Anti^n| T^*M)$ be the canonical metric
density induced by $g$ which we shall use for various integrations
over $M$. First we can use $\mu_g$ to define the \emph{transpose} of a
differential operator $D \in \Diffop^\bullet(E)$ to be the unique
differential operator $D^\Trans \in \Diffop^\bullet(E^*)$ such that
\begin{equation}
    \label{eq:TransposeD}
    \int_M (D^\Trans \varphi) \cdot u \; \mu_g
    =
    \int_M \varphi \cdot (D u) \; \mu_g
\end{equation}
for all $\varphi \in \Secinfty(E^*)$ and $u \in \Secinfty(E)$, at
least one of them having compact support. Here $\cdot$ means pointwise
natural pairing. Note that $D^\Trans$ depends on the choice of $\mu_g$
and has the same order as $D$. Taking into account also the fiber
metric $h$ we can define the \emph{adjoint} of $D$ to be the unique
differential operator $D^* \in \Diffop^\bullet(E)$, again of the same
order as $D$, such that
\begin{equation}
    \label{eq:AdjointD}
    \int_M h(D^*u, v) \; \mu_g
    =
    \int_M h(u, Dv) \; \mu_g
\end{equation}
for $u, v \in \Secinfty(E)$, at least one of them having compact
support. Denoting the musical isomorphisms induced by $h$ by
$\sharp\colon E^* \longrightarrow E$ and $\flat\colon E
\longrightarrow E^*$ as usual, we get $D^*u = (D^\Trans
u^\flat)^\sharp$ for $u \in \Secinfty(E)$.

This allows now to formulate the last requirement on $D =
\dAlembert^\nabla + B$, namely we need $D$ to be a \emph{symmetric}
operator, i.e.
\begin{equation}
    \label{eq:Dsymmetric}
    D^* = D.
\end{equation}
Since the connection $\nabla^E$ is required to be metric, it is easy
to see that \eqref{eq:Dsymmetric} is equivalent to $B^* = B$.

%
% The Wave Equation and Green Operators
%

\subsection{The Wave Equation and Green Operators}
\label{subsec:WaveEquationGreenOperators}

Let $D$ be a normally hyperbolic differential operator as before. Then
the \emph{wave equation} we are interested in is simply given by
\begin{equation}
    \label{eq:Wave}
    Du = 0
\end{equation}
for a section $u$ of $E$. Depending on the regularity of $u$ we can
interpret \eqref{eq:Wave} as a pointwise equation or as an equation in
a distributional sense. Dualizing, we have the corresponding wave
equation
\begin{equation}
    \label{eq:WaveTrans}
    D^\Trans \varphi = 0
\end{equation}
for a section $\varphi$ of the dual bundle $E^*$.

Under our general assumption that $(M, g)$ is globally hyperbolic we
have the existence and uniqueness of advanced and retarded \emph{Green
  operators}
\begin{equation}
    \label{eq:GreenOperators}
    G_M^\pm\colon \Secinfty_0(E) \longrightarrow \Secinfty(E)
\end{equation}
for $D$. This means that there are unique, linear, and continuous maps
$G_M^\pm$ such that
\begin{equation}
    \label{eq:DGisId}
    D G_M^\pm = \id_{\Secinfty_0(E)} = G_M^\pm D\at{\Secinfty_0(E)}
\end{equation}
and
\begin{equation}
    \label{eq:GreenCausal}
    \supp G_M^\pm \subseteq J_M^\pm(\supp u)
\end{equation}
for all $u \in \Secinfty_0(E)$. The continuity refers to the usual LF
and Fréchet topologies of $\Secinfty_0(E)$ and $\Secinfty(E)$,
respectively. Using the volume density $\mu_g$ we can identify the
distributional sections $\Sec[-\infty](E^*)$ with the dual
$\Secinfty_0(E)'$ and $\Sec[-\infty]_0(E^*)$ becomes identified with
the dual $\Secinfty(E)'$.

We need the following space of sections: Let $K \subseteq M$ be
compact. Then denote by $\Secinfty_{J_M(K)}(E)$ those sections in
$\Secinfty(E)$ with $\supp u \subseteq J_M(K)$. Since on a globally
hyperbolic spacetime $J_M(K)$ is a closed subset,
$\Secinfty_{J_M(K)}(E) \subseteq \Secinfty(E)$ is a closed subspace
and hence a Fréchet space itself. Moreover, for $K \subseteq K'$ we
have $\Secinfty_{J_M(K)}(E) \subseteq \Secinfty_{J_M(K')}(E)$ and the
thereby induced topology on $\Secinfty_{J_M(K)}(E)$ coincides with the
original. Hence we can consider the inductive limit
\begin{equation}
    \label{eq:Secsc}
    \Secinftysc(E)
    =
    \bigcup_{\substack{K \subseteq M \\ K \; \mathrm{compact}}}
    \Secinfty_{J_M(K)}(E)
\end{equation}
of those smooth sections of $E$ which have compact support in
spacelike directions. It is a strict inductive limit, and since we can
exhaust $M$ with a sequence of compact subsets, it is a countable
strict inductive limit, endowing $\Secinftysc(E)$ with a LF
topology. The continuity statement \eqref{eq:GreenOperators} can then
be sharpened to the statement that
\begin{equation}
    \label{eq:GreenContinuousSC}
    G_M^\pm\colon
    \Secinfty_0(E) \longrightarrow \Secinftysc(E)
\end{equation}
is continuous. In fact, this follows in a straightforward manner from
the continuity of \eqref{eq:GreenOperators} and the causality
condition \eqref{eq:GreenCausal}.
\begin{remark}
    \label{remark:SecscNuclear}%
    Being a close subspace of a nuclear Fréchet space,
    $\Secinfty_{J_M(K)}(E)$ is nuclear itself. Hence the countable
    strict inductive limit $\Secinftysc(E)$ is again nuclear by
    \cite[Cor.~21.2.3]{jarchow:1981a}.
\end{remark}

We consider now the \emph{propagator} which is defined by
\begin{equation}
    \label{eq:Propagator}
    G_M = G_M^+ - G_M^-\colon
    \Secinfty_0(E) \longrightarrow \Secinftysc(E),
\end{equation}
for which one has the following crucial properties: the sequence
\begin{equation}
    \label{eq:ExactSequence}
    0
    \longrightarrow
    \Secinfty_0(E)
    \stackrel{D}{\longrightarrow}
    \Secinfty_0(E)
    \stackrel{G_M}{\longrightarrow}
    \Secinftysc(E)
    \stackrel{D}{\longrightarrow}
    \Secinftysc(E)
\end{equation}
of continuous linear maps is \emph{exact}. Note that the exactness
relies crucially on the assumption that $(M, g)$ is globally
hyperbolic.

The Green operators can now be used to give a solution to the Cauchy
problem of the wave equation \eqref{eq:Wave}. For the time $t = 0$
level surface $\Sigma$ we want to specify initial conditions $u_0,
\dot{u}_0 \in \Secinfty_0(E_\Sigma)$. Then we want to find a section
$u \in \Secinfty(E)$ with
\begin{equation}
    \label{eq:CauchyProblem}
    Du = 0,
    \quad
    \iota^\# u = u_0,
    \quad
    \textrm{and}
    \quad
    \iota^\# \nabla^E_{\mathfrak{n}} u = \dot{u}_0.
\end{equation}
A core result in the globally hyperbolic case is that this is indeed a
well-posed Cauchy problem: for any $(u_0, \dot{u}_0)$ we have a unique
solution $u$ of the Cauchy problem \eqref{eq:CauchyProblem} such that
the map
\begin{equation}
    \label{eq:SolutionMap}
    \Secinfty_0(E_\Sigma) \oplus \Secinfty_0(E_\Sigma)
    \ni (u_0, \dot{u}_0)
    \; \mapsto \;
    u \in
    \Secinftysc(E)
\end{equation}
is continuous and $\supp u \subseteq J_M(\supp u_0 \cup \supp
\dot{u}_0)$. Moreover, this solution $u$ can be characterized by the
formula
\begin{equation}
    \label{eq:Solution}
    \int_M \varphi \cdot u \; \mu_g
    =
    \int_\Sigma
    \left(
        \iota^\#\left(\nabla_{\mathfrak{n}}^E F_M(\varphi)\right)
        \cdot u_0
        -
        \iota^\#\left(F_M(\varphi)\right)
        \cdot \dot{u}_0
    \right) \; \mu_\Sigma,
\end{equation}
where $F_M$ is the propagator of $D^\Trans$ and $\varphi \in
\Secinfty_0(E^*)$. The density $\mu_\Sigma$ is the one induced by
$\mu_g$.

Since we assume that $D^* = D$ we have a last property of the Green
operators, namely
\begin{equation}
    \label{eq:GreenSymmetric}
    \left(G_M^\pm\right)^* = G_M^\mp
    \quad
    \textrm{and}
    \quad
    G_M^* = - G_M.
\end{equation}
This antisymmetry of the propagator will play a crucial role in the
definition of the covariant Poisson bracket. Moreover, for all
$\varphi \in \Secinfty_0(E^*)$ we have
\begin{equation}
    \label{eq:GMpmsharp}
    G_M^\pm(\varphi^\sharp) = \left(F_M^\pm(\varphi)\right)^\sharp.
\end{equation}

The results and their proofs in this section as well as many more
additional features of the Cauchy problem of the wave equation on a
globally hyperbolic spacetime can be found in the beautiful book
\cite{baer.ginoux.pfaeffle:2007a}, some additional remarks can be
found in the lecture notes \cite{waldmann:2012a:pre}.

%
% The Canonical Poisson Algebra
%

\subsection{The Canonical Poisson Algebra}
\label{subsec:CanonicalPoissonAlgebra}

The canonical i.e. Hamiltonian approach uses an algebra of functions
on the initial data, which constitute the classical phase space
\begin{equation}
    \label{eq:PhaseSpace}
    \mathcal{P}_\Sigma
    =
    \Secinfty_0(E_\Sigma) \oplus \Secinfty_0(E_\Sigma).
\end{equation}
We view $\mathcal{P}_\Sigma$ as symplectic vector space via the
symplectic form
\begin{equation}
    \label{eq:SymplecticForm}
    \omega_\Sigma\left(
        (u_0, \dot{u}_0),
        (v_0, \dot{v}_0)
    \right)
    =
    \int_\Sigma
    \left(
        h_\Sigma (u_0, \dot{v}_0) - h_\Sigma(\dot{u}_0, v_0)
    \right)
    \;
    \mu_\Sigma,
\end{equation}
where $h_\Sigma$ is the restriction of $h$ to $E\at{\Sigma}$.  We have
the following basic result:
\begin{lemma}
    \label{lemma:omegaSymplectic}%
    The two-form $\omega_\Sigma$ on $\mathcal{P}_\Sigma$ is
    antisymmetric, non-degenerate, and continuous.
\end{lemma}
\begin{proof}
    The non-degeneracy and the antisymmetry are clear. For the
    continuity we can rely on several standard arguments: first we
    note that every vector bundle can be written as a subbundle of a
    suitable trivial vector bundle $\Sigma \times \mathbb{R}^N$. This
    gives an identification $\Secinfty_0(E_\Sigma) \subseteq
    \Secinfty_0(\Sigma \times \mathbb{R}^N)$ as a \emph{closed
      embedded} subspace. We can extend $h_\Sigma$ in some way to a
    smooth fiber metric on the trivial bundle and this way,
    $\omega_\Sigma$ is just the restriction of the corresponding
    symplectic form on $\Secinfty(\Sigma \times \mathbb{R}^N)$. Thus
    it suffices to consider a trivial bundle from the beginning. There
    we have $\Secinfty_0(\Sigma \times \mathbb{R}^N) \cong
    \Cinfty_0(\Sigma)^N$. Thus we have to show the continuity of a
    bilinear map of the form
    \begin{equation*}
        \Cinfty_0(\Sigma)^N \times \Cinfty_0(\Sigma)^N
        \ni
        \left(
            (u_i), (v_i)
        \right)
        \; \mapsto \;
        \sum_{i, j = 1}^N u_i H_{ij} v_j
        \in \Cinfty_0(\Sigma),
        \tag{$*$}
    \end{equation*}
    where $H_{ij} \in \Cinfty(\Sigma)$. But since the multiplication
    of compactly supported smooth functions is continuous (and not
    just separately continuous) the continuity of ($*$) follows. The
    final integration needed for \eqref{eq:SymplecticForm} is
    continuous as well.
\end{proof}
Note that it is obvious that $\omega_\Sigma$ is separately continuous,
however, we are interested in continuity. In the case where $\Sigma$
is compact, this would follow directly from separate continuity as
then $\Secinfty_0(E_\Sigma) = \Secinfty(E_\Sigma)$ is a Fréchet space.
But of course, the case of a non-compact $\Sigma$ is of interest, too.

The idea is now to look at certain polynomial functions on
$\mathcal{P}_\Sigma$ and endow them with the Poisson bracket
originating from $\omega_\Sigma$. It turns out that the symmetric
algebra over the dual $\mathcal{P}_\Sigma'$ will be too big and
problematic when it comes to the comparison with the covariant Poisson
structure. Hence we decide here for a rather small piece of all
polynomials, namely for the symmetric algebra over
\begin{equation}
    \label{eq:VSigma}
    V_\Sigma
    =
    \Secinfty_0(E^*_\Sigma) \oplus \Secinfty_0(E^*_\Sigma).
\end{equation}
Using the density $\mu_\Sigma$ we can indeed pair elements from
$V_\Sigma$ with points in $\mathcal{P}_\Sigma$:
\begin{lemma}
    \label{lemma:EvaluateContinuous}%
    The integration
    \begin{equation}
        \label{eq:EvaluateVSigmaOnPSigma}
        (\varphi_0, \dot{\varphi}_0) (u_0, \dot{u}_0)
        =
        \int_\Sigma
        \left(
            \varphi_0 \cdot u_0 + \dot{\varphi}_0 \cdot \dot{u}_0
        \right)
        \;
        \mu_\Sigma
    \end{equation}
    provides a continuous bilinear pairing between $V_\Sigma$ and
    $\mathcal{P}_\Sigma$.
\end{lemma}
The proof of the continuity is analogous to the one in
Lemma~\ref{lemma:omegaSymplectic}. In particular, we can view points
in $\mathcal{P}_\Sigma$ as elements of the dual of $V_\Sigma$ and vice
versa.
\begin{lemma}
    \label{lemma:PoissonBracketSigma}%
    The symplectic form $\omega_\Sigma$ induces a non-degenerate
    antisymmetric continuous bilinear form
    \begin{equation}
        \label{eq:LambdaSigma}
        \Lambda_\Sigma\colon
        V_\Sigma \times V_\Sigma \longrightarrow \mathbb{R},
    \end{equation}
    explicitly given by
    \begin{equation}
        \label{eq:LambdaSigmaExpl}
        \Lambda_\Sigma\left(
            (\varphi_0, \dot{\varphi}_0),
            (\psi_0, \dot{\psi}_0)
        \right)
        =
        \int_\Sigma
        \left(
            h_\Sigma^{-1}(\varphi_0, \dot{\psi}_0)
            -
            h_\Sigma^{-1}(\dot{\varphi}_0, \psi_0)
        \right)
        \;
        \mu_\Sigma.
    \end{equation}
\end{lemma}
Here $h_\Sigma^{-1}$ stands for the induced fiber metric on
$E^*_\Sigma$ and the Poisson bracket is determined by $\omega_\Sigma$
in the sense that the Hamiltonian vector field of the linear function
$(\varphi_0, \dot{\varphi}_0)$ on $\mathcal{P}_\Sigma$ is determined
via $\omega_\Sigma$ and the Poisson bracket is determined by the
Hamiltonian vector field as usual.

We can now use the Poisson bracket $\BracketSigma{\argument,
  \argument}$ for the symmetric algebra $\Sym^\bullet(V_\Sigma)$ as
described in
Section~\ref{subsec:ConstantPoissonStructuresFormalStarProducts}
together with its quantization given by the star product $\starSigma =
\star_{\frac{\I\hbar}{2}\Lambda_\Sigma}$ for the corresponding nuclear
Weyl algebra $\WeylR(V_\Sigma \tensor \mathbb{C})$. This will be the
canonically quantized model of our Hamiltonian picture of the field
theory. Since we started with a real vector bundle, the resulting Weyl
algebra carries the complex conjugation as $^*$-involution. Note
however that up to now we only described the kinematic part, the field
equation did not yet enter at all.

%
% The Covariant Poisson Algebra
%

\subsection{The Covariant Poisson Algebra}
\label{subsec:CovariantPoissonAlgebra}

As the covariant ``phase space'' we take simply all possible field
configurations on the spacetime, i.e.
\begin{equation}
    \label{eq:CovariantPhaseSpace}
    \mathcal{P}_\cov = \Secinftysc(E),
\end{equation}
whether or not they satisfy the wave equation.  This will not be a
symplectic vector space in any reasonable sense as $\mathcal{P}_\cov$
contains all the unwanted field configurations as well. Nevertheless,
and this is perhaps the surprising observation, the symmetric algebra
over its dual allows for a Poisson bracket: again, we take only a
small part of the dual, namely $\Secinfty_0(E^*)$, where we evaluate
$\varphi \in \Secinfty_0(E^*)$ on $u \in \Secinftysc(E)$ by means of
the integration with respect to $\mu_g$ as usual. As before, we denote
this integration simply by $\varphi(u)$.

The Poisson bracket will then be determined by a bilinear form on
$\Secinfty_0(E^*)$ as before. Using the propagator $F_M$ of $D^\Trans$
we define
\begin{equation}
    \label{eq:LambdaCov}
    \Lambda_\cov (\varphi, \psi)
    =
    \int_M h^{-1}(F_M(\varphi), \psi) \; \mu_g.
\end{equation}
Note that the compact support of $\psi$ makes this integration
well-defined. Moreover, we have the following property:
\begin{lemma}
    \label{lemma:LambdaCovContinuous}%
    The bilinear form $\Lambda_\cov\colon \Secinfty_0(E^*) \times
    \Secinfty_0(E^*) \longrightarrow \mathbb{R}$ is antisymmetric and
    continuous.
\end{lemma}
\begin{proof}
    The antisymmetry is clear since $D^\Trans$ is symmetric and hence
    \eqref{eq:GreenSymmetric} applies also to $F_M$. The continuity is
    slightly more involved: first we note that $F_M\colon
    \Secinfty_0(E^*) \longrightarrow \Secinfty(E^*)$ is continuous by
    the continuity of the Green operators $F_M^\pm$. Next, we use the
    fact that the inclusion $\Secinfty(E^*) \longrightarrow
    \Secinfty_0(E^*)'$ given by the integration with respect to
    $\mu_g$ using $h^{-1}$ is also continuous where we equip the dual
    $\Secinfty_0(E^*)'$ with the \emph{strong} topology. This shows
    that the corresponding ``musical'' homomorphism
    \[
    \sharp_\cov\colon
    \Secinfty_0(E^*) \ni \varphi
    \; \mapsto \;
    \Lambda_\cov(\varphi, \argument) \in \Secinfty_0(E^*)'
    \]
    is continuous with respect to the LF and the strong topology,
    respectively. Hence the Kernel Theorem for the nuclear space
    $\Secinfty_0(E^*)$ states that $\Lambda_\cov(\argument,
    \argument)$ is a distribution on the Cartesian product, or,
    equivalently, a continuous bilinear map, see
    e.g. \cite[Thm.~21.6.9]{jarchow:1981a}.
\end{proof}
\begin{remark}
    \label{remark:CompactSigmaDoesNotHelp}%
    Contrary to the continuity of $\Lambda_\Sigma$ the above argument
    does not simplify for a compact Cauchy surface $\Sigma$ since a
    globally hyperbolic spacetime is always non-compact.
\end{remark}
\begin{definition}[Covariant Poisson algebra]
    \label{definition:CovariantPoissonAlgebra}%
    The covariant Poisson algebra for $D$ is the symmetric algebra
    $\Sym^\bullet(V_\cov)$, where $V_\cov = \Secinfty_0(E^*)$, with the
    constant Poisson bracket $\Bracketcov{\argument, \argument}$
    coming from $\Lambda_\cov$.
\end{definition}

This is indeed a Poisson algebra with a continuous Poisson bracket if
we endow it with one of the topologies discussed in
Section~\ref{subsec:TopologyForSymV}, see
Proposition~\ref{proposition:ContinuityOfPossionBracket} and
Remark~\ref{remark:PoissonBracketContinuousForWeylTop}.  A first and
heuristic appearance of this Poisson bracket in a very particular case
seems to be \cite{peierls:1952a}.

From our general theory we know that the corresponding \emph{covariant
  Weyl algebra} $\WeylR(V_\cov \tensor \mathbb{C}, \starcov)$ with the
\emph{covariant star product} $\starcov = \star_{\frac{\I\hbar}{2}
  \Lambda_\cov}$ is a nuclear $^*$-algebra with respect to the complex
conjugation, where as usual $R \ge \frac{1}{2}$.

As a first result we note that $\Lambda_\cov$ is now degenerate. In
fact, we can determine its degeneracy space explicitly
\cite[Lem.~4.4.18]{waldmann:2012a:pre}:
\begin{lemma}
    \label{lemma:Casimirs}%
    Let $\varphi \in \Secinfty_0(E^*)$. Then the following statements
    are equivalent:
    \begin{lemmalist}
    \item \label{item:Casimir} $\varphi$ is a Casimir element of
        $\Sym^\bullet(V_\cov)$, i.e. $\Bracketcov{\varphi, \argument}
        = 0$.
    \item \label{item:VanishesOnSolutions} $\varphi$ vanishes on
        solutions $u \in \Secinftysc(E)$, i.e. we have
        \begin{equation}
            \label{eq:VanishesOnSolutions}
            \int_M \varphi \cdot u \; \mu_g = 0
            \quad
            \textrm{whenever}
            \quad
            Du = 0.
        \end{equation}
    \item \label{item:KernelOfFM} $\varphi \in \ker F_M$.
    \end{lemmalist}
\end{lemma}
\begin{proof}
    Assume $\Bracketcov{\varphi, \argument} = 0$ then we have $0 =
    \Bracketcov{\varphi, \psi} = \int_M h^{-1}(F_M(\varphi), \psi) \;
    \mu_g$ for all $\psi \in \Secinfty_0(E^*)$. Since $h^{-1}$ is
    non-degenerate this implies $F_M(\varphi) = 0$. Next, assume
    $F_M(\varphi) = 0$. Then we know $\varphi = D^\Trans \chi$ for
    some $\chi \in \Secinfty_0(E^*)$ by \eqref{eq:ExactSequence}
    applied to $D^\Trans$. Thus \eqref{eq:VanishesOnSolutions} follows
    by definition of $D^\Trans$ as in \eqref{eq:TransposeD}. Finally,
    assume \eqref{eq:VanishesOnSolutions} and let $\psi \in
    \Secinfty_0(E^*)$ be arbitrary. Then $(F_M(\psi))^\sharp =
    G_M(\psi^\sharp)$ by \eqref{eq:GMpmsharp} and it solves the wave
    equation $D G_M(\psi^\sharp) = 0$ by
    \eqref{eq:ExactSequence}. Thus $\Bracketcov{\varphi, \psi} = 0$
    follows. Since $\Sym^\bullet(V_\cov)$ is generated by $V_\cov$ and
    $\Bracketcov{\varphi, \argument}$ is a derivation,
    $\Bracketcov{\varphi, \argument} = 0$ follows.
\end{proof}

Since the elements of $\ker F_M \subseteq \Secinfty_0(E^*)$ are
Casimir elements, the ideal they generate inside
$\Sym^\bullet(V_\cov)$ is a Poisson ideal. It turns out that it is
even a two-sided ideal with respect to $\starcov$:
\begin{lemma}
    \label{lemma:StarCovIdeal}%
    Let $\langle\ker F_M\rangle \subseteq \Sym^\bullet(V_\cov)$ be the
    ideal generated by $\ker F_M$ with respect to the symmetric tensor
    product. Then we have:
    \begin{lemmalist}
    \item \label{item:kerFMPoissonIdeal} $\langle\ker F_M\rangle$ is a
        Poisson ideal with respect to $\Bracketcov{\argument,
          \argument}$.
    \item \label{item:kerFMStarcovIdeal} $\langle\ker F_M\rangle
        \tensor \mathbb{C} \subseteq \WeylR(V_\cov \tensor
        \mathbb{C})$ is a $^*$-ideal for $\starcov$, in fact generated
        by $\ker F_M$.
    \end{lemmalist}
\end{lemma}
\begin{proof}
    The first part is clear by Lemma~\ref{lemma:Casimirs},
    \refitem{item:Casimir}. Now let $\Phi \in \WeylR(V_\cov \tensor
    \mathbb{C})$ be an arbitrary tensor and let $\varphi \in
    \Secinfty_0(E^*)$. Then we have
    \[
    \Phi \starcov \varphi
    =
    \Phi \varphi + \frac{\I\hbar}{2} \Bracketcov{\Phi, \varphi}
    \quad
    \textrm{and}
    \quad
    \varphi \starcov \Phi
    =
    \Phi \varphi + \frac{\I\hbar}{2} \Bracketcov{\varphi, \Phi},
    \]
    since $\varphi$ has tensor degree $1$ and hence the higher order
    contributions in $\starcov$ all vanish. Thus for $\varphi \in \ker
    F_M$ we get $\Phi \starcov \varphi = \Phi\varphi = \varphi
    \starcov \Phi$. But this shows that
    \[
    \langle \ker F_M \rangle \tensor \mathbb{C}
    =
    \WeylR(V_\cov \tensor \mathbb{C})
    \starcov
    \ker F_M
    \starcov
    \WeylR(V_\cov \tensor \mathbb{C}).
    \]
    Since $\ker F_M$ consists of real sections, it is clear that
    $\langle \ker F_M \rangle \tensor \mathbb{C}$ is a $^*$-ideal.
\end{proof}

%
% The Relation between the Canonical and the Covariant Poisson Algebra
%

\subsection{The Relation between the Canonical and the Covariant Poisson Algebra}
\label{subsec:RelationCanonicalCovariantPoissonAlgebra}

Let us now relate the two Poisson algebras $\Sym^\bullet(V_\Sigma)$
and $\Sym^\bullet(V_\cov)$. In view of \eqref{eq:Solution} it is
reasonable to relate a section $\varphi \in \Secinfty_0(E^*)$ to
sections $\varphi_0, \dot{\varphi}_0 \in \Secinfty_0(E^*_\Sigma)$ by
defining
\begin{equation}
    \label{eq:NullViaRho}
    \varphi_0
    =
    \iota^\#\left(\nabla^{E^*}_{\mathfrak{n}} F_M(\varphi)\right)
    \quad
    \textrm{and}
    \quad
    \dot{\varphi}_0
    =
    - \iota^\#\left(F_M(\varphi)\right),
\end{equation}
thereby defining a linear map
\begin{equation}
    \label{eq:rhoSigmaDef}
    \varrho_\Sigma\colon
    \Secinfty_0(E^*) \ni \varphi
    \; \mapsto \;
    (\varphi_0, \dot{\varphi}_0) \in
    \Secinfty_0(E^*_\Sigma) \oplus \Secinfty_0(E^*_\Sigma),
\end{equation}
with $\varphi_0$, $\dot{\varphi}_0$ given as in
\eqref{eq:NullViaRho}. This map $\varrho_\Sigma$ has the following
property:
\begin{lemma}
    \label{lemma:RhoSigmaEvaluation}%
    The map $\varrho_\Sigma$ is continuous and for all solutions $u
    \in \Secinftysc(E)$ of the wave equation with initial conditions
    $u_0, \dot{u}_0$ we have
    \begin{equation}
        \label{eq:PairingRhoSigma}
        \varphi(u)
        =
        \varrho_\Sigma(\varphi) (u_0, \dot{u}_0).
    \end{equation}
\end{lemma}
\begin{proof}
    The propagator $F_M$ gives a continuous map into
    $\Secinftysc(E^*)$ and the covariant derivative is clearly
    continuous, too, mapping $\Secinftysc(E^*)$ into
    $\Secinftysc(E^*)$. For every compact subset $K \subseteq \Sigma$
    the restriction
    \[
    \iota^\#\colon
    \Secinfty_{J_M(K)}(E^*) \longrightarrow \Secinfty_K(E^*_\Sigma)
    \]
    is a continuous map between Fréchet spaces. But then also
    $\iota^\#\colon \Secinftysc(E^*) \longrightarrow
    \Secinfty_0(E^*_\Sigma)$ is continuous by the universal property
    of LF topologies. This shows the continuity of $\varrho_\Sigma$,
    the equality in \eqref{eq:PairingRhoSigma} is just
    \eqref{eq:Solution}.
\end{proof}

Since $\Sym^\bullet(V_\cov)$ is freely generated by $V_\cov$ we get a
unique unital algebra homomorphism extending $\varrho_\Sigma$ which we
still denote by the same symbol
\begin{equation}
    \label{eq:RhoSigma}
    \varrho_\Sigma\colon
    \Sym^\bullet(V_\cov) \longrightarrow \Sym^\bullet(V_\Sigma).
\end{equation}
Since \eqref{eq:rhoSigmaDef} is continuous, also \eqref{eq:RhoSigma}
is continuous as linear map from $\Sym^\bullet_R(V_\cov)$ to
$\Sym^\bullet_R(V_\Sigma)$ by
Lemma~\ref{lemma:ContinuityA}. Moreover, we have
\begin{equation}
    \label{eq:RhoSigmaEvaluation}
    \Phi(u) = \varrho_\Sigma(\Phi)(u_0, \dot{u}_0)
\end{equation}
for all $\Phi \in \Sym^\bullet(V_\cov)$ and all solutions $u \in
\Secinftysc(E)$ of the wave equation $Du = 0$ with initial conditions
$(u_0, \dot{u}_0)$. This is clear since evaluation of an element in
the symmetric algebra on a point is a homomorphism and
$\varrho_\Sigma$ is a homomorphism as well. Since we only have to
check the equality of two homomorphism on generators,
\eqref{eq:PairingRhoSigma} is all we need to conclude
\eqref{eq:RhoSigmaEvaluation}.
\begin{lemma}
    \label{lemma:RhoSigmaContinuousBetweenWeylR}%
    The algebra homomorphism $\varrho_\Sigma$ is a Poisson morphism as
    well as a continuous $^*$-algebra homomorphism
    \begin{equation}
        \label{eq:rhoSigmaWeylR}
        \varrho_\Sigma\colon
        \WeylR(V_\cov \tensor \mathbb{C}, \starcov)
        \longrightarrow
        \WeylR(V_\Sigma \tensor \mathbb{C}, \starSigma).
    \end{equation}
\end{lemma}
\begin{proof}
    Thanks to Proposition~\ref{proposition:Functorial} and
    Proposition~\ref{proposition:FunctorialityWeylR} we only have to
    show that \eqref{eq:rhoSigmaDef} is a Poisson map. Thus let
    $\varphi, \psi \in \Secinfty_0(E^*)$ be given and let $(\varphi_0,
    \dot{\varphi}_0) = \varrho_\Sigma(\varphi)$ as well as $(\psi_0,
    \dot{\psi}_0) = \varrho_\Sigma(\psi)$ be their images in
    $V_\Sigma$ under $\varrho_\Sigma$. Consider now $u = (F_M
    \psi)^\sharp = G_M(\psi^\sharp) \in \Secinftysc(E)$ which is a
    solution of the wave equation $D u = 0$ with initial conditions
    $u_0 = \iota^\# u = \iota^\#\left(F_M(\psi)\right)^\sharp$ and
    $\dot{u}_0 = \iota^\#\left(\nabla^E_{\mathfrak{n}} u\right) =
    \iota^\#\left(\nabla^{E^*}_{\mathfrak{n}}
        F_M(\psi)\right)^\sharp$. Here we use that $\nabla^E$ is
    metric and hence compatible with the musical isomorphism $\sharp$
    induced by $h$. Now we have
    \begin{align*}
        &\Lambda_\Sigma\left(
            (\varphi_0, \dot{\varphi}_0),
            (\psi_0, \dot{\psi}_0)
        \right) \\
        &\quad=
        - \int_\Sigma \Big(
            \left(
                \iota^\#
                \nabla^{E^*}_{\mathfrak{n}} F_M(\varphi)
            \right)
            \cdot
            \underbrace{
              \left(
                  \iota^\# F_M(\psi)
              \right)^\sharp
            }_{u_0}
            -
            \left(
                \iota^\# F_M(\varphi)
            \right)
            \cdot
            \underbrace{
              \left(
                  \iota^\#
                  \nabla^{E^*}_{\mathfrak{n}} F_M(\psi)
              \right)^\sharp
            }_{\dot{u}_0}
        \Big)
        \;
        \mu_\Sigma \\
        &\quad\stackrel{\mathclap{\eqref{eq:Solution}}}{=}
        \quad
        - \int_M \varphi \cdot u \; \mu_g \\
        &\quad=
        - \int_M h^{-1}(\varphi, F_M(\psi)) \; \mu_g \\
        &\quad=
        \Lambda_\cov(\varphi, \psi).
    \end{align*}
    Clearly, $\varrho_\Sigma$ is real and hence commutes with the
    complex conjugation.
\end{proof}
\begin{lemma}
    \label{lemma:kerRhoSigma}%
    The kernel of $\varrho_\Sigma$ coincides with the Poisson ideal
    generated by $\ker F_M$ which consists of those elements in
    $\Sym^\bullet(V_\cov)$ which vanish on all solutions.
\end{lemma}
\begin{proof}
    Clearly, the kernel of $\varrho_\Sigma\at{V_\cov}$ is given by
    $\ker F_M$ by Lemma~\ref{lemma:Casimirs}. This implies $\ker
    \varrho_\Sigma = \langle\ker F_M \rangle$ in general. The second
    statement then follows from \eqref{eq:RhoSigmaEvaluation} at once.
\end{proof}
This statement has a very natural physical interpretation: if we view
$\Phi, \Psi \in \Sym^\bullet(V_\cov)$ as observables of the field
theory, their expectation values for a given field configuration are
just the evaluations $\Phi(u), \Psi(u) \in \mathbb{R}$, where $u \in
\Secinftysc(E)$. But since physically only those $u \in
\Secinftysc(E)$ occur which also satisfy the wave equation $D u = 0$,
we have to identify the observables $\Phi$ and $\Psi$ as soon as they
coincide on the solutions. This is the case iff $\Phi - \Psi \in
\langle\ker F_M \rangle$.
\begin{lemma}
    \label{lemma:RhoSigmaSurjective}%
    The homomorphism $\varrho_\Sigma$ is surjective.
\end{lemma}
\begin{proof}
    Let $\varphi_0, \dot{\varphi}_0 \in \Secinfty_0(E^*_\Sigma)$ be
    given. Then there is a (unique) solution $\Phi \in
    \Secinftysc(E^*)$ of the wave equation $D^\Trans \Phi = 0$ with
    the initial conditions
    \[
    \iota^\# \Phi = - \dot{\varphi}_0
    \quad
    \textrm{and}
    \quad
    \iota^\# \nabla^{E^*}_{\mathfrak{n}} \Phi = \varphi_0,
    \]
    since $D^\Trans$ is normally hyperbolic as well. By
    \eqref{eq:ExactSequence} for $D^\Trans$ and $F_M$ we know that
    $\Phi = F_M \varphi$ for some $\varphi \in \Secinfty_0(E^*)$. Then
    $\varrho_\Sigma(\varphi) = (\varphi_0, \dot{\varphi}_0)$ follows.
\end{proof}

We can collect now the above results in the following statement
leading to the comparison between the covariant and the canonical
Poisson bracket and their Weyl algebras, see
\cite[Thm.~4.4.22]{waldmann:2012a:pre} for the classical part:
\begin{theorem}
    \label{theorem:CovariantVSCanonical}%
    Fix $R \ge \frac{1}{2}$.  Let $(M, g)$ be a globally hyperbolic
    spacetime and let $E \longrightarrow M$ be a real vector bundle
    with fiber metric $h$ and metric connection $\nabla^E$. Moreover,
    let $D = \dAlembert^\nabla + B$ with $B = B^* \in
    \Secinfty(\End(E))$ be a symmetric normally hyperbolic
    differential operator on $E$ and denote by $F_M$ the propagator of
    its adjoint $D^*$. Finally, let $\iota\colon \Sigma
    \longrightarrow M$ be a smooth spacelike Cauchy surface.
    \begin{theoremlist}
    \item \label{item:SubspacesCoincide} The following subspaces of
        $\Sym^\bullet(V_\cov)$ coincide:
        \begin{compactitem}
        \item The vanishing ideal of the solutions of the wave
            function $Du = 0$.
        \item The Poisson ideal generated by the Casimir elements
            $\varphi \in V_\cov$.
        \item The ideal $\langle\ker F_M\rangle$.
        \item The kernel of the Poisson homomorphism
            $\varrho_\Sigma\colon \Sym^\bullet(V_\cov) \longrightarrow
            \Sym^\bullet(V_\Sigma)$.
        \end{compactitem}
    \item \label{item:PoissonCovQuotient} The locally convex quotient
        Poisson algebra $\Sym^\bullet_R(V_\cov) \big/ \langle\ker
        F_M\rangle$ is canonically isomorphic to the Poisson algebra
        $\Sym^\bullet_R(V_\cov \big/\ker F_M)$, with the Poisson
        bracket coming from \eqref{eq:LambdaCov} defined on classes.
    \item \label{item:CovToCanPoisson} The Poisson homomorphism
        $\varrho_\Sigma$ induces a continuous Poisson isomorphism
        \begin{equation}
            \label{eq:varrhoSigmaPoissonIso}
            \varrho_\Sigma\colon
            \Sym^\bullet_R(V_\cov) \big/ \langle\ker F_M\rangle
            \longrightarrow
            \Sym^\bullet_R(V_\Sigma).
        \end{equation}
    \item \label{item:WeylCovQuotient} The locally convex quotient
        $^*$-algebra $\WeylR(V_\cov \tensor \mathbb{C}, \starcov)
        \big/ (\langle \ker F_M\rangle \tensor \mathbb{C})$ is
        $^*$-isomorphic to the Weyl algebra $\WeylR((V_\cov \big/ \ker
        F_M) \tensor \mathbb{C}, \starcov)$ with $\starcov$ coming
        from the Poisson bracket \eqref{eq:LambdaCov} of $V_\cov \big/
        \ker F_M$.
    \item \label{item:CovToCanWeyl} The $^*$-homomorphism
        $\varrho_\Sigma$ induces a continuous $^*$-isomorphism
        \begin{equation}
            \label{eq:varrhoSigmaStarIso}
            \varrho_\Sigma\colon
            \WeylR(V_\cov \tensor \mathbb{C}, \starcov)
            \big/ (\langle\ker F_M\rangle \tensor \mathbb{C})
            \longrightarrow
            \WeylR(V_\Sigma \tensor \mathbb{C}, \starSigma).
        \end{equation}
    \end{theoremlist}
\end{theorem}
\begin{proof}
    The only things left to prove are the continuity statements with
    respect to the quotient topologies. But these follow from the
    general situation discussed in
    Lemma~\ref{lemma:SeminormQuotientIota} and
    Corollary~\ref{corollary:QuotientIotaIsomorphism}.
\end{proof}
\begin{remark}
    \label{remark:Evaluation}%
    It follows that the evaluation of $\Phi \in \Sym^\bullet(V_\cov)$
    on a solution $u \in \Secinftysc(E)$ of the wave equation depends
    only on the equivalence class of $\Phi$ modulo $\langle \ker
    F_M\rangle$. Hence this evaluation yields a well-defined and still
    continuous linear functional on $\Sym^\bullet_R(V_\cov) \big/
    \langle\ker F_M\rangle$ and on $\WeylR(V_\cov \tensor \mathbb{C},
    \starcov) \big/ (\langle \ker F_M\rangle \tensor
    \mathbb{C})$. This way, the elements in these quotients can be
    thought of as observables of the field theory described by the
    wave equation $Du = 0$.
\end{remark}

%
% Locality and Time-Slice Axiom
%

\subsection{Locality and Time-Slice Axiom}
\label{subsec:LocalityTimeSlice}

In this last section we collect some further properties of the
covariant Poisson bracket and its Weyl algebra quantization as
required by the Haag-Kastler approach to (quantum) field theory
\cite{haag:1993a}: locality and the time-slice axiom.

Let $U \subseteq M$ be a non-empty open subset. Then we denote by
$\tilde{\mathcal{A}}_\cl(U) \subseteq \Sym^\bullet(V_\cov)$ the unital
Poisson subalgebra generated by those $\varphi \in \Secinfty_0(E^*)$
with $\supp \varphi \subseteq U$. Analogously, we define
$\tilde{\mathcal{A}}(U) \subseteq \WeylR(V_\cov \tensor \mathbb{C},
\starcov)$ to be the unital $^*$-subalgebra generated by those
$\varphi \in \Secinfty_0(E^*)$ with $\supp \varphi \subseteq U$. For
$U = \emptyset$ we set $\tilde{\mathcal{A}}_\cl(\emptyset) =
\mathbb{C}\Unit = \tilde{\mathcal{A}}(\emptyset)$.  Finally, we set
$\mathcal{A}_\cl(U)$ and $\mathcal{A}(U)$ for the images of
$\tilde{\mathcal{A}}_\cl(U)$ and $\tilde{\mathcal{A}}(U)$ in the
quotients $\Sym^\bullet(V_\cov) \big/ \langle\ker F_M\rangle$ and
$\WeylR(V_\cov \tensor \mathbb{C}, \starcov) \big/ (\langle\ker
F_M\rangle \tensor \mathbb{C})$, respectively. Clearly,
$\mathcal{A}_\cl(M)$ and $\mathcal{A}(M)$ yield again everything.
Then the following properties are obvious from the causal properties
of $F_M$:
\begin{proposition}[Local net of observables]
    \label{proposition:LocalNet}%
    Let $U, U' \subseteq M$ be open subsets of $M$.
    \begin{propositionlist}
    \item \label{item:Locality} We have the locality property
        \begin{equation}
            \label{eq:Locality}
            \Bracketcov{\mathcal{A}_\cl(U), \mathcal{A}_\cl(U')}
            = 0
            \quad
            \textrm{and}
            \quad
            \left[
                \mathcal{A}(U), \mathcal{A}(U')
            \right]_{\starcov}
            = 0
        \end{equation}
        whenever $U$ and $U'$ are spacelike.
    \item \label{item:NetProperty} For $U \subseteq U'$ we have
        \begin{equation}
            \label{eq:NetProperty}
            \mathcal{A}_\cl(U) \subseteq \mathcal{A}_\cl(U')
            \quad
            \textrm{and}
            \quad
            \mathcal{A}(U) \subseteq \mathcal{A}(U')
        \end{equation}
    \item \label{item:GetEverything} We have
        \begin{equation}
            \label{eq:UnionIsEverything}
            \bigcup_{U \subseteq M \textrm{open}}
            \mathcal{A}_\cl(U)
            =
            \mathcal{A}_\cl(M)
            \quad
            \textrm{and}
            \quad
            \bigcup_{U \subseteq M \textrm{open}}
            \mathcal{A}(U)
            =
            \mathcal{A}(M).
        \end{equation}
    \end{propositionlist}
\end{proposition}
\begin{proof}
    For $\phi, \psi \in \Secinfty_0(E^*)$ we clearly have
    $\Bracketcov{\phi, \psi} = 0 = \left[\phi, \psi\right]_{\starcov}$
    whenever $U$ and $U'$ are spacelike and $\supp \phi \subseteq U$
    and $\supp \psi \subseteq U'$. This follows immediately from
    \eqref{eq:LambdaCov} and the fact that $\supp F_M (\phi) \subseteq
    J_M(\supp \phi) \subseteq J_M(U)$ which does not intersect
    $U'$. Then the Leibniz rule shows \eqref{eq:Locality} in
    general. The remaining statements are clear.
\end{proof}

The time-slice axiom requires that a small neighbourhood of a Cauchy
surface contains already all the information about the observables. In
our framework, this can be formulated as follows:
\begin{proposition}[Time-slice axiom]
    \label{proposition:TimeSliceAxiom}%
    Let $\iota\colon \Sigma \longrightarrow M \cong \mathbb{R} \times
    \Sigma$ be a smooth Cauchy surface and let $\epsilon > 0$. Then
    \begin{equation}
        \label{eq:ASigmaEpsAM}
        \mathcal{A}_\cl(\Sigma_\epsilon) = \mathcal{A}_\cl(M)
        \quad
        \textrm{and}
        \quad
        \mathcal{A}(\Sigma_\epsilon) = \mathcal{A}(M),
    \end{equation}
    where $\Sigma_\epsilon = (-\epsilon, \epsilon) \times \Sigma$ is
    the $\epsilon$-time slice around $\Sigma$.
\end{proposition}
\begin{proof}
    First we note that $\Sigma_\epsilon$ is a globally hyperbolic
    spacetime by its own and the inclusion $\Sigma_\epsilon \subseteq
    M$ is causally compatible, i.e. we have
    $J^\pm_{\Sigma_\epsilon}(p) = J_M^\pm(p) \cap \Sigma_\epsilon$ for
    all $p \in \Sigma_\epsilon$. Restricting $D$ and $D^\Trans$ to
    $\Sigma_\epsilon$ gives globally hyperbolic differential operators
    with Green operators $G^\pm_{\Sigma_\epsilon}$ and
    $F^\pm_{\Sigma_\epsilon}$, respectively. By the uniqueness of the
    Green operators we have for $\varphi \in
    \Secinfty_0(E^*\at{\Sigma_\epsilon})$ the equality
    \[
    F^\pm_{\Sigma_\epsilon}(\varphi)
    =
    F^\pm_M(\varphi) \at{\Sigma_\epsilon}.
    \]
    Thus the covariant Poisson bracket for
    $\Sym^\bullet(\Secinfty_0(E^*\at{\Sigma_\epsilon})$ is the
    restriction of the one from $\Sym^\bullet(\Secinfty_0(E^*))$ to
    the subalgebra
    $\Sym^\bullet(\Secinfty_0(E^*\at{\Sigma_\epsilon})$. Next, let
    $\varphi \in \Secinfty_0(E^*\at{\Sigma_\epsilon})$ be given. Then
    in the condition $\varphi(u) = 0$ only $u\at{\Sigma_\epsilon}$
    enters. This shows that also the kernels of $F_M$ and
    $F_{\Sigma_\epsilon}$ correspond, i.e. we have
    \[
    \ker F_{\Sigma_\epsilon}
    =
    \ker\left(F_M\at{\Secinfty_0(E^*|_{\Sigma_\epsilon})}\right).
    \]
    Putting this together we conclude that the two isomorphisms
    $\varrho_\Sigma$ with respect to the Cauchy surface $\Sigma$, once
    sitting inside $M$ and the other time sitting inside $(-\epsilon,
    \epsilon) \times \Sigma$, give the desired isomorphism needed for
    \eqref{eq:ASigmaEpsAM}.
\end{proof}

For more information on the locality properties and the time-slice
axiom in the context of quantum field theories on globally hyperbolic
spacetimes we refer to the recent works \cite{hollands.wald:2010a,
  brunetti.fredenhagen.ribeiro:2012a:pre,
  brunetti.fredenhagen.verch:2003a} as well as
\cite[Chap.~4]{baer.ginoux.pfaeffle:2007a}, where the $C^*$-algebraic
version and the functorial aspects of the above construction are
discussed in detail.

%
% The bibliographies: dq* are the standard one, but also include all
% the others.
% Make it smaller than other text
%

\begin{footnotesize}
    \renewcommand{\arraystretch}{0.5}
    % \bibliographystyle{ewde}
    % \bibliography{dqarticle,dqbook,dqprocentry,dqproceeding,dqthesis,misc,script,preprints}

\begin{thebibliography}{10}

\bibitem {baer.ginoux.pfaeffle:2007a}
\textsc{B{\"a}r, C., Ginoux, N., Pf{\"a}ffle, F.: }\newblock \emph{Wave
  equations on {L}orentzian manifolds and quantization}.
\newblock \emph{ESI Lectures in Mathematics and Physics}.
\newblock European Mathematical Society (EMS), Z\"urich, 2007.

\bibitem {bayen.et.al:1978a}
\textsc{Bayen, F., Flato, M., Fr{{\o}}nsdal, C., Lichnerowicz, A., Sternheimer,
  D.: }\newblock \emph{Deformation Theory and Quantization}.
\newblock Ann. Phys.  \textbf{111} (1978), 61--151.

\bibitem {beiser.roemer.waldmann:2007a}
\textsc{Beiser, S., R{\"o}mer, H., Waldmann, S.: }\newblock \emph{Convergence
  of the {W}ick Star Product}.
\newblock Commun. Math. Phys.  \textbf{272} (2007), 25--52.

\bibitem {beiser.waldmann:2011a:pre}
\textsc{Beiser, S., Waldmann, S.: }\newblock \emph{Fr{é}chet algebraic
  deformation quantization of the Poincar{é} disk}.
\newblock Preprint  \textbf{arXiv:1108.2004} (2011), 57 pages.
\newblock To appear in Crelle's J. reine angew. Math.

\bibitem {binz.honegger.rieckers:2004b}
\textsc{Binz, E., Honegger, R., Rieckers, A.: }\newblock \emph{Field-theoretic
  {W}eyl Quantization as a Strict and Continuous Deformation Quantization}.
\newblock Ann. H. Poincar{\'e}  \textbf{5} (2004), 327--346.

\bibitem {borchers.yngvason:1976a}
\textsc{Borchers, H.~J., Yngvason, J.: }\newblock \emph{Necessary and
  sufficient conditions for integral representations of {W}ightman functionals
  at {S}chwinger points}.
\newblock Commun. Math. Phys.  \textbf{47}.3 (1976), 197--213.

\bibitem {brunetti.fredenhagen.ribeiro:2012a:pre}
\textsc{Brunetti, R., Fredenhagen, K., Ribeiro, P.~L.: }\newblock
  \emph{Algebraic Structure of Classical Field Theory I: Kinematics and
  Linearized Dynamics for Real Scalar Fields}.
\newblock Preprint  \textbf{arXiv:1209.2148} (2012), 53.

\bibitem {brunetti.fredenhagen.verch:2003a}
\textsc{Brunetti, R., Fredenhagen, K., Verch, R.: }\newblock \emph{The
  generally covariant locality principle---a new paradigm for local quantum
  field theory}.
\newblock Commun. Math. Phys.  \textbf{237} (2003), 31--68.

\bibitem {buchholz.grundling:2008a}
\textsc{Buchholz, D., Grundling, H.: }\newblock \emph{The resolvent algebra: a
  new approach to canonical quantum systems}.
\newblock J. Funct. Anal.  \textbf{254}.11 (2008), 2725--2779.

\bibitem {bursztyn.waldmann:2000a}
\textsc{Bursztyn, H., Waldmann, S.: }\newblock \emph{On Positive Deformations
  of {$^*$}-Algebras}.
\newblock In: \textsc{Dito, G., Sternheimer, D. (eds.): }\newblock
  \emph{Conf{\'e}rence Mosh{\'e} Flato 1999. Quantization, Deformations, and
  Symmetries}. \cite{dito.sternheimer:2000b},   69--80.

\bibitem {cuntz:1997a}
\textsc{Cuntz, J.: }\newblock \emph{Bivariante {$K$}-{T}heorie f\"ur
  lokalkonvexe {A}lgebren und der {C}hern-{C}onnes-{C}harakter}.
\newblock Doc. Math.  \textbf{2} (1997), 139--182.

\bibitem {cuntz:2005a}
\textsc{Cuntz, J.: }\newblock \emph{Bivariant {K}-Theory and the {W}eyl
  Algebra}.
\newblock K-Theory  \textbf{35} (2005), 93--137.

\bibitem {dito:2005a}
\textsc{Dito, G.: }\newblock \emph{Deformation quantization on a {H}ilbert
  space}.
\newblock In: \textsc{Maeda, Y., Tose, N., Miyazaki, N., Watamura, S.,
  Sternheimer, D. (eds.): }\newblock \emph{Noncommutative geometry and
  physics},   139--157. World Scientific, Singapore, 2005.
\newblock Proceedings of the CEO International Workshop.

\bibitem {dito.sternheimer:2000b}
\textsc{Dito, G., Sternheimer, D. (eds.): }\newblock \emph{Conf{\'e}rence
  Mosh{\'e} Flato 1999. Quantization, Deformations, and Symmetries}.
\newblock \emph{Mathematical Physics Studies} no. \textbf{22}.
\newblock Kluwer Academic Publishers, Dordrecht, Boston, London, 2000.

\bibitem {duetsch.fredenhagen:2001a}
\textsc{D{\"u}tsch, M., Fredenhagen, K.: }\newblock \emph{Algebraic Quantum
  Field Theory, Perturbation Theory, and the Loop Expansion}.
\newblock Commun. Math. Phys.  \textbf{219} (2001), 5--30.

\bibitem {duetsch.fredenhagen:2001b}
\textsc{D{\"u}tsch, M., Fredenhagen, K.: }\newblock \emph{Perturbative
  Algebraic Field Theory, and Deformation Quantization}.
\newblock Field Inst. Commun.  \textbf{30} (2001), 151--160.

\bibitem {fedosov:1996a}
\textsc{Fedosov, B.~V.: }\newblock \emph{Deformation Quantization and Index
  Theory}.
\newblock Akademie Verlag, Berlin, 1996.

\bibitem {gerstenhaber:1968a}
\textsc{Gerstenhaber, M.: }\newblock \emph{On the Deformation of Rings and
  Algebras III}.
\newblock Ann. Math.  \textbf{88} (1968), 1--34.

\bibitem {haag:1993a}
\textsc{Haag, R.: }\newblock \emph{Local Quantum Physics}.
\newblock Springer-Verlag, Berlin, Heidelberg, New York, 2. edition, 1993.

\bibitem {hollands.wald:2010a}
\textsc{Hollands, S., Wald, R.~M.: }\newblock \emph{Axiomatic quantum field
  theory in curved spacetime}.
\newblock Commun. Math. Phys.  \textbf{293}.1 (2010), 85--125.

\bibitem {jarchow:1981a}
\textsc{Jarchow, H.: }\newblock \emph{Locally Convex Spaces}.
\newblock B. G. Teubner, Stuttdart, 1981.

\bibitem {meyer:2004a}
\textsc{Meyer, R.: }\newblock \emph{Smooth group representations on
  bornological vector spaces}.
\newblock Bull. Sci. Math.  \textbf{128}.2 (2004), 127--166.

\bibitem {minguzzi.sanchez:2008a}
\textsc{Minguzzi, E., S{\'a}nchez, M.: }\newblock \emph{The causal hierarchy of
  spacetimes}.
\newblock In: \textsc{Alekseevsky, D.~V., Baum, H. (eds.): }\newblock
  \emph{Recent developments in pseudo-{R}iemannian geometry}, \emph{ESI
  Lectures in Mathematics and Physics},   299--358. European Mathematical
  Society (EMS), Z\"urich, 2008.

\bibitem {omori.maeda.miyazaki.yoshioka:2000a}
\textsc{Omori, H., Maeda, Y., Miyazaki, N., Yoshioka, A.: }\newblock
  \emph{Deformation quantization of {F}r{\'e}chet-{P}oisson algebras:
  convergence of the {M}oyal product}.
\newblock In: \textsc{Dito, G., Sternheimer, D. (eds.): }\newblock
  \emph{Conf{\'e}rence Mosh{\'e} Flato 1999. Quantization, Deformations, and
  Symmetries}. \cite{dito.sternheimer:2000b},   233--245.

\bibitem {omori.maeda.miyazaki.yoshioka:2002a}
\textsc{Omori, H., Maeda, Y., Miyazaki, N., Yoshioka, A.: }\newblock \emph{Star
  exponential functions for quadratic forms and polar elements}.
\newblock Contemp. Math.  \textbf{315} (2002), 25--38.

\bibitem {omori.maeda.miyazaki.yoshioka:2007a}
\textsc{Omori, H., Maeda, Y., Miyazaki, N., Yoshioka, A.: }\newblock
  \emph{Orderings and non-formal deformation quantization}.
\newblock Lett. Math. Phys.  \textbf{82} (2007), 153--175.

\bibitem {peierls:1952a}
\textsc{Peierls, R.~E.: }\newblock \emph{The commutation laws of relativistic
  field theory}.
\newblock Proc. Royal Soc. A  \textbf{214} (1952), 143--157.

\bibitem {pflaum.schottenloher:1998a}
\textsc{Pflaum, M.~J., Schottenloher, M.: }\newblock \emph{Holomorphic
  deformation of Hopf algebras and applications to quantum groups}.
\newblock J. Geom. Phys.  \textbf{28} (1998), 31--44.

\bibitem {voigt:2008a}
\textsc{Voigt, C.: }\newblock \emph{Bornological quantum groups}.
\newblock Pacific J. Math.  \textbf{235}.1 (2008), 93--135.

\bibitem {waldmann:2007a}
\textsc{Waldmann, S.: }\newblock \emph{Poisson-{G}eometrie und
  {D}eformationsquantisierung. {E}ine {E}inf{\"u}hrung}.
\newblock Springer-Verlag, Heidelberg, Berlin, New York, 2007.

\bibitem {waldmann:2012a:pre}
\textsc{Waldmann, S.: }\newblock \emph{Geometric Wave Equations}.
\newblock Preprint  \textbf{arXiv:1208.4706} (2012), 279 + vi pages.
\newblock Lecture notes for the lecture `Wellengleichgungen auf Raumzeiten'
  held in Freiburg in 2008/2009 and 2009.

\end{thebibliography}

\end{footnotesize}

%
% end of weyl
%

\end{document}